\documentclass[oneside]{amsart}
\pdfoutput=1

\usepackage[a4paper]{geometry}
\usepackage[T1]{fontenc}
\usepackage[utf8]{inputenc}
\usepackage{lmodern}
\usepackage{amssymb,amsmath}

\usepackage{xcolor}

\definecolor{col1}{RGB}{100,143,255}
\definecolor{col2}{RGB}{120, 94, 240}
\definecolor{col3}{RGB}{254,97,0}
\definecolor{col4}{RGB}{220, 38, 127}
\definecolor{col5}{RGB}{255, 176, 0}

\usepackage{tikz}
\usetikzlibrary{shapes.misc}
\tikzset{cross/.style={cross out, draw,
         minimum size=2*(#1-\pgflinewidth),
         inner sep=0pt, outer sep=0pt}}
\newcommand\colonevertex[1]{\fill[col1] #1 circle (2.3pt)}

\usepackage{hyperref}

\newcommand{\Z}{{\mathbb{Z}}}
\newcommand{\Q}{{\mathbb{Q}}}
\newcommand{\R}{{\mathbb{R}}}

\newcommand{\malpha}{\gamma}

\DeclareMathOperator{\SG}{{\mathfrak S}}  %
\newcommand{\GL}{{\operatorname{GL}}}  %
\newcommand{\Aut}{{\operatorname{Aut}}}  %
\newcommand{\Out}{{\operatorname{Out}}}  %
\newcommand{\sign}{{\operatorname{sign}}}

\newcommand{\stab}{{\operatorname{Stab}}}

\newcommand{\Set}{{\operatorname{Set}}}
\newcommand{\Cyc}{{\operatorname{Cyc}}}
\newcommand{\ACyc}{{\mathcal {A}\Cyc}}
\newcommand{\Etwo}{\mathcal E}%
\newcommand{\EtwoSet}{E_{\Set}}%

\newcommand{\Outn}{\Out(F_n)}
\newcommand{\Autn}{\Aut(F_n)}

\newcommand{\parts}{\vdash}

\newcommand{\bb}[1]{{\overline{#1}}}

\DeclareMathOperator{\id}{id}

\DeclareMathOperator{\ee}{\widehat{e}}
\newcommand{\eO}{e^{\mathrm{odd}}}

\DeclareMathOperator{\eee}{e}

\newcommand{\TT}{\widehat{T}}
\DeclareMathOperator{\Ch}{\widehat{\chi}}
\newcommand{\dd}{\mathrm{d}}
\newcommand{\ld}{\delta} %
\newcommand{\md}{\delta^\prime}%
\newcommand{\wt}[1]{#1'}
\newcommand{\bigO}{\mathcal{O}}

\newcommand{\grph}{G}
\newcommand{\grp}{\Gamma}
\newcommand{\MG}{\mathcal{MG}}

\newtheorem{proposition}{Proposition}[section]

\newtheorem{theorem}[proposition]{Theorem}
\newtheorem*{theorem*}{Theorem}
\newtheorem{lemma}[proposition]{Lemma}

 \newtheorem*{notation}{Notation}
\newtheorem{corollary}[proposition]{Corollary}

\theoremstyle{remark}
\newtheorem{remark}[proposition]{Remark}

\hypersetup{
  pdfauthor={Michael Borinsky, Karen Vogtmann},
  pdftitle={The Euler characteristic of the moduli space of graphs},
  pdfsubject={},
  pdfkeywords={},
  linkcolor  = col2,
  citecolor  = col4,
  urlcolor   = col5,
  colorlinks = true,
}

\title{The Euler characteristic of the moduli space of graphs}
\author{Michael Borinsky \and Karen Vogtmann}

\address{
Michael Borinsky\\
Institute for Theoretical Studies\\
ETH Z\"urich\\
8092 Zürich, Switzerland
}
\address{
Karen Vogtmann\\
Mathematics Institute\\
University of Warwick\\
Coventry CV4 7AL, United Kingdom
}

\begin{document}

\begin{abstract}
The moduli space of rank $n$ graphs, the outer automorphism group of the free group of rank $n$ and Kontsevich's Lie graph complex have the same rational cohomology.
We show that the associated Euler characteristic grows like $-e^{-1/4}\,(n/e)^n/(n\log n)^2$  as $n$ goes to infinity, and thereby
prove that the total dimension of this cohomology grows rapidly with $n$.
\end{abstract}

\maketitle

\section{Introduction}

The moduli space $\MG_n$ of finite metric graphs with fundamental group $F_n$ was introduced in \cite{CV} as a tool for studying the group $\Out(F_n)$ of outer automorphisms of a free group.  By the main result in that paper $\MG_n$ is the quotient of a contractible space, called {\em Outer space}, on which $\Out(F_n)$ acts with finite stabilizers.  Thus the homology of $\Out(F_n)$ with trivial rational coefficients is equal to the homology of $\MG_n$.

 Kontsevich showed in \cite{Ko93,Ko2} that the homology of $\MG_n$ can also be identified with the cohomology of his Lie graph complex, which can in turn be identified with the primitive part of the cohomology of the Lie algebra of symplectic derivations of a free Lie algebra (see \cite{CoVo} for a detailed exposition of Kontsevich's results).    In \cite{BeMa} Berglund and Madsen found this Lie algebra in a very different context, and proved that its cohomology is a sub-algebra of the cohomology of the block diffeomorphism group of even-dimensional products of spheres.
In more recent years, algebraic geometers have studied $\MG_n$ as a \emph{tropical analog} of  the classical moduli space $\mathcal M_n$ of smooth complex curves of genus $n$.  The simplicial completion of Outer space descends to a natural compactification of $\MG_n$,  which the tropical geometers have dubbed the \emph{moduli space of tropical curves,}  by analogy with the Deligne-Mumford compactification  of $\mathcal M_n$ (see, e.g., \cite{CGP}).
In yet another context, $\MG_n$ may be considered a  natural parameter space for the $n$-loop contribution to  certain Feynman amplitudes.  This direction has been explored, for example, by Bloch,  Berghoff and Kreimer~\cite{Bloch:2015efx, Berghoff:2017dyq}.

In this paper we prove a formula for the Euler characteristic of $\MG_n$, and then determine its asymptotic growth rate.  The asymptotic result depends on our results in \cite{BV}, where we determined the asymptotic growth rate of the {\em rational} or {\em virtual} Euler characteristic  $\chi(\Out(F_n))$.   This is a rational number closely related to the alternating sum of the Betti numbers, but which has better group-theoretic properties, making it easier to compute.
 The rational Euler characteristic of $\Out(F_n)$ coincides with the number Kontsevich referred to as the {\em orbifold Euler characteristic} of his Lie graph complex.  The actual alternating sum of the Betti numbers is denoted $e(\Outn)$ in this paper to distinguish it from $\chi(\Outn)$, and is called the {\em integral Euler characteristic} to conform with other terminology in the literature.

If we are primarily interested in the cohomology of the space $\MG_n$ or (equivalently) the group  $\Out(F_n)$, then the number $e(\Out(F_n)) $ is clearly more relevant.  Brown \cite{Br2} showed that the rational and integral Euler characteristics of a group $\grp$ are closely related, namely $e(\grp)$ can be calculated from the rational Euler characteristics of centralizers  of finite-order elements by the formula
$$e(\grp)=\sum_{[ \alpha]} \chi(C(\alpha)).$$
Here the sum is over representatives $\alpha$ for the conjugacy classes of finite-order elements (including the identity), and $C(\alpha)$ is the centralizer of $\alpha$.

The number $e(\Outn)$  was calculated for $n\leq 11$ by Morita, Sakasai and Suzuki  \cite{MSS}, using methods from symplectic representation theory.  In the present  paper we use Brown's formula, results on centralizers from \cite{KV}, an adaptation of Joyal's theory of species \cite{joyal1981theorie} and further development of the asymptotic methods of \cite{BV} to first give an effective formula for $e(\Outn)$ and then to determine its asymptotic growth rate.

The effective formula is developed in Sections~\ref{sec:firstformula}-\ref{sec:effective} and summarized in Theorem~\ref{thm:compute}. Based on it, we wrote a computer program to compute $e(\Outn)$ for $n\leq 15$. The results are listed in Appendix~\ref{sec:table}. Further optimizations of this program enabled the computation of the numbers $e(\Outn)$ for all $n \leq 100$~\cite{BVer}.

Section~\ref{sec:asymptotics} is devoted to proving the following asymptotic result.
\begin{theorem}
\label{thm:eOutFnAsy}
The integral Euler characteristic $e(\Outn)$ has the asymptotic behavior
\begin{gather*} e(\Out(F_{n})) \sim - {e^{-\frac14}} \left(\frac{n}{ e}\right)^{n} \frac{1}{(n\log n)^2} \text{ as } n\rightarrow \infty. \end{gather*}
\end{theorem}
Here the notation $a_n\sim b_n$ means that $\lim_{n \rightarrow \infty} a_n/b_n = 1$.
In particular this verifies the fact,  suggested by our results on the rational Euler characteristic in \cite{BV}, that there is a huge amount of cohomology in odd dimensions. Since the cohomology of $\Outn$ is a direct summand of the cohomology of $\Autn$, we reach the same conclusion for $\Autn$.   The cohomology  $H^k(\Outn;\Q)$ is known to vanish for both $k<4n/5$   and    $k>2n-3$ (see \cite{CKHV} and the references there), thus all of this cohomology must be concentrated in dimensions $4n/5\leq k\leq 2n-3$.    The only odd-dimensional class known to date occurs in  $H_{11}(\Out(F_7))$ \cite{Bartholdi}.  It is interesting to note that this is the largest possible dimension, the virtual cohomological dimension of $\Out(F_7)$. This is in contrast to the fact that  the groups $\GL_n(\Z)$ and mapping class groups of punctured surfaces, both of which are often considered analogs of  $\Outn$,  have no rational cohomology in their virtual cohomological dimension.

Comparing Theorem~\ref{thm:eOutFnAsy} with our results in~\cite{BV}   on %
$\chi(\Outn)$ gives
 \begin{corollary} The ratio $\lim_{n \rightarrow \infty} e(\Outn)/\chi(\Outn) = e^{-\frac{1}{4}} \approx 0.78.$
\end{corollary}
\begin{proof}
Use Theorem~\ref{thm:eOutFnAsy}, \cite[Thm.~A]{BV} and Stirling's formula~(see, e.g., \cite[Eq.~(3.9)]{artin1964gamma}).
\end{proof}
 This solves Problem 6.5 of the paper \cite{MSS} by Morita, Sakasai and Suzuki. More precise asymptotic statements and  comments about  the rate of convergence  are given in Section~\ref{sec:precise}.

In his original paper \cite{Ko93} Kontsevich introduced {\em commutative} and {\em associative} graph complexes in addition to the Lie graph complex. The methods of the present paper can be modified to compute the Euler characteristics for both of these other graph complexes  as well as to determine their asymptotic behavior  (see~\cite{MBeuler}); they can also be used to do the same for other moduli spaces of graphs, such as colored graphs or graphs with leaves.   As Kontsevich noted, the associative graph complex computes the homology of mapping class groups of  punctured surfaces. Both the rational and integral Euler characteristics of these groups for a once-punctured surface were originally computed by Harer and Zagier \cite{HaZa}.  They also computed the asymptotics and deduced the existence of lots of cohomology.   Getzler and Kapranov partially extended this to surfaces with more than one puncture (without determining the asymptotics)  as an application of their general theory of modular operads \cite{getzler1998modular}, and we remark that our method of finding the generating function is similar to theirs.

In~\cite{Ko93} Kontsevich also defined  odd versions of his graph complexes, and noted that in the Lie case the primitive part of the homology computes the cohomology of $\Outn$ with twisted coefficients $\widetilde\Q$, where the twisting is given by composing the natural map from $\Outn$ to $\GL_n(\Z)$ with the determinant map.  This odd version of Lie graph homology occurs, for example, in the study of  groups of homotopy equivalences of odd-dimensional products of spheres \cite{Stoll}.  In a final section, Section~\ref{sec:odd}, we explain how to modify our results to compute the Euler characteristic of this odd Lie graph complex, which we denote $\eO(\Outn)$. The results also extend to the analysis of the asymptotics, and we find

\begin{theorem}
\label{thm:easyodd}
The ratio $\lim_{n \rightarrow \infty} \eO(\Outn)/\chi(\Outn) = e^{\frac{1}{4}} \approx 1.28.$
\end{theorem}

\section*{Acknowledgments}
We are grateful to Jos Vermaseren for generous FORM programming help and to Thomas Willwacher for illuminating discussions.
MB was supported by Dr.\ Max Rössler, the Walter Haefner Foundation and the ETH Zürich Foundation.

\section{A first formula for \texorpdfstring{$e(\Outn)$}{eOutFn}}
\label{sec:firstformula}
\subsection{The rational Euler characteristic}
 As noted in the introduction, Brown's theorem says we can compute $e(\Outn)$ by adding up the rational Euler characteristics of centralizers of finite-order elements.  One way to compute the rational Euler characteristic of a group $\grp$ is to find a contractible cell complex $Y$ on which  $\grp$ acts properly and cocompactly;    the rational Euler characteristic   $\chi(\Gamma)$ is then given by the formula
$$
\chi(\Gamma)=\sum_{[\sigma ] } \frac{(-1)^{\dim(\sigma)}}{|\stab(\sigma)|}
,$$
where we sum over all orbits $[\sigma]$ of cells, $\sigma$ is a representative from the orbit and $\stab(\sigma)$ is its stabilizer under the group action. Fortunately, we have such complexes $Y$ for centralizers of finite-order elements of $\Outn$.

The entire group $\Outn$ centralizes the identity. Recall from \cite{CV} that $\Outn$ acts properly and cocompactly on the spine $K_n$ of Outer space.  $K_n$ is a contractible cube complex with one $k$-dimensional cube for each equivalence class $[\grph,\Phi,g]$ of triples $(\grph,\Phi, g)$, where
\begin{itemize}
\item $\grph$ is a connected admissible graph,  
\item $\Phi$ is a subforest of $\grph$ with $k$ edges and
\item $g\colon F_n\to \pi_1(\grph)$ is an isomorphism, called a {\em marking}.
\end{itemize}
Here by a \emph{graph} we mean a   CW-complex of dimension $0$ or $1$,   a \emph{forest} is a graph without cycles, and a {\em subforest} of $\grph$ is a subcomplex that is a forest and contains all of the vertices of $\grph$. We say a graph is of \emph{rank} $n$ if its fundamental group has rank $n$.  A graph is  {\em admissible}  if it has no isolated, univalent or bivalent vertices. Two triples  $(\grph,\Phi, g)$ and $(\grph',\Phi', g')$ are {\em equivalent}  if there is a graph isomorphism $h\colon\grph\to\grph'$ sending $\Phi$ to $\Phi'$ and inducing an isomorphism $h_*\colon \pi_1(\grph) \rightarrow \pi_1(G')$, such that $g^{-1} h_* g=\id.$ A  pair $(G,\Phi)$ consisting of an admissible graph $G$ and a subforest $\Phi$ will be called a  {\em forested graph}.

 The spine $K=K_n$ is contractible, and the action of $\Outn$ on $K$ simply changes the marking $g$. Thus there is one orbit for each isomorphism class $[\grph,\Phi]$ of connected forested graphs.   The stabilizer of a cube $[\grph,\Phi,g]$ is isomorphic to $\Aut(\grph,\Phi)$, the automorphisms of $\grph$ that preserve $\Phi.$ Thus,
 $$ \chi(\Outn)=\chi(C(\id))=\sum_{
\substack{
[\grph,\Phi]\\
G \, \mathrm{connected}\\
\mathrm{rank}(G) = n
}
}\frac{(-1)^{e(\Phi)}}{|\Aut(\grph,\Phi)|},$$
 where $e(\Phi)$ is the number of edges in $\Phi$.

\subsection{The equivariant spine}
In \cite{KV} an equivariant version of $K$ was introduced, which can be used to study the centralizer of any finite order element of $\Outn$ (in fact the centralizer of any finite-order subgroup).   We briefly summarize the construction.  A graph $\grph$ is said to {\em realize} a finite-order automorphism $\alpha$ if there is some marking $g\colon F_n\to \pi_1(\grph)$ and automorphism $f_\alpha$ of $\grph$ such that  $g^{-1}(f_\alpha)_*g=\alpha.$  Every finite-order element of $\Outn$ can be realized on some admissible graph $\grph,$ ~\cite{Cu, Kh}.
This translates to the statement that
the action of $\alpha$ on  $K$ has at least one fixed point.   The centralizer $C(\alpha)$  acts on the entire fixed-point set $K_\alpha$ of $\alpha$. It follows from \cite{KV}  that this fixed point set is contractible, cocompact and has the structure of a cube complex. Specifically, $K_\alpha$ has one cube for each equivalence class of triples $(\grph,\Phi, g)_\alpha$ where $(\grph,g)$ realizes $\alpha$ and $\Phi$ is a (possibly empty) forest in $\grph$ that is invariant under the action of $\alpha$.  Here  $(\grph,\Phi, g)_\alpha$ is equivalent to $(\grph',\Phi', g')_\alpha$ if there is an $\alpha$-invariant automorphism $h\colon\grph\to \grph'$ sending $\Phi$ to $\Phi'$ such that $g^{-1}h_*g=\id$, and again we write $[\grph,\Phi,g]_\alpha$ for the equivalence class. The dimension of the cube $[\grph, \Phi, g]_\alpha$ is the number $e_\alpha(\Phi)$ of edge-orbits in $\Phi$.

The stabilizer $\stab[\grph,\Phi,g]_\alpha$ is isomorphic to $\Aut_\alpha(\grph,\Phi),$ the automorphisms of $\grph$ that commute with the action of $\alpha$  and send $\Phi$ to itself, so
$$\chi(C(\alpha))=\sum_{
\substack{
[\grph,\Phi]_\alpha \\
G\, \mathrm{connected}\\
\mathrm{rank}(G) = n
}
} \frac{(-1)^{e_\alpha(\Phi)}}{|\Aut_\alpha(\grph,\Phi)|},$$
where $[\grph,\Phi]_\alpha$ runs over isomorphism classes of pairs that realize $\alpha$ on  a connected graph $G$ of rank $n$.
 Brown's theorem \cite{Br2} then gives
$$e(\Outn)=\sum_{[\alpha]} \sum_{
\substack{
[\grph,\Phi]_\alpha
\\
G\, \mathrm{connected}\\
\mathrm{rank}(G) = n
}
} \frac{(-1)^{e_\alpha(\Phi)}}{|\Aut_\alpha(\grph,\Phi)|},$$
where $[\alpha]$ runs over conjugacy classes of finite-order elements $\alpha$.

The group $\Aut(\grph,\Phi)$ acts on itself by conjugation, so the orbit-stabilizer theorem gives
$|\Aut(\grph,\Phi)|=|(\hbox{orbit of }\alpha)|\cdot |\stab_{\Aut(\grph,\Phi)}(\alpha)|$.  Since  $\stab_{\Aut(\grph,\Phi)}(\alpha)= \Aut_\alpha(\grph,\Phi)$ this gives:

\begin{theorem}~\label{thm:tau}
$$e(\Outn)=\sum_{\substack{[G,\Phi]\\ G\, \mathrm{connected}\\ \mathrm{rank}(G)=n}}\frac{1}{|\Aut(G,\Phi)|}\sum_{\alpha\in \Aut(G,\Phi)}(-1)^{e_\alpha(\Phi)},$$
where we sum over the set of isomorphism classes $[G,\Phi]$ of   connected forested graphs of rank $n$.%
\end{theorem}

\subsection{Disconnected graphs}
\label{sec:disc}

Recall that the Euler characteristic of a connected graph of rank $n$ is $\chi(G)=1-n$.
The graph's Euler characteristic is often better behaved than its rank.
For instance, the formula in Theorem~\ref{thm:tau} is easier to work with if we drop the requirement  that $G$ be connected and shift the index by one,  i.e.\ define
\begin{align} \label{eq:eendef} \ee_n &=\sum_{\substack{[G,\Phi]\\ \chi(G)=-n}}\frac{1}{|\Aut(G,\Phi)|}\sum_{\alpha\in \Aut(G,\Phi)}(-1)^{e_\alpha(\Phi)},  \end{align}
where we sum over all isomorphism classes $[G,\Phi]$ of (not necessarily connected) forested graphs    with $\chi(G)=-n$.
In this subsection we show how to recover $e(\Out(F_{n+1}))$ once the numbers $\ee_n$ are known.

We begin by deriving new formulas for $e(\Out(F_{n+1}))$ and $\ee_n$.
Define a forested graph $(G,\Phi)$ to be {\em even} if   $G$ has no automorphisms that induce an odd permutation of the edges of $\Phi$.
\begin{proposition}\label{prop:eOutFnFGC}
If we sum  over all isomorphism classes  $[G,\Phi]$ ---
\begin{enumerate}
\item ---  of even forested graphs with  $\chi(G)=-n$, then
$$\sum_{\substack{[G,\Phi] \, \mathrm{even} \\ \chi(G)=-n}}(-1)^{e(\Phi)}=\ee_n.$$
 \item --- of \emph{connected} even forested graphs with  $\chi(G)=-n$, then 
$$\sum_{\substack{[G,\Phi] \, \mathrm{even} \\ G \, \mathrm{connected} \\ \chi(G)=-n}}(-1)^{e(\Phi)}=e(\Out(F_{n+1})).$$
 \end{enumerate}
\end{proposition}

\begin{remark}
 Conant and Vogtmann showed in \cite{CoVo} that Kontsevich's Lie graph complex, as defined in~\cite{Ko93} is quasi-isomorphic to the complex spanned by even forested graphs.  Thus this proposition shows that $e(\Out(F_n))$ is equal to the Euler characteristic of the Lie graph complex.
 \end{remark}

Proposition~\ref{prop:eOutFnFGC} is an immediate consequence of the following lemma.
\begin{lemma}
\label{lmm:AutSum}
The sum
$\sum_{\alpha\in \Aut(G,\Phi)}(-1)^{e_\alpha(\Phi)}$
vanishes if $(G,\Phi)$ has an automorphism that induces an odd permutation on $\Phi$ and is equal to
$(-1)^{e(\Phi)} |\Aut(G,\Phi)|$ otherwise.
\end{lemma}
\begin{proof}
An element $\alpha \in \Aut(G,\Phi)$ induces a permutation $\alpha_\Phi \in \SG_{e(\Phi)}$ on the
set of edges of $\Phi$. By definition, the number $e_\alpha(\Phi)$ is equal to the number of cycles   in the cycle decomposition of  $\alpha_\Phi$.
The sign of a permutation is the parity of the number of its even cycles. Since the parity of the number of
odd cycles of a permutation on an $n$-element set is equal to the parity of $n$, we have $(-1)^{e_\alpha(\Phi)} = \sign(\alpha_\Phi) (-1)^{e(\Phi)}$.
The sign function gives a homomorphism from $\Aut(G,\Phi)$ to the cyclic group of order 2, which is surjective if and only if $\Aut(G,\Phi)$ contains an odd permutation.  If it is surjective, then half of the elements of $\Aut(G,\Phi)$ have each sign, so
$\sum_{\alpha\in \Aut(G,\Phi)} \sign(\alpha_\Phi)=0;$ otherwise $\sum_{\alpha\in \Aut(G,\Phi)} \sign(\alpha_\Phi)=(-1)^{e(\Phi)}|\Aut(G,\Phi)|.$%
\end{proof}

\begin{proof}[Proof of Proposition~\ref{prop:eOutFnFGC}]
Use Lemma~\ref{lmm:AutSum}, Eq.~\eqref{eq:eendef} and
Theorem~\ref{thm:tau}.
\end{proof}

 Each pair consisting of a forested graph and an automorphism contributes to the sum in the definition of $\ee_n$,
so evaluating this sum means doing a weighted count of forested graphs and their automorphisms. We will do this counting by means of generating functions, i.e.\ formal power series whose coefficients encode the counts we are interested in.

We can also use  formal power series to describe the relation between $\ee_n$ and $e(\Out(F_{n+1}))$.
By standard topological  quantum field theory convention, we use the symbol $\hbar$ as the formal variable that marks the negative Euler characteristic of the graphs.

\begin{theorem}
\label{thm:eeOutFn}
\begin{align*} \sum_{n \geq 0} \ee_n \hbar^{n} = \prod_{n = 1}^\infty \left( \frac{1}{1-\hbar^n} \right)^{e(\Out(F_{n+1}))}. \end{align*}
\end{theorem}

\begin{proof}
For any  (possibly disconnected) admissible graph $G$ and subforest $\Phi\subset G$, set $\theta(G,\Phi) = (-1)^{e(\Phi)} \hbar^{-\chi(G)}$.  Since an admissible graph is either trivial or has strictly negative Euler characteristic, $\theta(G,\Phi)\in \Q[[\hbar]]$.
The function $\theta$ factors over connected components, i.e.\ $$\theta((G_1,\Phi_1) \sqcup (G_2,\Phi_2)) = \theta(G_1,\Phi_1) \theta(G_2,\Phi_2).$$
By Proposition~\ref{prop:eOutFnFGC}(1),
$$
\sum_{n \geq 0}
\ee_n
\hbar^{n}
=
\sum_{
 [\grph,\Phi]\,
\text{even}
}
\theta(\grph,\Phi),
$$
where we sum over all isomorphism classes of (possibly disconnected) even forested graphs $[\grph,\Phi]$.

Each isomorphism class $[G,\Phi]$ can be described by giving a set of isomorphism classes of connected graphs together with the multiplicity with which each connected class appears in the disconnected class. In the sum above, a component $[g,\varphi]$ of $[G,\Phi]$ such that $\varphi$ has an odd number of edges can appear at most once, as otherwise $[G,\Phi]$ would have an odd  automorphism. Components with an even number of edges in the forest can appear with any multiplicity. The sum of the function $\theta$ over all  even admissible forested graphs  hence satisfies the following identity:
\begin{align*} \sum_{n \geq 0} \ee_n \hbar^{n} = \sum_{[G,\Phi]\,\text{even}} \theta(G,\Phi) = \left( \prod_{ \substack{[g,\varphi]\\
e(\varphi)\, \text{odd} } } \left( 1+ \theta(g,\varphi) \right) \right) \left( \prod_{ \substack{[g,\varphi]\\
e(\varphi)\, \text{even} } } \sum_{m \geq 0} \theta(g,\varphi)^m \right), \end{align*}
where the $[g,\varphi]$ are isomorphism classes of connected forested graphs and $e(\varphi)$ denotes the number of edges of $\varphi$.

Using the fact that $\theta(g,\varphi)$ is $-\hbar^{-\chi(g)}$  if $[g,\varphi]$ is odd and $\hbar^{-\chi(g)}$ if $[g, \varphi]$ is even and evaluating the sum over $m$, we obtain
\begin{align*} \sum_{n \geq 0} \ee_n \hbar^{n} = \left( \prod_{ \substack{ [g,\varphi] \\
e(\varphi)\, \text{odd} } } \left( 1- \hbar^{ -\chi(g)} \right) \right) \left( \prod_{ \substack{ [g,\varphi] \\
e(\varphi)\, \text{even} } } \frac{1}{ 1 -\hbar^{-\chi(g)}} \right) = \prod_{n \geq 1} \frac{(1-\hbar^n)^{\beta_{\text{odd}}^n} } { (1-\hbar^n)^{\beta_{\text{even}}^n} }, \end{align*}
where $\beta_{\text{odd}}^n$ and $\beta_{\text{even}}^n$ are
the numbers of connected forested admissible graphs without odd edge-automorphisms
with Euler characteristic $-n$ with an odd or even number of edges in the forest.
A connected graph with Euler characteristic $-n$ has rank $n+1$. Hence, by Proposition~\ref{prop:eOutFnFGC}(2), $\beta_{\text{even}}^n - \beta_{\text{odd}}^n = e(\Out(F_{n+1}))$.
\end{proof}

We can now explain how  to calculate the numbers $e(\Out(F_{n}))$ recursively from the numbers $\ee_n$.  Recall that the classical {\em M\"obius function}  $\mu$ is defined recursively for positive integers by
 $\mu(1) = 1$ and $\sum_{d \mid n} \mu(d) = 0$ for all $n \geq 2$.
\begin{corollary}\label{cor:mM} Let $\sum_{n\geq 1} \eee_n\hbar^n = \log f(\hbar)$ be the logarithm of the series $f(\hbar) = \sum_{n\geq 0} \ee_n\hbar^n$.
The numbers  $\eee_n$ for $n\geq 1$ are given recursively by
\begin{align*} \eee_n = \ee_n - \frac{1}{n} \sum_{k=1}^{n-1} k \eee_k \ee_{n-k} \end{align*}
and  $e(\Out(F_{n+1}))$ for $n\geq 1$ by
\begin{align*} e(\Out(F_{n+1})) = \sum_{d \mid n} \frac{\mu(d)}{d} \eee_{n/d}. \end{align*}
\end{corollary}
\begin{proof}
The recursive expression for $\eee_n$ in terms of $\ee_n$ follows by taking the derivative of $\log f(x)$ with respect to   $\hbar$ and using $(\log f(x))' = {f'(x)/f(x)}$. Taking the logarithm of the statement of Theorem~\ref{thm:eeOutFn},  using $\log(1/(1-x)) = \sum_{n=1}^\infty x^n/n$ and the definition of the M\"obius function gives the second formula.
\end{proof}

\section{An effective formula for \texorpdfstring{$e(\Outn)$}{eOutFn}}
\label{sec:effective}
\subsection{Admissible trees}

In the last section we reduced the problem of computing $e(\Out(F_n))$ to the problem of doing a weighted count of forested graphs.  In order to count forested graphs  we will first count trees, then forests, then ways of matching the leaves of the forests to form forested graphs.   Throughout the rest of the paper we fix the following conventions and terminology.  By a
  {\em tree} we mean a connected graph with no cycles.  A {\em forest} is a  disjoint union of trees.  We give univalent vertices of trees and forests a special role and call them \emph{leaves}; other vertices are called {\em internal vertices}.  A tree or forest is {\em admissible} if it has no isolated or bivalent vertices.  
A \emph{rooted tree} is a tree where one univalent vertex is distinguished as the \emph{root} while the other univalent vertices are still called leaves.     We require admissible trees to have either a root or at least one internal vertex, so a single 1-cell which connects two non-root leaves is not allowed. The   {\em internal edges} of a tree or forest are the $1$-cells that are not attached to leaves or the root.
Our trees, rooted trees, and forests will be \emph{labeled}. That means each leaf shall be decorated with a unique element from some given finite set $U$.

The set of all admissible labeled (or labeled rooted) trees forms a  combinatorial species  in the sense of Joyal~\cite{joyal1981theorie} (see also~\cite{BLL}).
Since our counting problems fit neatly into the theory of species,  we review  the relevant parts of this theory in the next subsection.

\subsection{Species and generating functions}\label{sec:species}

A {\em combinatorial species} is a functor from the group\-oid of finite sets to itself. More explicitly, a species $\mathcal S$ associates to every finite set $U$ of \emph{labels}, a finite set $\mathcal S[U]$ of \emph{combinatorial objects} such that every bijection $U \rightarrow V$ gives rise to a bijection of sets $\mathcal S[U] \rightarrow \mathcal S[V]$ in a way compatible with composition. In the case $U = \{1,\ldots,n\}$ we write $\mathcal S[n]$ for $\mathcal S[U]$.  An element of $\mathcal S[U]$ for some $U$ is an {\em $\mathcal S$-object}, or more informally, an object  of $\mathcal S$.

For every species $\mathcal S$ there is a natural action of $\SG_n$ on $\mathcal S[n]$ that permutes  the labels.
The  orbit  of an element $\phi \in \mathcal S[n]$ under this action is an isomorphism class of combinatorial objects.
We will also call such a class an \emph{unlabeled object}.
The \emph{stabilizer} of an element $\phi \in \mathcal S[n]$ is the group of relabelings that leaves the object invariant. We will denote this stabilizer by  $\Aut(\phi)$.

A \emph{set partition} of $U$ with $k$ \emph{blocks} is a collection $\pi = \{B_1, \ldots, B_k\}$ of $k$ mutually disjoint subsets of $U$ such that $\bigcup_{B \in \pi} B = U$.
From a given species $\mathcal S$ we can construct a new species $\Set^k( \mathcal S)$ by associating to a set $U$ all collections of objects $\{\phi_1,\ldots,\phi_k\}$ such that $\phi_1 \in \mathcal S[B_1], \ldots, \phi_k \in \mathcal S[B_k]$ for some partition of $U$ into blocks $B_1,\ldots,B_k$.
The functor $\Set^k(\mathcal S)$ is the species of sets of size $k$ which contain objects of $\mathcal S$.
To make the formulas more compact we agree that $\Set^0(\mathcal S)[0]$ contains one element, the empty set, and that
$\Set^k(\mathcal S)[0]$ is empty for all $k \geq 1$.

Let $\mathbb A$ be an algebra and  $\omega:\mathcal S[U] \rightarrow \mathbb{A}$  a {\em weight function} that associates some element of  $\mathbb{A}$ to each object of $\mathcal S[U]$ in a way that is independent of the labeling. We can extend the weight function $\omega$ to $\Set^k( \mathcal S)$ by setting the weight of the empty set to $1$ and $\omega( \{\phi_1,\ldots,\phi_k \}) = \prod_{\ell=1}^k \omega(\phi_\ell)$.

When working with a formal power series $F(x)=\sum a_ix^i$, we will frequently use the  {\em coefficient extraction operator} $[x^n]$ to extract the coefficient of $x^n$, so $[x^n]F(x)=a_n$.
We now define two
formal  power series, in $\mathbb A[[x]]$   and $\mathbb A[[x,y]]$ respectively, by
\begin{align} \label{eq:expgen} S(x) &= \sum_{n \geq 1} \frac{x^n}{n!} \sum_{\phi \in \mathcal S[n]} \omega(\phi) &&\text{and} & S_{\Set}(x,y) &= \sum_{n,k \geq 0} \frac{x^n}{n!} y^k \sum_{\Phi \in \Set^k(\mathcal S)[n]} \omega(\Phi). \end{align}
These are exponential generating functions for weighted counts: $[x^n]S(x)$ is ${1/n!}$ times the number of   objects $\phi\in\mathcal S[n]$, counted with weight $\omega(\phi)$,  and $[x^ny^k]S(x,y)$ is ${1/n!}$ times the number of   elements $\Phi\in\Set^k(\mathcal S)[n]$, counted with weight $\omega(\Phi)$.
We will make frequent use of the following \emph{exponential formula}, a standard lemma that relates the power series $S$ and $S_\Set$.
\begin{lemma}
\label{lmm:expform}
\begin{align*} S_{\Set}(x,y) = \exp \left( y\, S(x) \right). \end{align*}
\end{lemma}

\begin{proof}
Use the expansion $\exp(X) = \sum_{n \geq 0} X^n/n!$ and the fact that the number of set partitions of a set of cardinality $n$ into $k$ blocks with sizes $m_1,\ldots,m_k$ is given by
${n!/(k! m_1! \ldots m_k!)}$.
\end{proof}

We next want   to count labeled combinatorial objects while keeping track of automorphisms. For this, we will need more sophisticated generating functions,  called \emph{cycle index series}.  The terminology comes from the notion of cycle type for a permutation, which we now review.

 Let $n\in\mathbb N$ be a positive integer.  An {\em integer partition} of $n$ is a sequence of positive integers $\lambda=(\lambda_1,\lambda_2,\ldots,\lambda_\ell)$ such that $\lambda_1\geq \lambda_2\geq\ldots\geq\lambda_\ell>0$ and $n=\lambda_1+\ldots+\lambda_\ell$. The $\lambda_i$ are called \emph{parts} of $\lambda$. We write $\lambda \vdash n$ or $|\lambda|=n$.  The integer $n$ is the {\em size} of the partition and $\ell$, the number of parts, is the {\em length} of the partition.  %
We will often use the more compact notation $\lambda=[1^{m_1}\cdots n^{m_n}],$ indicating that $\lambda$ has $m_k$ parts of size $k$ (the terms with $m_i=0$ are usually omitted; for example, $\lambda=(4,4,1,1,1)$ is written $[1^34^2]$).  In this notation the length of $\lambda$ is $\ell(\lambda)=m_1+\cdots +m_\ell$,  the size is $|\lambda|=m_1 + 2m_2 + \cdots + n m_n$. %
Each permutation $\pi \in \SG_n$ factors uniquely as a product of disjoint cycles.
If the orders of these cycles are $\lambda_1, \ldots, \lambda_\ell$, where $\lambda_1 \geq \ldots \geq \lambda_\ell$, then $\lambda(\pi) = (\lambda_1,\ldots,\lambda_\ell)$ is a partition of $n$ called the \emph{cycle type} of $\pi$.

Given an integer partition $\lambda= [1^{m_1}2^{m_2}\ldots n^{m_n}]$  let $x^\lambda$ denote the monomial  $x_1^{m_1}x_2^{m_2}\cdots x_n^{m_n}$.    A formal power series in the variables $\bb x=\{x_1,x_2,\ldots \}$ is an infinite sum of terms $a_\lambda x^\lambda$, with $a_\lambda\in\mathbb A$, such that, for all $n$, we only have a finite number of terms if restrict to partitions with $|\lambda| \leq n$. If $\alpha$ is a permutation of cycle type $\lambda$ we also define $x^\alpha=x^\lambda$.

For a given species $\mathcal S$, let $\mathcal{AS}$ be the species of pairs $(\phi,\alpha)$, where $\phi$ is an object of $\mathcal S$ and $\alpha$ an automorphism of $\phi$:
\begin{align*} \mathcal{AS}[n]=&\{(\phi,\alpha)| \phi \in \mathcal S[n] \text{ and }\alpha\in\Aut(\phi)\leq \mathfrak S_n\}. \end{align*}
Let $\omega$ be a weight function that attaches to an object $(\phi,\alpha)$ of  $\mathcal{AS}$ an element $\omega(\phi,\alpha)\in\mathbb A$   that does not  depend on the labeling.  The {\em cycle index series for $\mathcal S$  weighted by $\omega$}  is  the   formal power series $\mathbf S$ in $\bb x$ whose terms  mark the cycle type of the automorphism, i.e.
\begin{align*} {\mathbf S}(\bb x) = \sum_{n \geq 1} \frac{1}{n!} \sum_{(\phi,\alpha) \in \mathcal{AS}[n]} \omega(\phi,\alpha)\, x^\alpha. \end{align*}
In other words, $[x^\lambda]{\mathbf S}(\bb x)$ is ${1/|\lambda|!}$ times the weighted count of pairs $(\phi,\alpha)\in\mathcal{AS}$   such that $\alpha$ has cycle type $\lambda$.
The cycle index series ${\mathbf S}(\bb x)$ generalizes the generating function of labeled objects in Eq.~\eqref{eq:expgen}, as we recover
$S(x)$ by setting $\omega(\phi, \id)=\omega(\phi)$, $x_1 =x$ and all other $x_k$-variables to $0$.

We now want to extend our cycle index series on $\mathcal S$ to one for $\Set^k(\mathcal S)$, and show how to compute it from ${\mathbf S}(\bb x)$.  An element $\Phi\in \Set^k(\mathcal S)[n]$ is a set of $k$ elements of $\mathcal S$ with a total of $n$ labels.  An element $\malpha\in\Aut(\Phi)$ permutes the labels but preserves the set.  If some elements of $\mathcal S$ are isomorphic then $\malpha$ may permute them, so $\malpha$ induces a permutation $\malpha_\Phi\in\mathfrak S_k$.  We introduce a new infinite set of variables $\bb y=\{y_1,y_2,\ldots \}$ to mark the cycle type of $\malpha_\Phi$, and define
\begin{align*} {\mathbf S}_\Set(\bb x, \bb y) = \sum_{n,k \geq 0} \frac{1}{n!} \sum_{(\Phi,\malpha) \in {\mathcal A}\Set^k (\mathcal S)[n]} \omega(\Phi,\malpha)\, x^\malpha y^{\malpha_\Phi}. \end{align*}
Here $\omega(\Phi,\malpha)=\omega(\phi_1,\malpha_1)\ldots\omega(\phi_\ell,\malpha_\ell)$, where
\begin{itemize}
\item $\malpha_\Phi$ has cycle type $\lambda=(\lambda_1,\lambda_2,\ldots,\lambda_\ell)$ of length $\ell$
\item $\phi_i\in\Phi$ is a representative of the $i$th cycle, which has size $\lambda_i$
\item $\malpha_i$ is the restriction of $\malpha^{\lambda_i}$ to $\phi_i.$
\end{itemize}

\begin{proposition}
\label{prop:expformaut} Let   ${\mathbf S}(\bb x_{[k]})$ be the series obtained from ${\mathbf S}(\bb x)$ by replacing each occurrence of $x_i$ by $x_{ki}$.  Then
\begin{align*} {\mathbf S}_\Set(\bb x, \bb y) = \exp \left( \sum_{k \geq 1} y_k \frac{{\mathbf S}(\bb x_{[k]}) }{k} \right), \end{align*}
\end{proposition}

Ultimately, this proposition goes back to P{\'o}lya~\cite{polya1937}.
It follows from Lemma~\ref{lmm:expform} combined with the generating function for \emph{wreaths} of $\mathcal S$ structures.
 We give a proof below, but refer to  \cite{BLL} Chapter 4.3 for a more detailed argument, using a slightly different type of weight function $\omega$.

To prove Proposition~\ref{prop:expformaut} we first introduce a new species
$\ACyc^{k} \mathcal S$.  An element of $\ACyc^{k} \mathcal S[n]$ is a pair  $(\Phi, \malpha)$ where
$\Phi= \{\phi_1, \ldots,\phi_k\}$ is a collection of   objects in $\mathcal S$ with a total of $n$ labels, and $\malpha\in \Aut(\Phi)$ is a permutation of the labels of $\Phi$ that cyclically permutes the $\phi_i$, i.e.\ $\malpha_\Phi$ is a $k$-cycle in $\mathfrak S_k$.  In particular, all of the $\phi_i$ must be isomorphic, i.e.\ equivalent as unlabeled objects.

\begin{lemma}
\label{lmm:kranz}
\begin{align*} \sum_{n \geq 1} \frac{1}{n!} \sum_{(\Phi,\malpha) \in \ACyc^{k}( \mathcal S)[n]} \omega(\Phi,\malpha) x^\malpha=\frac{{\mathbf S}(\bb x_{[k]}) }{k}. \end{align*}
\end{lemma}

\begin{proof}
Let $(\Phi,\malpha)$ be an element of $\ACyc^k\mathcal S[n]$.  We claim that $(\Phi,\malpha)$ is equivalent to
 a tuple $(\kappa, \pi, \phi,\alpha, \gamma_1,\ldots,\gamma_{k-1}),$ where\begin{itemize}
\item $\kappa$ is a $k$-cycle in $\mathfrak S_k,$
\item $\pi$ is a partition of $\{1,\ldots,n\}$ into $k$ blocks, each of size $d$,
\item $\phi$ is an object of $\mathcal S[d],$
\item $\alpha$ is an element of $\Aut(\phi),$  and
\item $\gamma_i$ is an element of $\mathfrak S_d$ for each $i=1,\ldots,k-1.$
\end{itemize}
By definition $\Phi=\{\phi_1,\ldots,\phi_k\}, $ where  $\phi_i\in \mathcal S[B_i]$ for some $B_i\subset \{1,\ldots,n\}$, and $\malpha_\Phi$ is a  $k$-cycle, so we take $\kappa=\malpha_\Phi$ and $\pi=\{B_1,\ldots,B_k\}$.    Since $\malpha$ acts cyclically on $\Phi$, all of the $\phi_i$ are isomorphic, so  in particular they all have the same number $d$ of labels.  The unique order-preserving bijection $B_i\to \{1,\ldots,d\}$ identifies each $\phi_i$ with an element $\widehat\phi_i\in \mathcal S[d]$; set $\phi=\widehat\phi_k$.
The permutation $\malpha^k$ sends each $\phi_i$ to itself, so determines an automorphism $\alpha_i$ of $\phi_i$ and therefore of $\widehat\phi_i$; set $\alpha=\alpha_k$.   Note that all of the $\alpha_i$ have the same cycle type $\lambda=(\lambda_1,\ldots,\lambda_\ell)$, so the cycle type of $\malpha$ is $k\lambda=(k\lambda_1, \ldots,k\lambda_\ell)$.   For each $i=1,\ldots,k-1$, the $i$th power of the $k$-cycle $\malpha_\Phi$ maps $\phi_k$ to $\phi_{\malpha^i_\Phi(k)}$; let $\gamma_i$ be the induced  isomorphism $\gamma_i\colon \widehat\phi_k\to \widehat\phi_{\malpha^i _\Phi(k)}$ in $\mathfrak S_d$.

Finally, note that $\omega(\Phi,\malpha)=\omega(\phi_k,\malpha^k)=\omega(\phi,\alpha)$.  Since $\ACyc^k(\mathcal S)[n]$ is empty unless $k|n$, we can now write the left hand side of the equation in the statement of the lemma as,
\begin{equation}
\label{eq:ast1}
\sum_{n=dk\geq 1} \frac{1}{n!}
C_{n,k}
\sum_{(\phi,\alpha) \in \mathcal{AS}[d]}
\omega(\phi,\alpha)  x^{k \circ \alpha},
\end{equation}
where  $k \circ \alpha$ has cycle type $k\lambda$ and
$C_{n,k}$ is the number of partitions $\pi$ of $n$ into $k$ equal parts times the number of cyclic permutations $\kappa$ in $\mathfrak S_k$ times the number of permutations $\gamma_1,\ldots,\gamma_{k-1}$ in $(\mathfrak S_d)^{k-1}$, i.e.
\[
C_{n,k} =
  \frac{1}{k!} \frac{n!}{(d!)^k} \cdot (k-1)! \cdot (d!)^{k-1} = \frac{n!}{d! k }.
\]
Plugging this into Eq.~\eqref{eq:ast1} gives the statement of the lemma.
\end{proof}

\begin{proof}[Proof of Proposition~\ref{prop:expformaut}]
We have
\begin{align*} \mathcal A\Set(\mathcal S)[n]=&\{(\Phi,\malpha)| \Phi \text{ is a finite set of objects of $\mathcal S$ with a total of $n$ distinct labels, } \\&\text{ and }\malpha\in\Aut(\Phi)\leq \mathfrak S_n\}. \end{align*}
Since every permutation $\malpha$ can be uniquely decomposed into a product of cycles, we have an isomorphism of species
$$
\mathcal A\Set(\mathcal S)\cong \bigcup_{\ell \geq 1} \Set^{\ell} \left(\bigcup_{k \geq 1} \ACyc^k (\mathcal S) \right).
$$
Therefore the statement follows from an application of Lemma~\ref{lmm:expform} to the right hand side, using   the
generating function of the species $\ACyc^k (\mathcal S)$ given by
Lemma~\ref{lmm:kranz} and summation over all $k$, using the variables $\bb y$ to keep track of the cycle
type of the permutation.
\end{proof}

\subsection{Matchings}
\label{sec:matchings}
In order to form a forested graph with an automorphism preserving the forest, we will start with a forest equipped with an automorphism $\alpha$,  then pair its leaves by a fixed-point free involution that commutes with $\alpha$.  In this subsection we
 apply the counting methods from the previous subsection to count the number of such involutions.

 A fixed-point free involution is also called a {\em matching}; it divides the set into orbits of size $2$, so we will use the species $\Etwo$  of sets of cardinality $2$,   i.e.~$\Etwo[2] = \{1,2\}$ and $\Etwo[n] = \emptyset$ for  $n \neq 2$.   The generating function (with trivial weight) is $E(x) = {x^2/2}$.
Matchings on
a set of $2k$ elements correspond to elements of $\Set^k(\Etwo)[2k]$; for   example the elements of \,$\Set^2 \Etwo [4]$ are
$\{\{1,2\},\{3,4\}\}, \{\{1,3\},\{2,4\}\}$ and $\{\{1,4\},\{2,3\}\}$.
By Lemma~\ref{lmm:expform} we have
\begin{align*} \EtwoSet (x,y) = \exp\left(y \frac{x^2}{2} \right) = \sum_{k \geq 0} (2k-1)!! \frac{x^{2k} }{(2k)!} y^k, \end{align*}
 where  the second equality is obtained  using the expansion $\exp(X) = \sum_{n \geq 0} {X^n/n!}$ and the formula $(2k-1)!! = {(2k)!/(k!2^k)}.$  Thus we recover the (easy) fact that  the number of
matchings of a set of cardinality $2k$  is $(2k-1)!!$, and there are none if the cardinality is odd.

Now consider   the species $\mathcal{AE}=\mathcal{AE}[2]$ of sets of cardinality $2$ with automorphisms on them. A set of cardinality $2$ has only  two automorphisms, the trivial one (marked by $x_1x_1=x_1^2$)  and the transposition (marked by $x_2$). The cycle index series of $\Etwo$ with trivial weight is therefore
$$
{\mathbf E}(\bb x) = \frac{1}{2!} \sum_{(\phi,\alpha)\in \mathcal{AE}[2]}
x^\alpha = \frac{1}{2} \left( x_1^2 + x_2 \right).  $$
A matching of $2k$ elements corresponds to an element   $\Phi\in\Set^k (\Etwo)[2k]$.  The automorphisms of such an element $\Phi$   are permutations that commute with the corresponding fixed-point free involution $\iota=\iota_\Phi$. We may count such permutations using Proposition~\ref{prop:expformaut}, which gives
\begin{align} \label{eq:E2} {\mathbf E}_{\Set}(\bb x, \bb y) = \exp \left( \sum_{k \geq 1} \frac{y_k}{2k} \left( x_k^2 + x_{2k} \right) \right). \end{align}

\begin{proposition}
\label{prop:fpfinvcountaut}
\begin{align*} \sum_{n \geq 0} \frac{1}{(2n)!} \sum_{(\iota,\alpha)} x^\alpha = \exp \left( \sum_{k \geq 1} \frac{1}{2k} \left( x_k^2 + x_{2k} \right) \right), \end{align*}
where the sum is over all  pairs $(\iota,\alpha)$ consisting of a matching $\iota$ of $\{1,\ldots,2n\}$ and a permutation $\alpha \in \SG_{2n}$ that commutes with $\iota$.
\end{proposition}

\begin{proof}
By definition we have $${\mathbf E}_{\Set}(\bb x, \bb y)=\sum_{n,k\geq 0} \frac{1}{n!} \sum_{(\Phi,\alpha)\in\mathcal{A}\Set^k(\Etwo)[n]}%
 x^\alpha y^{\alpha_\phi}. $$
 Since $\Set^k\Etwo[n]$ is empty unless $n$ is even and $k=n/2$ we can rewrite this as
 $${\mathbf E}_{\Set}(\bb x, \bb y)=\sum_{n\geq 0}\frac{1}{(2n)!} \sum_{(\Phi,\alpha)\in\mathcal{A}\Set^n\Etwo[2n]}%
 x^\alpha y^{\alpha_\phi}. $$
As noted previously, we may identify the pairs $(\iota,\alpha)$ in the statement of the proposition with elements $(\Phi,\alpha)$ of $\mathcal{A}\Set^n(\Etwo)[2n]$.  Using Eq.~\eqref{eq:E2} and setting $y_k=1$ for all $k$  gives the result.
\end{proof}

\begin{corollary}\label{cor:matchings}
The number $\eta_\lambda$ of %
matchings that commute with a given permutation $\alpha \in \SG_n$ of cycle type $\lambda=[1^{m_1}2^{m_2}\ldots n^{m_n}]$ is given by the formula
$$
 \eta_{\lambda}  = \prod_{k = 1}^n \eta_{k, m_k},
 $$
 where
$$
\eta_{k,2s}=
\begin{cases}
k^{s} (2s -1)!!  & \text{if $k$ is odd}\\
\sum_{r=0}^{s} \binom{2s}{2r} k^r (2r-1)!! & \text{if $k$ is even}
\end{cases}
$$
and
 $$
\eta_{k,2s+1}=
\begin{cases}
0 & \text{if $k$ is odd}\\
\sum_{r=0}^{s} \binom{2s+1}{2r} k^r (2r-1)!! & \text{if $k$ is even}
\end{cases}.
$$
 \end{corollary}
\begin{remark}
This statement can also be proved without the theory of species by elementary combinatorial means. We have used species here as an easy example of  the method, which we use heavily in the next sections.  
Moreover, the extension to the analogous Corollary~\ref{cor:matchings_odd} is done effortlessly using species.
\end{remark}

\begin{proof}
Since $\iota$ commutes with $\alpha$ if and only if $\alpha$ commutes with $\iota$, the left hand side of Proposition~\ref{prop:fpfinvcountaut} can be written
\[
\sum_{n\geq 0} \frac{1}{(2n)!}\sum_{\alpha\in\SG_{2n}} \#\{\iota\in \SG_{2n}|\iota\circ\alpha=\alpha\circ\iota\}x^\alpha.
\]
Now note that the number of matchings that commute with $\alpha$  depends only on the cycle type of $\alpha$. If we denote   the set of permutations in $\mathfrak S_{n}$  with cycle type $\lambda$ by $\mathfrak S_{n}^\lambda$, the formula  becomes
\[=\sum_{n \geq 0} \frac{1}{(2n)!}
\sum_{\lambda \parts 2n}
|\SG_{2n}^\lambda|
\eta_\lambda
x^\lambda
=
\sum_{n \geq 0} \frac{1}{n!}
\sum_{\lambda \parts n}
|\SG_{n}^\lambda|
\eta_\lambda
x^\lambda,
\]
where we used the fact that $\eta_\lambda=0$ if $|\lambda|$ is odd.  For $\lambda=[1^{m_1}2^{m_2}\ldots n^{m_n}]$ it is easy to check that the number of  permutations in $\mathfrak S_n^\lambda$  is
\[
|\SG_{n}^\lambda| = \frac{n!}{1^{m_1}m_1! 2^{m_2} m_2! \cdots n^{m_n} m_n!},
\]
so the left hand side of Proposition~\ref{prop:fpfinvcountaut} is equal to
\begin{equation}\label{LHS}
\sum_{n \geq 0}
\sum_{\lambda \parts n}
\frac{
\eta_\lambda}{1^{m_1}m_1! 2^{m_2} m_2! \cdots n^{m_n} m_n!}
x^\lambda.
\end{equation}

We now turn to the right hand side of the equation stated in Proposition~\ref{prop:fpfinvcountaut}. Using $e^X = \sum_{s \geq 0} X^s/s!$, $(2s-1)!! = {(2s)!/(s!2^s )}$ and $(x+2)^s = \sum_{r=0}^s \binom{s}{r} 2^{s-r} x^r$ and the definition of the numbers $\eta_{k,m}$ above, we find that for even $k$,
\begin{align*} \exp\left(\frac{x^2}{2k}\right) &= \sum_{s \geq 0} \frac{x^{2s}}{2^s k^s s!} = \sum_{s \geq 0} k^s (2s-1)!! \frac{x^{2s}}{k^{2s} (2s)!} = \sum_{m \geq 0} \eta_{k,m} \frac{x^m}{k^m m!}\\
\intertext{and for $k$ odd, we have} {\exp\left(\frac{x^2+2x}{2k}\right)} &= \sum_{s \geq 0} \frac{x^s(x+2)^{s}}{2^s k^{s} s!} = \sum_{s \geq 0} \sum_{r \geq 0} \binom{s}{r} 2^{-r} \frac{x^{s+r}}{k^{s} s!} = \sum_{r \geq 0} \sum_{s \geq r} \frac{1}{r!(s-r)!} 2^{-r} \frac{x^{s+r}}{k^{s}} \\
&= \sum_{r \geq 0} \sum_{s \geq 2r} \frac{1}{r!(s-2r)!} 2^{-r} \frac{x^{s}}{k^{s-r}} = \sum_{s \geq 0} \sum_{s/2 \geq r \geq 0} \frac{s!}{(2r)!(s-2r)!} k^r \frac{(2r)!}{2^r r!} \frac{x^{s}}{k^{s}s! } \\
&= \sum_{m \geq 0} \eta_{k,m} \frac{x^m}{k^m m!}. \end{align*}
Using these two equations together with the fact that
$$
\sum_{k \geq 1}
\frac{1}{2k} \left( x_k^2 + x_{2k} \right)
 =\sum_{k \geq 1}
\left(
\frac{ x_{2k-1}^2}{2(2k-1)}
+
\frac{x_{2k}^2 + 2 x_{2k}}{4k}\right)
$$
  the right hand side of Proposition~\ref{prop:fpfinvcountaut} becomes
\begin{equation}\label{RHS}
\prod_{k \geq 1}
\left(
\sum_{m \geq 0}
\frac{\eta_{k,m}}{k^m m!}
x_k^m
\right).
\end{equation}
Equating coefficients of~(\ref{LHS}) and (\ref{RHS}) now gives the proposition.
\end{proof}

 \subsection{Rooted trees}

In order to count forested graphs with automorphisms, we begin by counting forests with automorphisms. In order to count forests with automorphisms, we begin by counting trees with automorphisms, and in order to count trees with automorphisms, we begin by counting rooted trees with automorphisms.

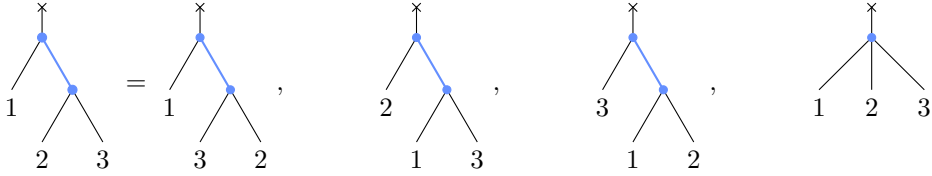
\begin{figure}
{
\def\scale{.8}
\begin{align*} & \begin{tikzpicture}[baseline={(0,{-sqrt(3)/2*\scale})}] \coordinate (t) at (0,{\scale/2}); \coordinate (n) at (0,0); \coordinate (v11) at ({-1/2*\scale},{-sqrt(3)/2*\scale}); \coordinate (v12) at ({ 1/2*\scale},{-sqrt(3)/2*\scale}); \coordinate (v21) at ([shift=(v12)]{-1/2*\scale},{-sqrt(3)/2*\scale}); \coordinate (v22) at ([shift=(v12)]{ 1/2*\scale},{-sqrt(3)/2*\scale}); \draw (t) node[cross=2pt] {}; \draw (t) -- (n); \draw (n) -- (v11); \draw[thick,col1] (n) -- (v12); \draw (v12) -- (v21); \draw (v12) -- (v22); \draw[col1,fill=col1] (n) circle (1.7pt); \draw[col1,fill=col1] (v12) circle (1.7pt); \draw[fill=white] (v11) circle (0pt) node[below] {$1$}; \draw[fill=white] (v21) circle (0pt) node[below] {$2$}; \draw[fill=white] (v22) circle (0pt) node[below] {$3$}; \end{tikzpicture} = \begin{tikzpicture}[baseline={(0,{-sqrt(3)/2*\scale})}] \coordinate (t) at (0,{\scale/2}); \coordinate (n) at (0,0); \coordinate (v11) at ({-1/2*\scale},{-sqrt(3)/2*\scale}); \coordinate (v12) at ({ 1/2*\scale},{-sqrt(3)/2*\scale}); \coordinate (v21) at ([shift=(v12)]{-1/2*\scale},{-sqrt(3)/2*\scale}); \coordinate (v22) at ([shift=(v12)]{ 1/2*\scale},{-sqrt(3)/2*\scale}); \draw (t) node[cross=2pt] {}; \draw (t) -- (n); \draw (n) -- (v11); \draw[thick,col1] (n) -- (v12); \draw (v12) -- (v21); \draw (v12) -- (v22); \draw[col1,fill=col1] (n) circle (1.5pt); \draw[col1,fill=col1] (v12) circle (1.5pt); \draw[fill=white] (v11) circle (0pt) node[below] {$1$}; \draw[fill=white] (v21) circle (0pt) node[below] {$3$}; \draw[fill=white] (v22) circle (0pt) node[below] {$2$}; \end{tikzpicture} , & & \begin{tikzpicture}[baseline={(0,{-sqrt(3)/2*\scale})}] \coordinate (t) at (0,{\scale/2}); \coordinate (n) at (0,0); \coordinate (v11) at ({-1/2*\scale},{-sqrt(3)/2*\scale}); \coordinate (v12) at ({ 1/2*\scale},{-sqrt(3)/2*\scale}); \coordinate (v21) at ([shift=(v12)]{-1/2*\scale},{-sqrt(3)/2*\scale}); \coordinate (v22) at ([shift=(v12)]{ 1/2*\scale},{-sqrt(3)/2*\scale}); \draw (t) node[cross=2pt] {}; \draw (t) -- (n); \draw (n) -- (v11); \draw[thick,col1] (n) -- (v12); \draw (v12) -- (v21); \draw (v12) -- (v22); \draw[col1,fill=col1] (n) circle (1.5pt); \draw[col1,fill=col1] (v12) circle (1.5pt); \draw[fill=white] (v11) circle (0pt) node[below] {$2$}; \draw[fill=white] (v21) circle (0pt) node[below] {$1$}; \draw[fill=white] (v22) circle (0pt) node[below] {$3$}; \end{tikzpicture} , & & \begin{tikzpicture}[baseline={(0,{-sqrt(3)/2*\scale})}] \coordinate (t) at (0,{\scale/2}); \coordinate (n) at (0,0); \coordinate (v11) at ({-1/2*\scale},{-sqrt(3)/2*\scale}); \coordinate (v12) at ({ 1/2*\scale},{-sqrt(3)/2*\scale}); \coordinate (v21) at ([shift=(v12)]{-1/2*\scale},{-sqrt(3)/2*\scale}); \coordinate (v22) at ([shift=(v12)]{ 1/2*\scale},{-sqrt(3)/2*\scale}); \draw (t) node[cross=2pt] {}; \draw (t) -- (n); \draw (n) -- (v11); \draw[thick,col1] (n) -- (v12); \draw (v12) -- (v21); \draw (v12) -- (v22); \draw[col1,fill=col1] (n) circle (1.5pt); \draw[col1,fill=col1] (v12) circle (1.5pt); \draw[fill=white] (v11) circle (0pt) node[below] {$3$}; \draw[fill=white] (v21) circle (0pt) node[below] {$1$}; \draw[fill=white] (v22) circle (0pt) node[below] {$2$}; \end{tikzpicture} , & & \begin{tikzpicture}[baseline={(0,{-sqrt(3)/2*\scale})}] \coordinate (t) at (0,{\scale/2}); \coordinate (n) at (0,0); \coordinate (v11) at ({-sqrt(3)/2*\scale},{-sqrt(3)/2*\scale}); \coordinate (v12) at (0,{-sqrt(3)/2*\scale}); \coordinate (v13) at ({ sqrt(3)/2*\scale},{-sqrt(3)/2*\scale}); \draw (t) node[cross=2pt] {}; \draw (t) -- (n); \draw (n) -- (v11); \draw (n) -- (v12); \draw (n) -- (v13); \draw[col1,fill=col1] (n) circle (1.5pt); \draw[fill=white] (v11) circle (0pt) node[below] {$1$}; \draw[fill=white] (v12) circle (0pt) node[below] {$2$}; \draw[fill=white] (v13) circle (0pt) node[below] {$3$}; \end{tikzpicture} \end{align*}
}
\caption{All admissible rooted trees with $3$ leaves, i.e.\ all elements of  $\mathcal R[3]$. The first three rooted trees have two internal vertices each and the fourth has one internal vertex. Internal edges are colored in blue. The automorphism groups of the first, second and third rooted tree are generated by the transpositions $(2,3)$, $(1,3)$ and $(1,2)$ respectively. The automorphism group of the fourth rooted tree is the full symmetric group $\SG_3$, i.e.\ it includes all permutations of  the labels.}
\label{fig:rooted_trees}\label{fig:Rthree}
\end{figure}

Let $\mathcal R$ be the species of leaf-labeled admissible rooted trees $\rho$, i.e.\ all internal vertices of $\rho$ must have valence at least $3$ and
$\mathcal R[n]$ is the set of all such rooted trees with $n$ leaves. The elements of $\mathcal R[3]$ are depicted in Figure~\ref{fig:rooted_trees}.
The rooted tree $\rho_0$ with one root, one leaf and one 1-cell  has no internal vertices so  satisfies the definition, but it plays a special role so we call it the {\em special rooted tree}. $\mathcal{AR}$ is then the species of pairs $(\rho,\alpha)$, where   $\rho\in\mathcal R$ and    $\alpha\in\Aut(\rho).$
Recall that an automorphism of a tree is determined by what it does to the leaves. So, for a rooted tree $\rho \in \mathcal R[n]$ we can (and will) identify  the group $\Aut(\rho) \leq \SG_n$ with the usual simplicial automorphisms of the tree $\rho$.
To each pair $(\rho,\alpha)\in\mathcal{AR}$ we assign the weight  $\omega(\rho,\alpha)= (-1)^{v_\alpha(\rho)},$  where $v_\alpha(\rho)$ is the number of  $\alpha$-orbits of internal vertices.  The special rooted tree $\rho_0$ has only the identity automorphism,  and the pair $(\rho_0,\id)$ has  weight   $(-1)^0=1$.
The associated generating function
\begin{align*} {\mathbf R}(\bb x)&=\sum_{n\geq 1}\frac{1}{n!}\sum_{(\rho,\alpha)\in\mathcal{AR}[n]}(-1)^{v_\alpha(\rho)} x^\alpha \end{align*}
is the \emph{Frobenius characteristic of rooted trees}.
The first few terms are
\begin{align*} {\mathbf R}(\bb x)&=x_1 - \frac{1}{2}(x_1^2+x_2)+\frac{1}{6}3(x_1^3+x_1x_2)-\frac{1}{6}(x_1^3+3x_1x_2+2x_3)+\ldots \end{align*}
The first term comes from the special rooted tree $\rho_0$. The second term comes from the rooted tree with one internal vertex and two leaves branching from it. We have two automorphisms that either switch the leaf-labels or do not. The third and fourth terms in this sum come from the rooted trees in Figure~\ref{fig:Rthree} and their automorphisms.

\begin{remark}
The name \emph{Frobenius characteristic}
comes from an interpretation of these generating functions in the context of the representation theory of the symmetric group:
We can think of ${\mathbf R}(\bb x)$ as an element of the
ring of \emph{symmetric functions} and each homogeneous part of ${\mathbf R}(\bb x)$ as the image of a certain representation of the symmetric group under the Frobenius characteristic map.
These representations are the vector spaces generated by
the elements of $\mathcal{R}[n]$ with the action of $\SG_n$
which alternates with $(-1)^{v_\alpha(\rho)}$ as above.
However, in this paper we will not make use of this
representation theoretical interpretation of these objects.
\end{remark}

The next proposition shows that the characteristic ${\mathbf R}(\bb x)$ has a remarkably simple form.

\begin{proposition}
\label{prop:rootedtrees}
\begin{align*} {\mathbf R}(\bb x) = \sum_{n \geq 1} \frac{\mu(n)}{n} \log(1+x_n), \end{align*}
where $\mu$ is the  M\"obius function.
\end{proposition}

\begin{proof}
The formula  will be established by relating pairs $(\rho,\alpha)$   recursively to pairs $(\Phi,\malpha)$ consisting of a set of rooted trees $\Phi$ and an automorphism $\malpha\in\Aut(\Phi)$.

Let $\Phi$ be a collection of two or more admissible rooted trees and $\malpha\in\Aut(\Phi)$. Set  $\omega(\Phi,\malpha)=(-1)^{v_\malpha(\Phi)}$. We can form a new admissible rooted tree  $\rho_\Phi$  by  gluing the roots of the trees in $\Phi$ together and growing a new root from the resulting vertex. The resulting rooted tree has one new internal vertex,  which is fixed by any automorphism, and at least $2$ leaves. The natural map $\Aut(\Phi)\to \Aut(\rho_\Phi)$ is a bijection, so we will use the same name $\gamma$, and we have  $\omega(\rho_\Phi,\gamma)=-\omega(\Phi,\gamma)$.
The only pair $(\rho,\alpha)$ which  cannot be obtained in this way is $(\rho_0,\id)$, so
\begin{align} \label{eq:ast2} {\mathbf R}(\bb x) &= x_1 - \sum_{k \geq 2}\sum_{n\geq k} \frac{1}{n!} \sum_{(\Phi,\malpha) \in \mathcal A\Set^k (\mathcal R)[n]} (-1)^{v_\malpha(\Phi)} x^\malpha, \end{align}
Proposition~\ref{prop:expformaut} tells us
\begin{align*} {\mathbf R}_\Set(\bb x,\bb y) &= \sum_{k \geq 0}\sum_{n \geq k} \frac{1}{n!} \sum_{(\Phi,\malpha) \in \mathcal A\Set^k (\mathcal R)[n]} (-1)^{v_\malpha(\Phi)} x^\malpha y^{\malpha_\Phi} = \exp \left( \sum_{k \geq 1} y_k \frac{{\mathbf R}(\bb x_{[k]}) }{k} \right), \end{align*}

Setting  $y_k = 1$ for all $k$  and subtracting the terms with $k=0$ or $k=1$   gives the summation term in Eq.~\eqref{eq:ast2} above, i.e.
\begin{align}\label{eqn:Rx} {\mathbf R}(\bb x)&=x_1-\left(\exp \left( \sum_{k \geq 1} \frac{{\mathbf R}(\bb x_{[k]}) }{k} \right) -1 - {\mathbf R}(\bb x_{[1]})\right). \end{align}
 Since ${\mathbf R}(\bb x_{[1]})={\mathbf R}(\bb x)$ this gives
$ \sum_{k \geq 1} {{\mathbf R}(\bb x_{[k]}) /k} = \log(1+x_1) $
and because $\bb x_{[k]} = (x_k,x_{2k},\ldots)$, we also have
$ \sum_{k \geq 1} {{\mathbf R}(\bb x_{[k n]}) /k} = \log(1+x_n) $
for all $n \geq 1$. Multiplying both sides of this equation with $\mu(n)/n$, summing over $n\geq 1$ and using the defining recursion of the M\"obius function results in the statement.
\end{proof}

\subsection{Unrooted trees}
We next use  this generating function for rooted trees (weighted by the parity of internal vertex orbits)  to find the Frobenius characteristic ${\mathbf V}(\bb x)$ for the species $\mathcal T$ of unrooted admissible trees, weighted by the parity of internal {\em edge}-orbits.  Here $\mathcal T[n]$ consists of labeled trees with $n$ leaves, and to a pair $(t,\alpha)\in \mathcal{AT},$ where $t\in\mathcal T$ and $\alpha\in \Aut(t),$ we assign the weight $\omega(t,\alpha)=(-1)^{e_\alpha(t)}$, where $e_\alpha(t)$ is the number of internal edge-orbits of $\alpha$ in $t$, and define
\begin{align}\label{eq:defV} {\mathbf V}(\bb x) = \sum_{n \geq 3}\frac{1}{n!}\sum_{(t,\alpha) \in \mathcal{AT}[n]}(-1)^{e_\alpha(t)}x^\alpha , \end{align}

 \begin{proposition}
\label{prop:treegen}
\begin{align*} {\mathbf V}(\bb x) &=x_1+\frac{x_1^2}{2}-\frac{x_2}{2}-(1+x_1){\mathbf R}(\bb x). \end{align*}
\end{proposition}
\begin{proof}
Suppose $t$ is an unrooted tree and $\alpha\in\Aut(t)$. The set of points fixed by $\alpha$ is connected, so is either a subtree $t^\alpha$ or an isolated point in the middle of an internal edge.     

Suppose first that there is a fixed subtree $t^\alpha$.  For each internal vertex $v\in t^\alpha$, we can think of $(t,\alpha)$ as a collection $\Phi$ of at least $3$ rooted trees, with $v$ as their root, and of $\alpha$ as an automorphism of $\Phi$.    Note that  $e_\alpha(t)=v_\alpha(t)-1=v_\alpha(\Phi)$.
Using Proposition~\ref{prop:expformaut}  and setting $y_k=1$ for all $k$, the generating functions of such collections $\Phi$ is
\begin{align*} & \exp\left(\sum_{k \geq 1} \frac{{\mathbf R}(\bb x_{[k]})}{k} \right) -1 - {\mathbf R}(\bb x_{[1]}) - \frac{{\mathbf R}(\bb x_{[2]})}{2} - \frac{{\mathbf R}(\bb x_{[1]})^2}{2}, \end{align*}
where we have subtracted the terms that correspond to  collections with fewer than 3 trees. By Eq.~\eqref{eqn:Rx} in the proof of Proposition~\ref{prop:rootedtrees} this is the same as
\begin{align}\label{eqn:vertices} x_1-{\mathbf R}(\bb x)- \frac{{\mathbf R}(\bb x_{[2]})}{2} - \frac{{\mathbf R}(\bb x)^2}{2}. \end{align}

If we place a root in the middle of an internal edge of $t^\alpha$, we can view $t$ as a pair $\{\rho_1,\rho_2\}$ of rooted trees and $\alpha$ as a pair $\{\alpha_1,\alpha_2\}$ of automorphisms of $\rho_1$ and $\rho_2$ respectively.
Since $e$ is not adjacent to a leaf neither rooted tree is the special rooted tree.  The generating function of such rooted trees is therefore
\begin{align}\label{eqn:edges} \frac{({\mathbf R}(\bb x)-x_1)^2}{2}. \end{align}
Here we have $e_\alpha(t)=v_\alpha(t)-1=v_\alpha(\{\rho_1,\rho_2\})$.  The sum of expressions \eqref{eqn:vertices} and \eqref{eqn:edges} counts the pair $(t,\alpha)$ multiple times, once for each internal vertex of $t^\alpha$ with sign $(-1)^{e_\alpha(t)},$ and once for each internal edge of $t^\alpha$ with the opposite sign.  Since $t^\alpha$ has one more internal vertex than internal edge,  this sum leaves us with exactly one contribution from each $(t,\alpha)$, with sign $(-1)^{e_\alpha(t)}$, i.e. the sum gives the contribution to ${\mathbf V}(\bb x)$ from all such pairs.

We still have to account for pairs $(t,\alpha)$ with no fixed vertices, i.e.   the only fixed point is at the midpoint of an internal edge.   The tree $t$ can then  be viewed as the union of two identical rooted trees $\{\rho_1,\rho_2\}$ rooted at this midpoint, which are exchanged by $\alpha$.  By Lemma~\ref{lmm:kranz} these are counted by ${\mathbf R}(\bb x_{[2]})/2$.  Since in this case $e_\alpha(t)=v_{\alpha}(\{\rho_1,\rho_2\})$, these contribute
\begin{align}\label{eqn:midpoints} \frac{{\mathbf R}(\bb x_{[2]})}{2}- \frac{x_2 }{2} \end{align}
to ${\mathbf V}(\bb x)$, where we have subtracted the term corresponding to the interval since that is not an admissible unrooted tree.
The sum of formulas \eqref{eqn:vertices}, \eqref{eqn:edges} and \eqref{eqn:midpoints}  now has one term for each pair $(t,\alpha),$ weighted by $(-1)^{e_\alpha(t)}$, i.e.\ the sum is equal to ${\mathbf V}(\bb x)$.
 \end{proof}
 As an immediate corollary of Proposition~\ref{prop:treegen} and Proposition~\ref{prop:rootedtrees} we have
 \begin{corollary}\label{cor:treegen}
 \begin{align*} {\mathbf V}(\bb x)&=x_1+\frac{x_1^2}{2}-\frac{x_2}{2}-(1+x_1)\sum_{k \geq 1} \frac{\mu(k)}{k} \log(1+x_k). \end{align*}
 \end{corollary}
Using $\log(1+x) = \sum_{n \geq 1} (-1)^{n+1} x^n/n$,
and the definition of the M\"obius function, we can expand this power series. As $\mu(1) = 1$, $\mu(2) = \mu(3) = -1$ and $\mu(4) = 0$, the first coefficients are
\begin{align*} {\mathbf V}(\bb x) &= x_1 + \frac{x_1^2}{2} -\frac{x_2}{2} \\
&\phantom{=}-(1+x_1) \left( \left( x_1 - \frac{x_1^2}{2} + \frac{x_1^3}{3} - \frac{x_1^4}{4} + \ldots \right) -\frac{1}{2} \left( x_2 - \frac{x_2^2}{2} + \ldots \right) -\frac{1}{3} \left( x_3 + \ldots \right) + \ldots \right) \\
&= \frac{x_1^3}{6} + \frac{x_1 x_2}{2} + \frac{x_3}{3} -\frac{x_1^4}{12} -\frac{x_2^2}{4} +\frac{x_1 x_3}{3} +\ldots \end{align*}
where we omitted terms of total degree higher than $4$.
As expected there are no terms of degree smaller than $3$, as such terms would correspond to trees with fewer than $3$ leaves.

We recall from Section~\ref{sec:species}  that setting $\bb x = (x,0,0,\ldots)$ in the Frobenius characteristic $\mathbf V(\bb x)$ for unrooted trees recovers the generating function
$V(x)$ for unrooted trees $t$ weighted by $(-1)^{e(t)}$, where $e(t)$ is the number of (internal) edges of $t$, i.e.
\begin{equation}\label{eq:vx}
V(x)=x+\frac{x^2}{2}-(1+x)\log(1+x).
\end{equation}
For a detailed exposition of this formula  see \cite{BVquantum}.

As forests are just collections of trees, %
Proposition~\ref{prop:expformaut} gives us the following expression for the alternating cycle index series for the species $\mathcal F$ of forests:
\begin{align} \label{eq:Vset} {\mathbf V}_\Set(\bb x, \bb y) =\sum_n\frac{1}{n!}\sum_{(\Phi,\gamma)\in\mathcal{AF}[n]} (-1)^{e_\gamma(\Phi)} {x^\gamma}y^{\gamma_\Phi} =\exp \left( \sum_{k \geq 1} y_k \frac{ {\mathbf V}(\bb x_{[k]})}{k}\right). \end{align}

\subsection{Forested graphs}
What we want to do with these forests is to glue their leaves together in pairs to form admissible graphs (which we can only do if there is an even number of leaves).  We also want to keep track of the Euler characteristic of the resulting graph $G$; this is the number of trees in the forest minus half the total the number of leaves.  It is equivalent to keep track of $-2\chi(G)=\#\{\text{leaves}\} - 2 \#\{\text{trees}\}$ (which is always a positive integer)  so   we mark this number with a new variable $u$, and define the Frobenius characteristic of admissible forests
\begin{align} \begin{aligned} \label{eq:defF} {\mathbf F}(u, \bb x) &= \sum_{s\geq 0} \frac{1}{s!}\sum_{(\Phi,\gamma)\in\mathcal{AF}[s]} (-1)^{e_{\gamma}(\Phi)}x^\gamma u^{s-2k(\Phi)}, \end{aligned} \end{align}
where $\mathcal{F}$ is the species of forests and
$k(\Phi)$ is the number of trees in the forest $\Phi$.

\begin{proposition}
\label{prop:tree_expression}
\begin{align*} {\mathbf F}(u,\bb x) =\exp \left(\sum_{k \geq 1} u^{-2k} \frac{ {\mathbf V}( (u \cdot \bb x)_{[k]})}{k}\right), \end{align*}
where ${\mathbf V}((u \cdot \bb x)_{[k]})$ means that we replace each variable $x_i$ in ${\mathbf V}(\bb x)$ with $u^{ki} x_{ki}$.
\end{proposition}
By Eq.~\eqref{eq:defV} the lowest term of  ${\mathbf V}(u\cdot x)$ is of order $u^3$, and hence the lowest term of ${\mathbf V}((u \cdot \bb x)_{[k]})$   is of order $u^{3k}$.  Hence, ${\mathbf F}(u,\bb x)$  is the exponential of a power series in only positive powers of $u$.

\begin{proof}
Starting with the definition of ${\mathbf V}_\Set(\bb x, \bb y)$, the substitution $x_i\mapsto u^ix_i$ sends $x^\lambda$ to $u^{|\lambda|}x^\lambda$, and the substitution $y_i\mapsto u^{-2i}$ sends $y^{\gamma_\Phi}$ to $u^{2k(\Phi)}$ where $k(\Phi)$ is the number of trees in $\Phi$.   Doing both substitutions gives ${\mathbf F}(u,\bb x)$.
\end{proof}

\begin{proposition}\label{prop:etaF}
\begin{align*} \sum_\lambda \eta_\lambda[u^{2n} x^\lambda] {\mathbf F}(u,\bb x)=\ee_n,    \end{align*}
where we sum over all integer partitions $\lambda$  and  $\eta_\lambda$ is the number of   matchings   that commute with a permutation of cycle type $\lambda$.
The sum is finite since $[u^{2n} x^\lambda] {\mathbf F}(u,\bb x)=0$ unless
$2n \leq |\lambda| \leq 6n$.
\end{proposition}

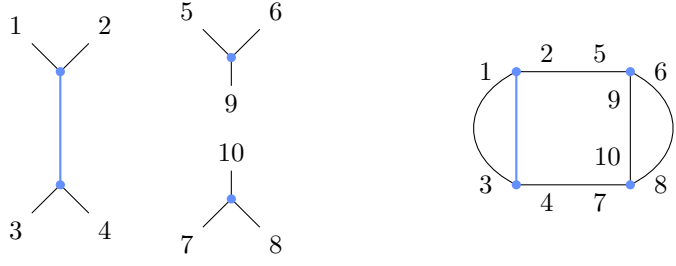
\begin{figure}
\begin{center}
\begin{tikzpicture}[scale=.75] \draw[thick,col1] (0,0) to (0,2); \draw (0,2) to (-.5,2.5);\node (v1) [above left] at (-.5,2.5) {$1$}; \draw (0,2) to (.5,2.5);\node (v2) [above right] at (.5,2.5) {$2$}; \draw (0,0) to (-.5,-.5);\node (v3) [below left] at (-.5,-.5) {$3$}; \draw (0,0) to (.5,-.5);\node (v4) [below right] at ( .5,-.5) {$4$}; \draw (3,2.25) to (2.5,2.75); \node (v5) [above left] at (2.5,2.75) {$5$}; \draw (3,2.25) to (3.5,2.75);\node (v6) [above right] at (3.5,2.75) {$6$}; \draw (3,-.25) to (2.5,-.75);\node (v7) [below left] at (2.5,-.75) {$7$}; \draw (3,-.25) to ( 3.5,-.75);\node (v8) [below right] at (3.5,-.75) {$8$}; \draw (3,2.25) to ( 3 ,1.75);\node (v9) [below] at (3,1.75) {$9$}; \draw (3,-.25) to ( 3,.25);\node (v10) [above] at (3, .25) {$10$}; \colonevertex{(0,0)};\colonevertex{(3,-.25)};\colonevertex{(0,2)};\colonevertex{(3,2.25)}; \draw (8,0) .. controls (7,.5) and (7,1.5) .. (8,2); \draw (10,0) .. controls (11,.5) and (11,1.5) .. (10,2); \draw[thick,col1] (8,0) to (8,2); \draw (10,0) to (10,2); \draw (8,0) to (10,0); \draw (8,2) to (10,2); \colonevertex{(8,0)};\colonevertex{(10,0)};\colonevertex{(8,2)};\colonevertex{(10,2)}; \node (w1) [left] at (7.75,2) {$1$};\node (w2) [above right] at (8.25,2) {$2$}; \node (w3) [left] at (7.75,0) {$3$};\node (v4) [below right] at (8.25,0) {$4$}; \node (w5) [above left] at (9.75,2) {$5$};\node (26) [right] at (10.25,2) {$6$}; \node (w7) [below left] at (9.75,0) {$7$};\node (28) [right] at (10.25,0) {$8$}; \node (w9) [left] at (10,1.5) {$9$};\node (210) [left] at (10,.5) {$10$}; \end{tikzpicture}
\end{center}
\caption{A forest  with automorphism $\gamma=(13)(24)(57)(68)(9\,10)$ and matching $\iota=(13)(25)(47)(68)(9\,10)$   that commutes with $\gamma$, corresponds to a  forested graph with the automorphism that flips the graph over the horizontal axis.}\label{fig:ForestedGraph}
\end{figure}

\begin{proof}
Our generating function ${\mathbf F}(u,\bb x)$ counts forests together with automorphisms, i.e.\ if we set
\begin{gather*} \mathcal F_{r,\lambda}= \{(\Phi,\gamma)|\Phi \hbox{ has } |\lambda| \hbox{ leaves and } k=\frac{|\lambda|-r}{2} \hbox{ components and } \gamma \hbox{ has cycle type } \lambda\}, \end{gather*}
 then
\begin{align*} [u^{r}x^\lambda]{\mathbf F}(u,\bb x) = \frac{1}{|\lambda|!}\sum_{(\Phi,\gamma)\in \mathcal F_{r,\lambda}}(-1)^{e_\gamma(\Phi)}. \end{align*}
 Because a nonempty forest has at least one component, we  have $[u^r x^\lambda ] {\mathbf F}(u, \bb x)=0$ if $|\lambda|<r$, and because every connected component of an admissible forest must have at least three leaves we have $[u^r x^\lambda ] {\mathbf F}(u, \bb x) = 0$ if $|\lambda| > 3r$.  In particular, for a given $r$ there are only finitely many integer partitions   $\lambda$ such that $[u^r x^\lambda ] {\mathbf F}(u, \bb x)$ is nonzero, i.e.\ the terms in the sum
$$\sum_\lambda \eta_\lambda[u^{r} x^\lambda] {\mathbf F}(\bb x, u)$$
are only nonzero when $r \leq |\lambda| \leq 3r$.

Let $M_r \subset \SG_r$ be the set of fixed-point free involutions on the set $\{1,\ldots,r\}$
and
$\mathcal{IF}_{r,\lambda}=\{(\Phi,\gamma,\iota)|(\Phi,\gamma)\in \mathcal F_{r,\lambda} \hbox{ and } \iota \in M_r \hbox{ commutes with }\gamma \}.$
Then $\eta_\lambda[u^{r} x^\lambda] {\mathbf F}(\bb x, u)$ is equal to $1/|\lambda|!$ times the  sum  of elements in $\mathcal{IF}_{r,\lambda}$ weighted by $(-1)^{e_\gamma(t)}$.
We can make a forested graph equipped with an automorphism from an element $(\Phi,\gamma,\iota)\in\mathcal{IF}_{r,\lambda}$ by using $\iota$ to identify the leaves of $\Phi$ in pairs, provided $r=2n$ is even (see Figure~\ref{fig:ForestedGraph}).   The result will be a forested graph $(G,\Phi_0)$ with $\chi(G)=-n$, where $\Phi_0$ is the subforest of $G$ consisting of just the internal edges and vertices of $\Phi$, and  the leaves of $\Phi$ have become (labeled) half-edges in $G\setminus\Phi_0$. This gives us a one-to-one correspondence between elements of $\mathcal{IF}_{2n,\lambda}$ and pairs $((G,\Phi_0),\gamma)$ consisting of a forested graph $(G,\Phi_0)$ with $\chi(G)=-n$ and an automorphism $\gamma\in\Aut(G,\Phi_0)$, where  the half-edges of $G\setminus \Phi_0$ are labeled by $\{1,\ldots,|\lambda|\}$. The symmetric group $\SG_{|\lambda|}$ acts on $\mathcal{IF}_{2n,\lambda}$ by permuting the labels. The stabilizer of $((G,\Phi_0),\gamma)$ is $\Aut_\gamma(G,\Phi_0)$, the automorphisms of $(G,\Phi_0)$ that commute with $\gamma$, and the orbit is the unlabeled pair $[G,\Phi_0]$ together with a conjugacy class   $[\gamma]$ of the automorphism group $\Aut(G,\Phi_0)$.  %
If we now take the sum over all integer partitions $\lambda$, the  orbit-stabilizer theorem gives
\begin{align*} \sum_\lambda \eta_\lambda[u^{2n} x^\lambda] {\mathbf F}(\bb x, u)&=\frac{1}{|\lambda|!}\sum_\lambda \sum_{(\Phi,\gamma,\iota)\in\mathcal{IF}_{2n,\lambda}} (-1)^{e_\gamma(\Phi)}\\
&=\sum_{\substack{[G,\Phi_0] \\ \chi(G)=-n}}\sum_{[\gamma]\in\Aut(G,\Phi_0)}\frac{1}{|\Aut_\gamma(G,\Phi_0)|} (-1)^{e_\gamma(\Phi_0)}\\
&=\sum_{\substack{[G,\Phi_0] \\ \chi(G)=-n}}\frac{1}{|\Aut(G,\Phi_0)|}\sum_{\gamma\in\Aut(G,\Phi_0)} (-1)^{e_\gamma(\Phi_0)}. \end{align*}
The second step uses the orbit stabilizer theorem again on the centralizer $\Aut_\gamma(G,\Phi_0)$.
Since the last line is the definition of $\ee_n$, the proposition is proved.
 \end{proof}

The following theorem summarizes the steps needed for computing $e(\Out(F_{n}))$ for a given $n$.

\begin{theorem}
\label{thm:compute}
For fixed $n \geq 2$, the numbers $e(\Out(F_{2})), \ldots, e(\Out(F_{n}))$ can be computed by the following steps:
\begin{enumerate}
\item Calculate ${\mathbf V}(\bb x)$ up to homogeneous degree $6(n-1)$ in $\bb x$ using the formula
in Corollary~\ref{cor:treegen}.
\item Calculate the coefficients $[u^{2k}x^\lambda]{\mathbf F}(u,\bb x)$ from ${\mathbf V}(\bb x)$ using the formula in
Proposition~\ref{prop:tree_expression} for all pairs $k,\lambda$ with $k \leq n-1$ and $\lambda$ a partition of size $\leq 6k$.
\item Calculate the numbers $\eta_\lambda$ using Corollary~\ref{cor:matchings} for all $\lambda$ of size $\leq 6(n-1)$.
\item Calculate the numbers $\ee_0,\ldots, \ee_{n-1}$ using the
formula
$$\ee_k=\sum_\lambda \eta_\lambda[u^{2k} x^\lambda] {\mathbf F}(u,\bb x)$$
from Proposition~\ref{prop:etaF}. Recall that this is a finite sum: the terms are nonzero only for partitions $\lambda$ of size $2k\leq |\lambda|\leq 6k$.
\item Recover $e(\Out(F_{n}))$ from $\ee_0, \ldots,\ee_{n-1}$ using the recursive formula in Corollary~\ref{cor:mM}.
\end{enumerate}
\end{theorem}

 The most demanding part of the computation is the expansion of the generating function ${\mathbf F}(u, \bb x)$ for the Frobenius characteristic of admissible forests. We used Jos Vermaseren's \texttt{FORM} programming language \cite{Vermaseren:2000nd} to compute this expansion up to an appropriate order and calculated $e(\Outn)$ for $n \leq 15$. The result is listed in Appendix~\ref{sec:table}. We report on an optimized program with which the numbers $e(\Outn)$ were computed for all $n \leq 100$ in a separate publication~\cite{BVer}. %

\section{Asymptotic behavior of \texorpdfstring{$e(\Outn)$}{eOutFn}}\label{sec:asymptotics}
\subsection{Overview of the proof of Theorem~\ref{thm:eOutFnAsy}}
In this section, we will prove Theorem~\ref{thm:eOutFnAsy}, which 
gives the asymptotic behavior of the numbers $e(\Outn)$.
The argument goes through the following steps:
In Section~\ref{sec:precise}, we will show that 
$\ee_n$ and $e(\Outn)$ have the same leading asymptotic behavior. 
Hence, it is sufficient to prove the asymptotic expression for $\ee_n$ that we state in Theorem~\ref{thm:asymptotics}.
The proof of Theorem~\ref{thm:asymptotics} relies partly on asymptotic results from   \cite{BV}.  In that paper we defined power series $T(\hbar)$ and $\widehat T(\hbar)$ with coefficients $\chi_n$ and $\Ch_n$ respectively by  
\begin{align} \label{eq:Thatdef} \widehat T(\hbar)= \sum_{n\geq 1} \Ch_n \hbar^n= \exp\left(\sum_{n\geq 1}\chi_n\hbar^n\right) = \exp(T(\hbar)), \end{align}
 where   $\chi_n=\chi(\Out(F_{n+1}))$ is the rational Euler characteristic  of $\Out(F_{n+1}).$ 
We recall the precise formula for the asymptotic behavior of $\Ch_n$ from ~\cite{BV} in Proposition~\ref{prop:Ch_asymp}. We next find  
a new expression for the numbers  $\ee_n$ 
that explicitly involves the generating function $\widehat T(\hbar)$. For this we need a slight extension of the combinatorial discussion from Section~\ref{sec:effective} that we   give in Section~\ref{sec:forestwithlegs}.
The new expression for $\ee_n$ is then derived in Section~\ref{sec:relating} and summarized in Theorem~\ref{thm:eesecondexpr}. 

Theorem~\ref{thm:eesecondexpr} expresses the number $\ee_n$ 
as a sum over  partitions. We   eventually show that the contribution 
to this sum from many of these partitions is negligible for large $n$, and that the remaining terms have the same asymptotic behavior as certain numbers $P_n$, which we define in 
Eq.~\eqref{eq:Pn}.   
We begin by proving some elementary estimates in Section~\ref{sec:split}.
In Section~\ref{sec:Pn}, we use these together with some further estimates to determine the asymptotic behavior of the numbers $P_n$
(see Proposition~\ref{prop:hatChi}).
It remains to prove that the numbers $P_n$ have the same asymptotic behavior 
as the numbers $\ee_n$ (stated precisely in Proposition~\ref{prop:enasy1}). This technical part of the argument also involves the estimates from Section~\ref{sec:split} as well as many further estimates.  The argument is split between Sections~\ref{sec:estimates_eneasy} and \ref{sec:proofeneasy} and finishes the proof of Theorem~\ref{thm:eOutFnAsy}.

\subsection{Precise asymptotic statements}
\label{sec:precise}

To discuss the asymptotic behavior of sequences we use the following standard conventions.

\begin{notation} Let $\{c_n\}$ be a sequence defined for all but finitely many  positive integers  $n$.
The set $\bigO(c_n)$ consists of all  such sequences $\{a_n\}$ for which
$ \limsup_{n\rightarrow \infty}\left| a_n/ c_n \right|< \infty. $
The notation $a_n = b_n +\bigO(c_n)$ means $a_n -b_n \in \bigO(c_n)$. %
Recall that the notation $a_n\sim b_n$ means
$ \lim_{n\rightarrow \infty}a_n/ b_n=1. $
\end{notation}

To describe the asymptotic behavior of the numbers $\ee_n$ we will use the $\Gamma$ function, which is defined by $\Gamma(x) = \int_{0}^\infty z^{x-1} e^{-z}\dd z$ for all $x > 0$. At integer arguments it agrees with the factorial $n!=\Gamma(n+1)$.
By Stirling's formula, we have
$\Gamma(x) \sim \sqrt{2 \pi} x^{x-\frac12} e^{-x}$
\cite[Eq.~(3.9)]{artin1964gamma}, i.e.~$\Gamma(x)$ grows more than exponentially.
Explicitly, we will quantify the asymptotic behavior of $\ee_n$ using
the  sequences $B_n$ and $L_n$ defined by
\begin{align*} B_n&=-\frac{1}{\sqrt{2\pi}}\frac{\Gamma\left(n-\frac12 \right)}{\log^2 n} &\text{ and }&& L_n&= \frac{\log n}{\log \log n}. \end{align*}
It follows from Stirling's formula that
\begin{lemma}
\label{lmm:stirling}
$B_n \sim - n^{n} e^{-n}/(n\log^2 n).$
\end{lemma}
\begin{proof}
$B_n/(-n^{n} e^{-n}/(n\log^2 n)) \sim e^{\frac12} (1-\frac{1}{2n})^{n-1} \sim 1$, where we used $\lim_{n \rightarrow \infty} (1+\frac{x}{n})^n = e^x$.
\end{proof}
The remainder of the paper is devoted to proving the following theorem.

\begin{theorem}\label{thm:asymptotics}  $\ee_n$ has asymptotic behavior
$$\ee_n=e^{-\frac14} B_n + \bigO(B_n/L_n).$$
 \end{theorem}
Note that $\lim_{n \rightarrow \infty} L_n = \infty$, so it follows that
$\ee_n \sim e^{-\frac14} B_n$. By applying Stirling's formula, i.e.~Lemma~\ref{lmm:stirling}, we moreover find that $\ee_n\sim -e^{-\frac14} n^ne^{-n}/(n\log^2 n)$. %

The following statement from previous work by the authors is needed to prove this theorem.
\begin{proposition}[Proposition~8.1 and Lemma~8.7 in \cite{BV}] %
\label{prop:Ch_asymp}
$\Ch_n$, defined in Eq.~\eqref{eq:Thatdef}, has the asymptotic behavior,
$$\Ch_n=B_n+\bigO(B_n/L_n).$$
\end{proposition}
\begin{proof}
We just have to substitute the asymptotic formula for the numbers $v_k$
in \cite[Lemma~8.7]{BV} (where we uncover an obvious typo: there should be no minus sign in \cite[equation (8.7)]{BV})
into the formula for $\Ch_n$ from \cite[Proposition~8.1]{BV}.
\end{proof}
As an immediate corollary we get
\begin{corollary}
\label{crll:chiestimate}
There exists a constant $C$, such that
\begin{align*} |\Ch_n| \leq C~ \Gamma\left(n-\frac12\right) \text{ for all $n \geq 1$.} \end{align*}
\end{corollary}
\begin{proof}
This follows from
$\bigO(\Ch_n) = \bigO(B_n) = \bigO\left( \Gamma(n-\frac12)/\log^2 n \right) \subset \bigO\left( \Gamma(n-\frac12) \right). $
Hence, $\lim_{n \rightarrow \infty} | \Ch_n / \Gamma(n-\frac12) |$ is finite and as $\Gamma(n-\frac12) > 0$ for $n\geq 1$, $\Ch_n / \Gamma(n-\frac12)$ stays bounded.
\end{proof}

We also showed that the numbers $\chi_n$ have the same asymptotic behavior as $\Ch_n$ \cite[Proposition~8.6]{BV}. A slight modification of the proof of this given in~\cite{BV} gives the following statement.

\begin{proposition}
\label{prop:eOutFnAsyB}
The numbers $e(\Out(F_{n+1}))$ have the same asymptotic behavior as $\ee_n$.
\end{proposition}

\begin{proof}
Lemma~8.8 of \cite{BV} gives a criterion for showing the asymptotic behavior of the coefficients of a series $\sum a_nx^n$ agrees with with that of the coefficients of $\exp(\sum a_nx^n)$. The proof of Proposition~8.6 in \cite{BV} immediately following Lemma ~8.8  applies almost verbatim to show that the coefficients $\eee_n$ defined by $\sum_{n=1}^\infty \eee_n \hbar^n = \log \left( \sum_{n= 0}^\infty \ee_n \hbar^{n} \right)$ have the same asymptotic behavior as $\ee_n$, i.e.~
$\eee_n =e^{-\frac14} B_n+\bigO(B_n/L_n).$
By Corollary~\ref{cor:mM} of the present paper, $e(\Out(F_{n+1})) = \eee_n + \sum_{\substack{ d \mid n, d\neq 1}} \frac{\mu(d)}{d} \eee_{n/d}$.
As $n$ has fewer than $n$ divisors and $|\mu(n)| \leq 1$, the sums $\sum_{\substack{ d \mid n, d\neq 1}} \frac{\mu(d)}{d} \eee_{n/d}$ form a sequence in $\bigO(n \eee_{n/2}) = \bigO(n B_{n/2}) \subset \bigO(n \Gamma(n/2-\frac12)) \subset \bigO(B_n/L_n),$ %
showing that $e(\Out(F_{n+1}))$ also has the same asymptotic behavior as $\ee_n$.
\end{proof}
\begin{proof}[Proof of Theorem~\ref{thm:eOutFnAsy}]
By  Theorem~\ref{thm:asymptotics} and  Proposition~\ref{prop:eOutFnAsyB} 
$$e(\Out(F_{n+1}))=e^{-\frac{1}{4}}B_n+\bigO(B_n/L_n).$$ The result follows by applying Lemma~\ref{lmm:stirling}  and from 
$\log (n-1) \sim \log n$ and $(1-1/n)^{n-1} \sim e^{-1}$.
\end{proof}

\begin{remark}
In fact, we prove a stronger statement than Theorem~\ref{thm:eOutFnAsy} %
which quantifies the error term in the asymptotic behavior.
By Propositions~\ref{prop:Ch_asymp} and~\ref{prop:eOutFnAsyB}, we have for large $n$
$$
e(\Outn)/\chi(\Outn)
=
{e^{-\frac14}}
 +
\bigO(
\log \log n / \log n).
$$
The error term $\bigO(\log \log n / \log n)$ above appears to be too pessimistic.
Based on the computations of $e(\Outn)$ up to $n= 100$ from \cite{BVer} and empirical comparison with $\chi(\Outn)$, we conjecture $$e(\Outn)/\chi(\Outn) = e^{-\frac{1}{4}} \left( 1 - \frac{29}{32 n} + \bigO\left(\frac{1}{n^2}\right) \right).$$
\end{remark}

\subsection{Forested graphs with legs} 
\label{sec:forestwithlegs}
In Section~\ref{sec:effective},
we have only needed to consider forested graphs $(G,\Phi)$  such that $G$ has no univalent vertices.  However, to determine the asymptotic behavior of the integral Euler characteristic we will need to do a finer analysis, which involves studying  pairs $(G,\Phi)$ where $G$ is allowed to have univalent (but not bivalent) vertices, while  $\Phi$ is not allowed to contain any edges adjacent to univalent vertices.  We call  the edges of $G$ adjacent to univalent vertices %
{\em legs} and the pairs $(G,\Phi)$ {\em forested graphs with legs}.  In this subsection we point out that a minor modification of the counting methods from Section~\ref{sec:effective} counts these more general graphs. %

Recall that admissible graphs are constructed by pairing the leaves of an admissible forest $\Phi$. The internal edges of the forest then become a subforest $\Phi_0$ of the admissible graph $G$.  To get the original forest back   you cut all the 1-cells of $G$ that are not in $\Phi_0$; the half-edges that result are attached to the leaves of the original forest $\Phi$.  If the graph $G$ has legs we are not allowing the subforest $\Phi$ to contain them so  they must be cut, which results in components with one 0-cell and one half-edge.  These are not admissible trees, so we will call these   {\em special components}, and mark them with a new variable $w$.  A forest which is allowed to have both admissible and special components will be called an {\em extended forest}. Matching the leaves of an extended forest results in a graph with a univalent vertex for each special component.

Let ${\mathcal F}^\star$ denote the species of extended forests, and define $F^\star(x,u,w)$ to be %
\begin{align*}  F^\star(x,u,w)  &= \sum_{\Phi \in {\mathcal F}^\star} (-1)^{e(\Phi)}u^{s(\Phi) - 2k(\Phi)} w^{j(\Phi)} \frac{x^{s(\Phi)}}{s(\Phi)!}, \end{align*}
where $e(\Phi)$ is the number of edges of $\Phi$, $s(\Phi)$ is the number of leaves,  $k(\Phi)$ is the total number of components  and $j(\Phi)$ is the number of special components, i.e.\ $[u^r x^sw^j]F^\star(u,x,w)$ is the edge-weighted count of forests with $s$ leaves and  $(s-r)/2$ components, of which $j$ are special.

Recall that $V(x)$ is the generating function for unrooted trees, weighted by the parity of their internal edges  (see  Eq.~(\ref{eq:vx})). %

\begin{proposition}\label{prop:F1} $F^\star(u,x,w)=\exp(u^{-1}wx + u^{-2}V(ux)).$
\end{proposition}
\begin{proof} A special component has $1$ leaf, $0$ edges and $1$ component (which is special!), so the term $u^{-1}wx$ accounts for special components. The term  $u^{-2}V(ux)$ accounts for admissible trees as before.  Exponentiating gives the generating function for collections, as per Lemma~\ref{lmm:expform}.
\end{proof}

 In \cite{BV} we showed that the rational Euler characteristic $\chi_n=\chi(\Out(F_{n+1}))$ is given by %
$$\chi_n=\sum_{
\substack{
[G,\Phi]\\
G \, \mathrm{connected}
\\
\chi(G) = -n
}
}\frac{(-1)^{e(\Phi)}}{ |\Aut(G,\Phi)|},$$
where the sum is over isomorphism classes $[G,\Phi]$ of connected forested graphs of Euler characteristic $-n$.
We defined the corresponding generating function,
 $T(\hbar)=\sum_{n\geq 1} \chi_n\hbar^n$,
 in Eq.~\eqref{eq:Thatdef}.

In \cite{BV}, we also defined an analogous generating function $T(\hbar, w)$ for isomorphism classes $[G,\Phi]$ of connected forested graphs with legs, which are weighted in the same way as in the formula above.  Here the variable $w$ marks the legs of $G$, i.e.\ $[\hbar^nw^s]T(\hbar,w)$ is the weighted sum over isomorphism classes  $[G,\Phi]$ such that $G$ is connected, has Euler characteristic $-n$, and has $s$ legs.  %
We can relate $T(\hbar,w)$ to the generating function $F^\star(u,x,w)$:
 \begin{lemma}\label{lem:That}
\begin{align*} \sum_{m \geq 0} (2m-1)!! [x^{2m}] F^\star(u,x,w) = \exp \left(\frac{u^{-2} w^2}{2} + T(u^2,w)\right). \end{align*}
\end{lemma}
\begin{proof} The coefficient of $u^{2n}w^{s}$ in $(2m-1)!! [x^{2m}] F^\star(u,x,w)$  gives the weighted count of forested graphs with $s$ legs and $\chi=-n$ that can be made by gluing the leaves of extended forests $\Phi$ with $2m$ leaves. Adding up over all $m$ gives the count of all forested graphs with $s$ legs  and $\chi=-n$.   The term %
$\exp(T(u^2,w))$  counts forests, but misses components with a single internal vertex and two legs. Those are taken care of by adding the term $ (u^{-2}w^2)/2$ to $T(u^2,w)$. %
\end{proof}
Finally, we recall from \cite{BV} that $T(\hbar,w)$ and $T(\hbar)$ are related in the following way.
\begin{proposition}[\cite{BV}, Proposition 3.1]
\label{prop:TandTT}
 $$ T(\hbar,w)= T(\hbar e^{-w})+ \frac{w}{2} +\hbar^{-1}\left(e^{w}-1-w-\frac{w^2}{2}\right) .$$
  \end{proposition}

Here  $T(\hbar e^{-w})$ accounts for connected forested graphs with negative Euler characteristic, the term $\frac{w}{2}$ accounts for those with  Euler characteristic zero, and the last term accounts for  trees.

\subsection{Relating the integral and rational Euler characteristics}\label{sec:relating}

Recall from Eq.~\eqref{eq:eendef} that each triple $[G,\Phi,\alpha]$ consisting of a forested graph with $\chi(G)=-n$ and an automorphism $\alpha$ preserving $\Phi$ contributes to $\ee_n$. Restricting the sum to only triples with $\alpha=\hbox{id}$ gives the exponentiated rational Euler characteristic $\widehat\chi_n$ from Eq.~\eqref{eq:Thatdef}.  To prove Theorem~\ref{thm:asymptotics} we will exploit the fact that we already know the asymptotics of $\widehat\chi_n$.  We will eventually find that the contributions to $\ee_n$ from triples $[G,\Phi,\alpha]$ such that $\alpha$ fixes $\Phi$ and has order at most $2$ dominate the asymptotics of $\ee_n$.   In this subsection  we manipulate our formula for $\ee_n$ to isolate these  contributions, and in the following subsections we will prove that they dominate by bounding the relative size of the remaining terms.

We start with the formula
\begin{align*} \ee_n= \sum_\lambda \eta_\lambda[u^{2n} x^\lambda] {\mathbf F}(u,\bb x) \end{align*}
from Proposition~\ref{prop:etaF}.  From Proposition~\ref{prop:tree_expression}  we have an expression for $ {\mathbf F}(u,\bb x)$,
so
\begin{align*} \ee_n &= \sum_\lambda \eta_\lambda[u^{2n} x^\lambda] \exp \left(\sum_{k \geq 1} u^{-2k} \frac{ {\mathbf V}( (u \cdot \bb x)_{[k]})}{k}\right). \end{align*}
Recall that a permutation is called a {\em derangement} if it has no fixed points; this is equivalent to saying that the corresponding cycle type has no parts of size 1, so we call it a {\em deranged partition}.
An integer partition $\lambda$ is equivalent to a pair $(m,\ld)$ where $m$ is the number of parts of size $1$ in $\lambda$ and
$\ld$ is the deranged  partition obtained from $\lambda$ by removing all parts of size $1$.

To make an admissible graph with an automorphism from a pair $(\Phi,\alpha)$  consisting of a forest and an automorphism,   leaves of $\Phi$ that are permuted by $\alpha$ in cycles of equal lengths must be paired with each other; in particular, leaves that are fixed by $\alpha$ must be paired with each other.  By first pairing the fixed leaves we can reduce the expression for $\ee_n$ above to a sum over
deranged partitions.  Specifically,  if there are $2m$ fixed leaves there are $(2m-1)!!$ ways to pair them, so if $\lambda=(2m,\delta)$ %
then $\eta_{\lambda} =(2m-1)!!\eta_{\ld}$ and
\begin{align*} \begin{aligned}  \ee_n    &= \sum_{\ld} \eta_{\ld} [u^{2n} x^{\ld}]\sum_{m = 0}^\infty(2m-1)!![x_1^{2m}] \exp \left(\sum_{k \geq 1} u^{-2k} \frac{ {\mathbf V}( (u \cdot \bb x)_{[k]})}{k}\right) \\
&= \sum_{\ld} \eta_{\ld} [u^{2n} x^{\ld}] \exp \left(\sum_{k \geq 2} u^{-2k} \frac{ {\mathbf V}( (u \cdot \bb x)_{[k]})}{k}\right) \sum_{m = 0}^\infty(2m-1)!![x_1^{2m}] \exp \left(u^{-2}{\mathbf V}(u\cdot \bb x)\right) , \end{aligned} \end{align*}
where we sum only over  deranged integer partitions $\ld$.  (Note that the variable $x_1$ does not appear in the power series ${\mathbf V}( (u \cdot \bb x)_{[k]})$ for $k \geq 2$.)  The  expression in the first exponential will not change in what follows, so we give it a name:
$$ \mathbf h_1(u,\bb x) =\sum_{k \geq 2} u^{-2k} \frac{ {\mathbf V}( (u \cdot \bb x)_{[k]})}{k}= \sum_{k\geq 2} \frac{u^kx^3_{k}}{6k} + \frac{u^kx_{k} {x_{2k}}}{2k}+\frac{u^kx_{3k}}{3k}-\frac{u^{2k}x_{k}^4 }{12k}+\ldots$$
and our expression reads
\begin{align} \begin{aligned} \label{eq:four} \ee_n = \sum_{\ld} \eta_{\ld} [u^{2n} x^{\ld}] \exp( \mathbf h_1(u,\bb x)) \sum_{m = 0}^\infty(2m-1)!![x_1^{2m}] \exp \left(u^{-2}{\mathbf V}(u\cdot \bb x)\right). \end{aligned} \end{align}

Next we look more closely at the term $\exp(u^{-2}\mathbf V(u\cdot\bb x)).$ %
We first  separate out the contribution to ${\mathbf V}(\bb x)$ of pairs $(t,\alpha)$ with $\alpha=\id$.  This is obtained by setting $x_i=0$ in ${\mathbf V}(\bb x)$ for all $i\geq 2$; recall that this gives us the generating function $V(x_1)$ for trees without automorphisms.
  By Corollary~\ref{cor:treegen} we have
\begin{align} \begin{aligned} \label{eq:Vdecomp} {\mathbf V}(\bb x) &=x_1 +\frac{x_1^2}{2}-\frac{x_2}{2}-(1+x_1){\mathbf R}(\bb x)\\
&= x_1 + \frac{x_1^2}{2} -\frac{x_2}{2}-(1+x_1)\sum_{k \geq 1} \frac{\mu(k)}{k} \log(1+x_k)\\
&= x_1 + \frac{x_1^2}{2} -(1+x_1)\log(1+x_1) -\frac{x_2}{2}-(1+x_1)\sum_{k \geq 2} \frac{\mu(k)}{k} \log(1+x_k)\\
&= V(x_1) -\frac{x_2}{2} + (1+x_1) {\mathbf W}(\bb x), \end{aligned} \end{align}
where  ${\mathbf W}(\bb x)=-\sum_{k \geq 2} \frac{\mu(k)}{k} \log(1+x_k)=-{\mathbf R}(\bb x)+(1+x_1)\log(1+x_1) $.  Note that ${\mathbf W}(\bb x)$ counts rooted trees with automorphisms that do not fix any leaves, in particular it does not involve the variable $x_1$.
By Eq.~\eqref{eq:Vdecomp} we have
\begin{align*} \exp \left(u^{-2}{\mathbf V}(u\cdot \bb x)\right) &= \exp \left(u^{-2} \left(V(u x_1) -u^2 \frac{x_2}{2} + (1+u x_1) {\mathbf W}(u \cdot \bb x)\right)\right)\\
&= \exp \left(\big(u^{-2} V(u x_1) + u^{-1} x_1 {\mathbf W}(u \cdot \bb x)\big)+u^{-2} \left({\mathbf W}(u \cdot \bb x)-u^2\frac{x_2}{2}\right) \right). \end{align*}
By Proposition~\ref{prop:F1},
$ \exp\left(u^{-2}V(u x)+u^{-1} w x\right)=F^\star(u,x,w), $
where $[u^{2n}x^sw^j]F^\star(u,x,w)$ counts forests with $j$ special components and $s$ leaves that glue up to graphs with $\chi=-n$ and $j$ legs.
Note that ${\mathbf W}(u \cdot \bb x)$ can be interpreted as a power series in $u$ whose coefficients are polynomials in $x_1,x_2,\ldots$.
This power series has no constant coefficient, so if $f(w)$ is another power series, then the  composition $f ( {\mathbf W}(u \cdot \bb x))$ is convergent in the usual power series topology. Substituting $x_1$ for $x$ and ${\mathbf W}(u\cdot\bb x)$ for $w$ in Proposition~\ref{prop:F1}, our formula for $\exp\left(u^{-2}{\mathbf V}(u\cdot \bb x\right))$ becomes
 \begin{gather*} \exp \left( u^{-2}{\mathbf V}(u\cdot \bb x)\right) = F^\star(u,x_1, {\mathbf W}(u \cdot \bb x) )\exp \left(u^{-2} \left( {\mathbf W}(u \cdot \bb x) -u^2 \frac{x_2}{2}\right) \right). \end{gather*}
Substituting the above into Eq.~(\ref{eq:four}) gives the following expression:
\begin{gather} \begin{gathered} \label{eq:five} \ee_n = \sum_{\ld} \eta_{\ld} [u^{2n} x^{\ld}] \exp(\mathbf h_1(u,\bb x))  \times \\
\sum_{m = 0}^\infty(2m-1)!![x_1^m]F^\star(u,x_1, {\mathbf W}(u \cdot \bb x) )\exp \left(u^{-2} \left( {\mathbf W}(u \cdot \bb x) -u^2 \frac{x_2}{2}\right)\right). \end{gathered} \end{gather}
Now recall the statement of Lemma~\ref{lem:That}
\begin{align*} \sum_{m \geq 0} (2m-1)!! [x^{2m}] F^\star(u,x,w) = \exp \left(\frac{u^{-2} w^2}{2} + T(u^2,w)\right). \end{align*}
After substituting $x_1$ for $x$ and ${\mathbf W}(u\cdot \bb x)$ for $w$ in this statement, Eq.~(\ref{eq:five}) becomes
\begin{gather} \begin{gathered}\label{eq:twentyone}       \ee_n=\sum_{\ld} \eta_{\ld} [u^{2n} x^{\ld}] \exp(\mathbf h_1(u,\bb x)) \times \\
\exp \left( \frac{u^{-2} {\mathbf W}(u\cdot\bb x)^2}{2} + T(u^2,{\mathbf W}(u\cdot\bb x))+u^{-2} \left( {\mathbf W}(u \cdot \bb x) -u^2 \frac{x_2}{2}\right)\right) \end{gathered} \end{gather}
By Proposition~\ref{prop:TandTT} we have a relation between $T(\hbar,w)$ and $T(\hbar)$.
Substituting $u^2$ for $\hbar$ and ${\mathbf W}(u\cdot\bb x)$ for $w$, this relation becomes
$$ T(u^2, {\mathbf W}(u\cdot\bb x))=T(u^2 e^{-{\mathbf W}(u\cdot\bb x)})+ \frac{{\mathbf W}(u\cdot\bb x)}{2} +u^{-2}\left(e^{{\mathbf W}(u\cdot\bb x)}-1-{\mathbf W}(u\cdot\bb x)-\frac{{\mathbf W}(u\cdot\bb x)^2}{2}\right).$$
Substituting this into Eq.~\eqref{eq:twentyone} and simplifying
 turns our expression for $\ee_n$ into
\begin{gather} \begin{gathered}\label{eq:twentythree} \sum_{\ld} \eta_{\ld} [u^{2n} x^{\ld}]\exp\left(\mathbf h_1(u,\bb x) + T\left(u^2 e^{-{\mathbf W}(u\cdot {\bar x)}}\right)+ \frac{{\mathbf W}(u\cdot\bb x)}{2} + u^{-2}\left(e^{{\mathbf W}(u\cdot\bar x)}-1- u^2 \frac{x_2}{2} \right)\right). \end{gathered} \end{gather}
Recall from Eq.~\eqref{eq:Thatdef} that $ \TT(\hbar)=\exp(T(\hbar))$ and set
 \begin{align*} \mathbf h_2(u,\bb x)&= \frac{{\mathbf W}(u\cdot\overline x)}{2}\\
 \mathbf h_3(u,\bb x)&=u^{-2}\left(e^{{\mathbf W}(u\cdot\bar x)}-1-u^2\frac{x_2}{2}\right).   \end{align*}

We record Equation~\eqref{eq:twentythree} formally as a theorem:
\begin{theorem}\label{thm:eesecondexpr}
    $$\ee_n=\sum_{\ld} \eta_{\ld} [u^{2n}x^\ld] \widehat T(u^2 e^{-{\mathbf W}(u\cdot \bar x)}){\mathbf H}(u,\bb x),$$
    where the sum is over all deranged integer partitions $\delta$ and $$\mathbf H(u,\bb x)=\exp\left(\mathbf h_1(u,\bb x) + \mathbf h_2(u,\bb x)+\mathbf h_3(u,\bb x)\right).$$
\end{theorem}

\begin{figure}
\begin{tikzpicture}[scale=.5] \draw (1,0) to (1,2) to (0,4) to (3,2) to (1,0); \draw (1,2) to (3,2); \draw (-1,0) to (-1,2) to (0,4) to(-3,2) to (-1,0) to (1,0);\draw (-1,2) to (-3,2); \node (v) [left] at (-3,2) {$v$}; \node (w) [right] at (-1,2) {$w$}; \colonevertex{(1,0)};\colonevertex{(-1,0)};\colonevertex{(1,2)};\colonevertex{(0,4)};\colonevertex{(3,2)}; \colonevertex{(-1,2)};\colonevertex{(-3,2)}; \draw[thick,col1] (-3,2) to (0,4) to (-1,2); \draw[thick,col1] ( 3,2) to ( 1,2); \node (a) at (0,-.75) {$(G,\Phi,\alpha)$}; \begin{scope}[xshift=8cm ] \draw (1,0) to (1,2) to (0,4) to(3,2) to (1,0)to (-1,0); \draw (1,2) to (3,2); \draw (-1,2.5) to (0,4) to(-3,2.5); \draw (-.5,1.6) to (-1,2.5) to (-1.5,1.6); \draw (-2.5,1.6) to (-3,2.5) to (-3.5,1.6); \draw (-.5,.35) to (-1,0) to (-2.5,.35); \node (1) at (-.5,1.25) {$1$}; \node (2) at (-1.5,1.25){$2$}; \node (3) at (-2.5,1.25) {$3$}; \node (4) at (-3.5,1.25) {$4$}; \node (5) at (-.4,.6) {$5$}; \node (6) at (-2.5,.6) {$6$}; \node (b) at (0,-.75) {$G'$}; \colonevertex{(1,0)};\colonevertex{(-1,0)};\colonevertex{(1,2)};\colonevertex{(0,4)};\colonevertex{(3,2)}; \colonevertex{(-1,2.5)};\colonevertex{(-3,2.5)}; \draw[thick,col1] (-3,2.5) to (0,4) to (-1,2.5); \draw[thick,col1] ( 3,2) to ( 1,2); \end{scope} \begin{scope}[xshift=14cm ] \draw (1,0) to (1,2) to (0,4) to(3,2) to (1,0)to (-1,0); \draw (1,2) to (3,2); \draw [thick, col3] (-1,3) to (0,4); \draw[thick,col3] (-1,1) to (-1,0); \node (b) at (0,-.75) {$G{''}$}; \colonevertex{(1,0)};\colonevertex{(-1,0)};\colonevertex{(1,2)};\colonevertex{(0,4)};\colonevertex{(3,2)};\draw[thick,col1] ( 3,2) to ( 1,2); \end{scope} \begin{scope}[xshift=20cm ] \draw (1,0) to (1,2) to (0,4) to(3,2) to (1,0); \draw (1,2) to (3,2); \draw[thick,col1] ( 3,2) to ( 1,2); \draw [thick, col3] (-1,3) to (0,4); \draw[thick,col3] (0,1) to (1,0); \node (b) at (0,-.75) {$(\overline G,\overline\Phi)$}; \colonevertex{(1,0)};\colonevertex{(1,2)};\colonevertex{(0,4)};\colonevertex{(3,2)}; \end{scope} \end{tikzpicture}
\caption{Reducing a forested graph $(G,\Phi)$ with automorphism $\alpha$ interchanging $v$ and $w$ to a forested graph with legs. The graph $G'$ is obtained by cutting edges not in $G^\alpha$ or $\Phi$, and  $\alpha$ induces the derangement $(14)(23)(56)$. The next graph  $G^{\prime\prime}$ replaces each set of trees at a vertex of $G^\alpha$ by a single orange leg, and the final graph $\overline G$ results from contracting all separating edges in $G^{\prime\prime}$  that are not legs.}\label{fig:Gprime}
\end{figure}
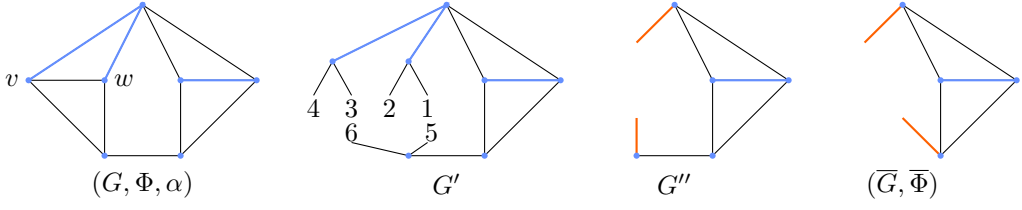

Note that $\mu(2) = \mu(3) = -1$ and $\mu(4) = 0$, hence the first few terms of ${\mathbf W}(u\cdot\bar x)$ are
\begin{align*} {\mathbf W}(u\cdot\bar x) = - \sum_{k \geq 2} \frac{\mu(k)}{k} \log(1+u^k x_k) &= \frac{1}{2} \left(u^2 x_2 - \frac{u^4 x_2^2}{2} \right) + \frac{ u^3 x_3}{3} + \ldots \\
&= \frac{u^2 x_2}{2} + \frac{u^3 x_3}{3} - \frac{u^4 x_2^2 }{4} + \ldots \end{align*}
where higher powers of $u$ were omitted.
It follows that
$ \mathbf h_2(u,\bb x)={{\mathbf W}(u\cdot\bar x)/2} $
 and
\[\mathbf h_3(u,\bb x)=u^{-2}\left(e^{{\mathbf W}(u\cdot\bar x)}-1-u^2\frac{x_2}{2}\right)
=
u
\frac{x_3}{3} - u^2 \frac{x_2^2}{8} + \ldots
\]
 are power series in only positive powers of $u$ (as is ${\mathbf h}_1(u,\bb x)$).

\begin{remark}
Here is a combinatorial interpretation of Theorem~\ref{thm:eesecondexpr}.  Given a triple $(G,\Phi,\alpha)$, let    $G^\alpha$ be the  subgraph of $G$ fixed by $\alpha$, and $\Phi^\alpha=\Phi\cap G^\alpha$.  If we cut all edges of $G$ which are not in $\Phi$ or in $G^\alpha$, we obtain a  graph $G'$ with leaves, and $\alpha$ induces a derangement of these leaves (see Figure~\ref{fig:Gprime}).  The fixed subgraph $G^\alpha$ is a subgraph of $G'$.  Components of $G'$ that do not intersect $G^\alpha$  are $k$-cycles of trees with $k\geq 2$;  these contribute  the  term $\mathbf h_1$ in   Eq.~\eqref{eq:twentythree}.  If $C$ is a component that  does intersect $G^\alpha$ then $(C\cap G^\alpha, C\cap \Phi^\alpha)$ is a forested graph, and the rest of $C$ consists of  $k$-cycles of rooted trees attached at various vertices of $C\cap G^\alpha$, where $k\geq 2$. At each of these vertices, remove all of the deranged trees that are attached there and   replace them with a single orange leg.  Then  contract all separating edges of the result that are not  orange legs to get an admissible forested graph $(\overline C,\overline\Phi)$ with legs. We can count such forested graphs using the generating functions
$T(\hbar e^{-w})$  (if $\chi(\overline C) <0)$,  $ \frac{w}{2}$ (if $\chi(\overline C)=0)$, and $\hbar^{-1}w -\frac{x_2}{2}$ (if $\overline C$ is a tree  but  $C$ is not a single vertex with two half-edges attached).   Replacing $\hbar$ by $u^2$ and $w$ by ${\mathbf W}(u\cdot\bb x)$ has the effect of marking the negative Euler characteristic by $u^2$ instead of $\hbar$ and reconstructing the components $C$ by adding rooted trees to the forested graphs.
\end{remark}

To determine the asymptotic behavior of $\ee_n$ we will estimate the size of the coefficients of $\mathbf H(u,\bb x)$, and show that their contribution to $\ee_n$ is dominated asymptotically by the contribution of $\TT(u^2 e^{-{\mathbf W}(u\cdot \bar x)})$.  We will also show that the sum $\sum_{\ld} \eta_{\ld} [u^{2n}x^\ld] \TT(u^2 e^{-{\mathbf W}(u\cdot \bar x)})$ is dominated by contributions of deranged partitions $\delta$ with all parts of size $2$.  We will then be able to determine the asymptotic behavior of $\ee_n$ using estimates on the size of $\eta_\lambda$ and the fact that we know the behavior of the coefficients $\Ch_n$ of $\TT(\hbar)$ from our previous work in \cite{BV}.

\subsection{Splitting and merging \texorpdfstring{$\Gamma$}{Gamma} functions and estimating the numbers \texorpdfstring{$\eta_\lambda$}{eta\_lambda}}
\label{sec:split}

In this subsection we show that the numbers $\eta_\lambda$ and  $\eta_{k,m}$ defined in Corollary~\ref{cor:matchings} are bounded by $\Gamma$ functions modulated by exponentials.  Recall that the $\Gamma$ function satisfies %
\begin{itemize}
\item $\Gamma(x+1) = x \Gamma(x)$ for all $x > 0$,
\item  $\Gamma(1/2) = \sqrt{\pi}$, and
\item $\Gamma$ is  is \emph{log-convex}, i.e.
\begin{gather} \label{eq:logconvex} \log \Gamma\left(t a+\left(1-t\right)b+\frac12\right) \leq t \log \Gamma\left( a+\frac12\right) + \left(1-t\right)\log \Gamma\left( b+\frac12\right), \end{gather}
 for all  $t \in [0,1]$ and $a,b \geq 0$.
 \end{itemize}
In fact, these three properties determine the function $\Gamma$ completely
\cite[Theorem 2.1]{artin1964gamma}. Lemma~\ref{lmm:gamma_shift}, Corollary~\ref{cor:gamma_merge} and Lemma~\ref{lmm:gamma_split} below  follow easily using these properties.

\begin{lemma}
\label{lmm:gamma_shift}
For all $x,y,z \in \R$ with $0 \leq z \leq y \leq x$ we have
\begin{align*}  \Gamma\left(x+\frac12\right) \Gamma\left(y+\frac12\right) &\leq \Gamma\left(x+z+\frac12\right) \Gamma\left(y-z+\frac12\right).  \end{align*}
\end{lemma}
\begin{proof}
Adding  Eq.~\eqref{eq:logconvex} to itself with $a$ and $b$ interchanged gives,
\begin{gather*} \log \Gamma\left(t a+\left(1-t\right)b+\frac12\right) + \log \Gamma\left(\left(1-t\right) a+tb+\frac12\right) \leq \log \Gamma\left( a+\frac12\right) + \log \Gamma\left( b+\frac12\right). \end{gather*}
We can choose $a=x+z$, $b=y-z$ and $t=1-z/(x-y+2z)$, as $t \in [0,1]$ and $a,b \geq 0$ are fulfilled because $z\leq y \leq x$. Exponentiating gives the stated inequality.
\end{proof}

\begin{corollary}
\label{cor:gamma_merge}
For all $x,y \in \R$   with  $x,y \geq 0$ we have
\begin{gather} \frac{\Gamma\left(x+\frac12\right)}{\sqrt{\pi}} \frac{\Gamma\left(y+\frac12\right)}{\sqrt{\pi}} \leq \frac{\Gamma\left(x+y+\frac12\right)}{\sqrt{\pi}} . \end{gather}
\end{corollary}
\begin{proof}
Specializing Lemma~\ref{lmm:gamma_shift} with $y=z$ and using $\Gamma(1/2) = \sqrt{\pi}$ gives the bound.
\end{proof}

\begin{lemma}
\label{lmm:gamma_split}
There is a constant $C$ such that for all $x,y \in \R$   with  $x,y \geq 0$ we have
\begin{gather} \Gamma\left(x+y+\frac12\right) \leq C^{1+x+y} \Gamma\left(x+\frac12\right) \Gamma\left(y+\frac12\right). \end{gather}
\end{lemma}
\begin{proof}
Specializing Lemma~\ref{lmm:gamma_shift} to $x = y = (x'+y')/2$ and $z = |x'-y'|/2$ gives,
\begin{align*} \Gamma\left(\frac{x'+y'+1}{2}\right)^2 &\leq \Gamma\left(x'+\frac12\right) \Gamma\left(y'+\frac12\right) \quad \text{ for all } x',y' \in \R \text{ with } x',y' \geq 0. \end{align*}
The duplication formula for the $\Gamma$ function \cite[Eq.~(3.11)]{artin1964gamma} can be written as
\begin{gather} \label{eq:duplication} \Gamma\left(\frac{z+1}{2}\right) = 2^{-z} \sqrt{\pi} \Gamma(z+1)/\Gamma\left(\frac{z}{2}+1\right). \end{gather}
Applying this identity once to the left hand side of the inequality above results in
\begin{align*} \Gamma\left(x+y+\frac12\right) &\leq \frac{ 2^{x+y}}{\sqrt{\pi}} F(x+y) \Gamma\left(x+\frac12\right) \Gamma\left(y+\frac12\right) \quad \text{ for all } x,y \in \R \text{ with } x,y \geq 0. \end{align*}
where
$F(z) = \Gamma\left(\frac{z}{2}+1\right) \Gamma\left(z+\frac12\right) /( \Gamma\left(\frac{z+1}{2}\right) \Gamma(z+1) )$.
To prove the upper bound in the statement it is therefore sufficient to show that there is a constant $C$ such that
$F(z) \leq C$ for all $z \geq 0$.
$F(z)$ is regular for all $z \geq 0$, so it is sufficient to prove the existence of such a constant for $z \geq 1$. In this case, we have $F(z) \leq 1$, because $\frac{z+1}{2}\leq z$ and can apply Lemma~\ref{lmm:gamma_shift} to establish
\begin{gather*} \Gamma\left(\frac{z}{2}+1\right) \Gamma\left(z+\frac12\right) \leq \Gamma\left(\frac{z+1}{2}\right) \Gamma(z+1) \quad \text{ for all } z\geq 1. \qedhere \end{gather*}
\end{proof}

We now apply these lemmas to bound the numbers $\eta_{\lambda}$ defined in Corollary~\ref{cor:matchings}.

\begin{lemma}
\label{lmm:eta_bound}
There exists a constant $C$ such that
\begin{align*} \eta_{k,m} &\leq \frac{(kC)^{m/2}}{\sqrt{\pi}} \Gamma\left( \frac{m+1}{2} \right) \text{ for all } k \geq 1 \text{ and } m \geq 0. \end{align*}
\end{lemma}
\begin{proof}
Because $\eta_{k,0} =1$ and $\Gamma(1/2) = \sqrt{\pi}$, the statement is obvious if $m=0$. In all other cases it is sufficient to prove the bound for $m$ large.
By the standard identity
$(2n-1)!! = 2^n \Gamma(n+\frac12)/\sqrt{\pi}$
and by the definition of $\eta_{k,m}$ the statement is true
for all odd $k$.
For even $k$,
\begin{gather*} \eta_{k,m} = \sum_{r=0}^{\lfloor m/2 \rfloor} \binom{m}{2r} k^r (2r-1)!! = \sum_{r=0}^{\lfloor m/2 \rfloor} \binom{m}{2r} k^r 2^r \frac{\Gamma(r+\frac12)}{\sqrt{\pi}} \leq 2^m \max_{0 \leq r \leq \lfloor m/2 \rfloor} (2k)^r \frac{\Gamma(r+\frac12)}{\sqrt{\pi}}, \end{gather*}
where we used the floor function $\lfloor \cdot \rfloor$ and $\sum_{r=0}^{\lfloor m/2 \rfloor} \binom{m}{2r} \leq 2^m$.
The statement follows since $\Gamma(x)$ is increasing for sufficiently large $x$.
\end{proof}

Recall that $\ell(\lambda)$ denotes the length of $\lambda$, i.e.\ the number of parts of the partition.

\begin{corollary}
\label{cor:etalambdaestimate}
There is a constant $C$ such that for all integer partitions $\lambda$,
\begin{align*} \eta_{\lambda} \leq \frac{C^{|\lambda|}}{\sqrt{\pi}} \Gamma\left( \frac{\ell(\lambda)+1}{2} \right). \end{align*}
\end{corollary}
\begin{proof}
Let $\lambda = [1^{m_1}2^{m_2} \cdots ]$ and note that $m_k = 0$ for all $k > |\lambda|$. Recall that $\log k \leq k$ for all $k \geq 1$.
Therefore,
$$\prod_{k=1}^{|\lambda|}k^{m_k/2}
= \exp\left(\sum_{k =1}^{|\lambda|} m_k/2\log k  \right)
\leq
\exp\left(\sum_{k =1}^{|\lambda|} k m_k/2  \right)
=
e^{|\lambda|/2}.
$$
By Lemma~\ref{lmm:eta_bound}
there is a constant $C$ such that
\begin{align*} \eta_{\lambda} = \prod_{k = 1}^{|\lambda|} \eta_{k,m_k} \leq \prod_{k=1}^{|\lambda|} \frac{(kC)^{m_k/2}}{\sqrt{\pi}} \Gamma\left( \frac{m_k+1}{2} \right). \end{align*}
The proof is completed by using  the first inequality, the fact that $\sum_{k \geq 1} m_k = \ell(\lambda) \leq |\lambda|$ and
the bound from Corollary~\ref{cor:gamma_merge}.
\end{proof}
Finally, the last lemma in this subsection shows how the constant $e^{-\frac{1}{4}}$ arises:
\begin{lemma}
\label{lmm:e14}
Recall the definition of $L_n$ in Section~\ref{sec:precise}.
For large $n$,
\begin{align*} \sum_{m =0}^n\frac{(-1)^m }{ 2^m m!}\eta_{2,m}=e^{-\frac14}+\bigO\left(\frac{1}{L_n}\right). \end{align*}
\end{lemma}

In fact this sum converges much faster than the indicated error term, but this rough error estimate is sufficient for our purpose.
\begin{proof}
By  Corollary~\ref{cor:matchings} we have
$ \eta_{2,m}=\sum_{r=0}^{\lfloor m/2\rfloor}\binom{m}{2r}2^r(2r-1)!!. $
Using this, we find that
\begin{gather*} \sum_{m =0}^\infty\frac{(-1)^m }{ 2^m m!}\eta_{2,m} = \sum_{m =0}^\infty\frac{(-1)^m }{ 2^m m!}m! \sum_{r=0}^{\lfloor m/2 \rfloor} \frac{2^r}{2^r r! ( m-2r)!} = \sum_{r=0}^{\infty}\frac{1}{r!}\sum_{m =2r}^\infty\frac{(-1)^m }{ 2^m ( m-2r)!} \\
=\sum_{r=0}^{\infty}\frac{(-1)^{2r}}{4^r r!}\sum_{m =0}^\infty\frac{(-1)^{m} }{ 2^{m} m!} =e^{\frac14}e^{-\frac12}= e^{-\frac14}. \end{gather*}
By Lemma~\ref{lmm:eta_bound} we find a constant $C$ such that the tail of this series is bounded as follows:
\begin{gather*} \left|\sum_{m =n+1}^\infty\frac{(-1)^m }{ 2^m m!}\eta_{2,m} \right| \leq \sum_{m =n+1}^\infty C^m\frac{\Gamma(\frac{m+1}{2})}{\sqrt{\pi}m!} = C^{n+1}\sum_{m =0}^\infty C^m\frac{\Gamma(\frac{m+n+2}{2})}{\sqrt{\pi}(m+n+1)!} \\
\leq C^{n+1} {C'}^{1+\frac{n+1}{2}}\frac{\Gamma(\frac{n+2}{2})}{\sqrt{\pi} n!}\sum_{m =0}^\infty C^m{C'}^{\frac{m}{2}}\frac{\Gamma(\frac{m+1}{2})}{\sqrt{\pi}m!} \in \bigO\left(C^{n} {C'}^{\frac{n}{2}}\frac{\Gamma(\frac{n+1}{2})}{n!}\right) \subset\bigO\left(\frac{1}{L_n} \right), \end{gather*}
where we used a constant $C'$ from Lemma~\ref{lmm:gamma_split} to split the $\Gamma$ function in the numerator and Corollary~\ref{cor:gamma_merge} to split the factorial function in the denominator. The convergence of the infinite sum and the last inclusion of sets follows from Stirling's approximation of the $\Gamma$ function.
\end{proof}

\subsection{A new sequence of numbers.}%
\label{sec:Pn}
Let
 \begin{gather} \label{eq:Pn} P_n=\sum_{m =0}^n\widehat{\chi}_{n-m}\binom{-\frac12 (n-m)}{m}\eta_{2,m}, \end{gather}
where the binomial coefficient $\binom{q}{k}$ is defined by $\frac{q(q-1)\cdots (q-k+1)}{k!}$ for integers $k \geq 0$ and all $q \in \R$.
In this subsection we will use the fact that we know the asymptotic behavior of the numbers $\widehat\chi_n$  (Proposition~\ref{prop:Ch_asymp}) together with the estimates from the previous subsection to determine the asymptotic behavior of the numbers $P_n$.
In the   two subsections after this one we will prove
 \begin{proposition}
 \label{prop:enasy1}
 $$\ee_n=P_n+\bigO \left(\Gamma\left(n - \frac{7}{12}\right)\right).$$
 \end{proposition}

At the end of this subsection we  observe that Proposition~\ref{prop:enasy1} together with the asymptotics of $P_n$ imply the main asymptotic result, Theorem~\ref{thm:asymptotics}.

\begin{remark}
Proposition~\ref{prop:enasy1} has the following combinatorial interpretation.
 Recall that the numbers $\ee_n$ are defined (in Eq.~\eqref{eq:eendef}) as an alternating sum over all admissible forested graphs and over all their automorphisms.
Proposition~\ref{prop:enasy1} gives an asymptotic formula for this sum
in terms of the numbers $\Ch_n$. The numbers  $\Ch_n$ take a sum over forested graphs without summing also over automorphisms.
The $m=0$ term in Eq.~\eqref{eq:Pn} is equal to $\Ch_n$, which means it accounts for all the summands in Eq.~\eqref{eq:eendef} where $\alpha$ is the trivial automorphism.
The combinatorial interpretation of Theorem~\ref{thm:eesecondexpr} (see also Lemma~\ref{lmm:Tcoeffs} in the next subsection),
implies that we can similarly interpret the term
$\widehat{\chi}_{n-m}\binom{-\frac12 (n-m)}{m}\eta_{2,m}$
in Eq.~\eqref{eq:Pn}
as accounting for all forested graphs with an automorphism which is generated by $m$ cycles of order $2$ and where each cycle acts on half-edges of the graph that are attached to a component of the fixed point set with positive Euler characteristic.
The most important consequence of Proposition~\ref{prop:enasy1} lies in the error term $\bigO \left(\Gamma\left(n - {7/12}\right)\right)$. It tells us that summands in Eq.~\eqref{eq:eendef} that involve more complicated automorphisms (e.g.~automorphisms of order $\geq 3$) are asymp\-tot\-ical\-ly negligible in comparison to $\Ch_n$, which grows in magnitude like $\Gamma\left(n - {1/2}\right)/\log^2 n$ (Proposition~\ref{prop:Ch_asymp}) and therefore dominates all 
sequences in $\bigO \left(\Gamma\left(n - {7/12}\right)\right) = \bigO \left(n^{-1/12}\, \Gamma\left(n - 1/2\right)\right)$ 
for large $n$.
\end{remark}

 In order to determine the asymptotic behavior of the $P_n$ we need a few more estimates.

\begin{lemma}
\label{lmm:prod_bound}
There exists a constant $C$ such that
\begin{gather*} 1 \leq \prod_{r = 0}^{m-1} \frac{n-m +2r}{n-m -\frac12 +r} \leq\exp\left(C\frac{ m(m +1)}{n}\right) \end{gather*}
for all integers $n,m$ with $0 \leq m \leq n-1$ and $n \geq 1$.
\end{lemma}

\begin{proof}
The lower bound is obvious as $n-m +2r \geq n-m -\frac12 +r$.
We start by proving the bound for all $m \geq \frac{n-1}{2}$.
We have
\begin{gather*} \prod_{r = 0}^{m-1} \frac{n-m +2r}{n-m -\frac12 +r} = 2^{m} \prod_{r = 0}^{m-1} \frac{\frac{n-m}{2} +r}{n-m -\frac12 +r} \leq 2^{m} = 2^{m \frac{n}{n}} \leq 2^{\frac{m (2m + 1 )}{n}} \leq 4^{\frac{m(m + 1)}{n}}, \end{gather*}
as $n-m \geq 1 \Rightarrow \frac{n-m}{2} +r \leq n-m -\frac12 +r$ and $2m +1 \geq n$.

We still have to prove the bound for $m \leq \frac{n-1}{2}$.
Recall that $1+x \leq e^x$ for all $x \in \R$. It follows that $\frac{1}{1-x} \leq e^{\frac{x}{1-x}}$ for $x < 1$. Therefore,
\begin{gather*} \prod_{r = 0}^{m-1} \frac{n-m +2r}{n-m -\frac12 +r} = \prod_{r = 0}^{m-1} \frac{1 +\frac{2r-m}{n} }{1- \frac{ m +\frac12 -r}{n} } \leq \prod_{r = 0}^{m-1} \exp\left( \frac{2r-m}{n} + \frac{\frac{ m +\frac12 -r}{n}}{1- \frac{ m +\frac12 -r}{n}} \right) \\
\leq \prod_{r = 0}^{m-1} \exp\left( \frac{2r-m}{n} + 2\frac{ m +\frac12 -r}{n} \right) = \exp\left( \frac{m(m+1)}{n} \right), \end{gather*}
where we used
$m \leq \frac{n-1}{2} \Rightarrow {1/\left(1- \frac{ m +\frac12 -r}{n}\right)} \leq {1/\left(\frac12 + \frac{r}{n}\right)} \leq 2$.
\end{proof}

\begin{lemma}
\label{lmm:log_bound}
There exists a constant $C$ such that
\begin{gather*} 1 \leq\frac{\log (n+1)}{\log (n - m + 1)}\leq\exp\left(C\frac{m(m+1)}{n}\right) \end{gather*}
for all integers $n,m$ with $0 \leq m \leq n-1$ and $n \geq 1$.
\end{lemma}
\begin{proof}
The lower bound is obvious.
We start by proving the estimate for $m > \frac{n+1}{2}$.
Because $\log$ grows slower than the exponential, there exists a constant $C'$ such that
\begin{gather*}\frac{\log (n+1)}{\log (n - m + 1)}\leq\frac{\log (n+1)}{\log 2}\leq{C'}^n \leq({C'}^{2})^{\frac{n+1}{2}}\leq({C'}^{2})^{m}\leq({C'}^{4})^{\frac{m (m+1)}{n}}. \end{gather*}
The bounds remains to be proven for all $m \leq \frac{n+1}{2}$.
Again we use the inequality $\frac{1}{1-x} \leq \exp{\frac{x}{1-x}} \Rightarrow \log \frac{1}{1-x} \leq \frac{x}{1-x}$ which holds for
all $x \in [0,1)$ to get
\begin{gather*} \frac{\log (n+1)}{\log (n - m + 1)} = \frac{1}{1+\frac{\log (1 - \frac{m}{n+1})}{\log(n+1)}} \leq \exp\left(\frac{- \frac{\log (1 - \frac{m}{n+1})}{\log(n+1)}}{1+\frac{\log (1 - \frac{m}{n+1})}{\log(n+1)}}\right) = \exp\left(\frac{- \log (1 - \frac{m}{n+1})}{\log (n - m+1)}\right) \\
\leq\exp\left(\frac{\log \frac{1}{1 - \frac{m}{n+1}}}{\log 2}\right) \leq \exp\left(\frac{1}{\log 2}\frac{\frac{m}{n+1}}{1-\frac{m}{n+1}}\right) \leq \exp\left(\frac{2}{\log 2}\frac{m}{n+1}\right), \end{gather*}
where we used $ {1/\left(1-\frac{m}{n+1}\right)} \leq \frac12$ in the last step.
\end{proof}

Recall the sequences
$B_n=- \frac{1}{\sqrt{2\pi}}\frac{\Gamma(n-\frac12)}{\log^2 n}$ and
$L_n= \frac{\log n}{\log \log n}$  from Section~\ref{sec:precise}. Because these sequences are not defined for $n=1$, it is convenient to use following ones instead
\begin{align*} \wt B_n&=-\frac{1}{\sqrt{2\pi}}\frac{\Gamma(n-\frac12)}{\log^2 (n+1)} & \wt L_n&=\frac{\log (n+1)}{\log \log (n+e)}. \end{align*}
It is clear that $\bigO(\wt B_n)=\bigO(B_n)$ and $\bigO(\wt L_n)=\bigO(L_n),$ but we also have control over the error:

\begin{lemma}
\label{lmm:wtBL}
$\wt B_n = B_n + \bigO\left(\frac{B_n}{n}\right)$
and
$1/\wt L_n = 1/L_n + \bigO\left(\frac{1}{nL_n}\right)$.
\end{lemma}
\begin{proof}
This follows from the fact that  the following elementary limits exist:
\begin{align*} &\lim_{n \rightarrow \infty} \frac{1/\log^2(n+1)-1/\log^2n}{1/(n\log^2 n)} &&\text{ and }& &\lim_{n \rightarrow \infty} \frac{\log\log(n+e)/\log(n+1)-\log\log n/\log n}{\log\log n/(n\log n)}. \end{align*}
This can be shown, for instance, by using $\log(n+x) - \log n \leq x/n$.
\end{proof}

\begin{lemma} \label{lmm:boundR1}
Let
\begin{gather*} Q_{n,m} = {\Ch_{n-m}/\wt B_{n-m}}-1.  \end{gather*}
There exists a constant $C$ such that
$\left| Q_{n,m} \right| \leq C^{m+1}/\wt L_n $
for all $n,m$ with $m \geq 0$ and $n \geq m+1$.
\end{lemma}
\begin{proof}
By Proposition~\ref{prop:Ch_asymp} and
Lemma~\ref{lmm:wtBL} we have
$\Ch_n - \wt B_n \in \bigO( \wt B_n/ \wt L_n )$.
Because $\wt B_n$, $\Ch_n$ and $\wt L_n$ are finite for all $n\geq 1$, there exists a constant $C'$ such that
$\left|{\widehat{\chi}_n / \wt B_n}-1\right|\leq C'/ \wt L_n    $
for all $n \geq 1$.
 It follows that
for all integers $n,m$ with $m \geq 0$ and $n \geq m+1$.
\begin{gather*} \left|Q_{n,m}\right| =\left| \Ch_{n-m} / \wt B_{n-m}-1\right|\leq C'/\wt L_{n-m} \leq C' \exp\left( C''\frac{m(m+1)}{n}\right)\frac{\log\log(n-m+e)}{\log(n+1)},  \end{gather*}
where we used a constant $C''$ from Lemma~\ref{lmm:log_bound}.
Using $m+1 \leq n$ and the monotonicity of $\log$ gives the bound.
\end{proof}

\begin{lemma}
\label{lmm:boundR2}
Let
\begin{gather*} R_{n,m} = (-1)^m 2^m m! \binom{ - \frac12 (n-m) }{ m } \wt B_{n-m}/ \wt B_n -1. \end{gather*}
There exists a constant $C$ such that
$ \left| R_{n,m}\right|\leq\frac{C^{m+1}}{n} $
for all $n,m$ with $m \geq 0$ and $n \geq m+1$.
\end{lemma}
\begin{proof}
Use
$ \Gamma(n-\frac12) =\Gamma(n-m-\frac12)\prod_{r=0}^{m-1}(n-m-\frac12 +r ) $
and
$ \binom{q}{m}=\frac{1}{m!}\prod_{r=0}^{m-1} (q-r) $
to get
\begin{align*} R_{n,m} +1&=(-1)^m2^{m}\frac{\log^2(n+1)}{\log^2(n-m+1)}\prod_{r=0}^{m-1}\frac{-\frac12 (n - m)-r}{n-m-\frac12 +r} \\
&= \frac{ \log^2(n+1) }{ \log^2(n-m+1) } \prod_{r=0}^{m-1} \frac{ n - m+2r }{ n-m-\frac12 +r }. \end{align*}
Therefore, by Lemma \ref{lmm:prod_bound} and \ref{lmm:log_bound} there exists a constant $C'$ such that
$$1 \leq
1+
R_{n,m}
\leq
\exp\left(C' \frac{m(m+1)}{n}  \right)
$$ for all integers $n, m$ with $m \geq 0$ and $n\geq m +1$. Because $1-x \leq e^{-x} \Rightarrow e^{x} \leq 1 + x e^x$ for all $x \in \R$, we  have
$ \left|R_{n,m} \right| \leq C' \frac{m(m+1)}{n} \exp\left(C' \frac{m(m+1)}{n} \right) \leq C' \frac{m(m+1)}{n} \exp\left(m C' \right) $ for all $m \geq 0$ and $n \geq m +1$.
The statement follows from the fact that we can find a constant $C''$ such that
$ C' \frac{m(m+1)}{n} \exp\left(C' \frac{m(m+1)}{n} \right) \leq \frac{{C''}^{m+1}}{n}$
for all $m \geq 0$ and $n \geq m +1$.
\end{proof}

We are now ready to estimate the numbers $P_n$.

\begin{proposition}\label{prop:hatChi}
\begin{gather*}   P_n = {e^{-\frac14}}\wt B_n+ \bigO\left(\wt B_n / \wt L_n\right).  \end{gather*}
\end{proposition}

\begin{proof}
With $Q_{n,m}$ and $R_{n,m}$ from Lemmas~\ref{lmm:boundR1} and \ref{lmm:boundR2} and Eq.~\eqref{eq:Pn}, we have
\begin{gather*} P_n=\sum_{m =0}^n\widehat{\chi}_{n-m}\binom{-\frac12 (n-m)}{m}\eta_{2,m} = \wt B_n \sum_{m =0}^n\frac{(-1)^m }{ 2^m m!}(1+Q_{n,m})(1+R_{n,m})\eta_{2,m} \end{gather*}
By Lemma~\ref{lmm:eta_bound} there exists a constant $C$ such that $\eta_{2,m} \leq C^{m}\Gamma(\frac{m+1}{2})$ for all $m \geq 0$.
From Lemma~\ref{lmm:boundR1} and \ref{lmm:boundR2} we get a constant $C'$ such that
\begin{gather*} \sum_{m =0}^n\left|\frac{(-1)^m }{ 2^m m!}(Q_{n,m}+R_{n,m}+Q_{n,m}R_{n,m})\eta_{2,m}\right| \\
\leq C'\left(\frac{1}{\wt L_n}+\frac{1}{n}+\frac{1}{n\wt L_n}\right)\sum_{m =0}^\infty\frac{1 }{ 2^m m!}(CC')^{m}\Gamma\left(\frac{m+1}{2}\right) \in \bigO\left(1/\wt L_n \right),  \end{gather*}
where the sum over $m$ is convergent due to the factorial $m!= \Gamma(m+1)$. Hence,
\begin{gather*} P_n=\wt B_n \sum_{m =0}^n\frac{(-1)^m }{ 2^m m!}\eta_{2,m} + \bigO\left(\wt B_n / \wt L_n\right).    \end{gather*}
An application of Lemma~\ref{lmm:e14} and using $\bigO(1/L_n) = \bigO(1/ {\wt L_n})$ concludes the proof.
\end{proof}

\begin{remark}
The appearance of the constant
$e^{-\frac14}$ in 
Proposition~\ref{prop:hatChi} deserves a discussion.
Asymptotically $\Ch_n$ and $\ee_n$ are
equal up to a multiplicative constant and  by 
Lemma~\ref{lmm:e14}, this constant is $e^{-\frac14}$.
This constant is reminiscent of constants that
appear in certain counting problems,  such as counting the number of regular graphs with a fixed
girth, where the girth of a graph is the minimal length of a cycle.
In this count,  numbers $e^{-\lambda}$ appear because the
numbers of $k$-cycles of large random regular graphs follow independent Poisson
distributions with rational mean $\lambda$
(\cite[Thm.~2]{bollobas1980probabilistic}).
Since a Poisson distribution is given by
$e^{-\lambda} {\lambda^k/k!}$,
constants such as $e^{-\frac14}$ appear naturally. The computation in Lemma~\ref{lmm:e14} suggests
that a large admissible random forested graph
has $r$ double edges with probability $e^{-1/4}/(4^r r!)$ and
$m$ self-loops with probability $e^{-1/2}/(2^m m!)$.
However, such an intuitive analogy to random graph theory is derailed by the
various signs that have to be taken into account, so this  seems to make a sound probabilistic interpretation impossible.
\end{remark}

\begin{proof}[Proof of Theorem~\ref{thm:asymptotics}]
We want to show that $\ee_n=e^{-\frac14} B_n + \bigO(B_n/L_n).$  By Proposition~\ref{prop:enasy1} (which we will prove in the next two subsections),  $\ee_n=P_n+\bigO \left(\Gamma\left(n - {7/12}\right)\right),$
and by
 Proposition~\ref{prop:hatChi} $P_n = {e^{-\frac14}}\wt B_n+ \bigO\left(\wt B_n / \wt L_n\right).$
 We now use the elementary fact that
$\bigO\left( \Gamma\left(n - {7/12}\right) \right) \subset \bigO\left( B_n'/L_n' \right) = \bigO\left( B_n/L_n \right)$
along with the asymptotic equality  $\wt B_n = B_n + \bigO\left({B_n/n}\right)$
 from Lemma~\ref{lmm:wtBL} to finish the proof.  \end{proof}

\subsection{Estimates for Proposition~\ref{prop:enasy1}}
\label{sec:estimates_eneasy}

We still need to prove  Proposition~\ref{prop:enasy1}.
To do so we need to establish that the contribution of ${\mathbf H}(u,\bb x)$ in Theorem~\ref{thm:eesecondexpr} is negligible to the asymptotic behavior of the numbers $\ee_n$ and the power series $\TT(u^2 e^{-{\mathbf W}(u \cdot \bar x)})$ only contributes partially.
We first show that the coefficients of $\TT(u^2 e^{-{\mathbf W}(u \cdot \bar x)})$ can be written explicitly using the numbers $\Ch_n$. %
\begin{lemma}
\label{lmm:Tcoeffs}
Let $r \geq 0$ and $\ld$ be a deranged partition.  Then
\begin{align*} [u^{r} x^{\ld}] \TT\left(u^2 e^{-\mathbf{W}(u \cdot \bar x)}\right) = \Ch_{\frac{r-|\ld|}{2}} \prod_{k= 2}^r \binom{(r-|\ld|)\frac{\mu(k)}{2k}}{m_k(\ld)}, \end{align*}
where we agree that
$ \Ch_{\frac{r-|\ld|}{2}} = 0 $
if $r-|\ld|$ is odd or negative. In particular,  $[u^0x^\delta]\TT\left(u^2 e^{-\mathbf{W}(u \cdot \bar x)}\right)=1$ if $\delta=\emptyset$ and $0$ otherwise. Also, $[u^1x^\delta]\TT\left(u^2 e^{-\mathbf{W}(u \cdot \bar x)}\right)=0$ for all $\delta$.
\end{lemma}
\begin{proof}
Use  $\TT(\hbar)=\sum \Ch_n\hbar^n$ and ${\mathbf W}(\bb x)=-\sum_{k \geq 2} \frac{\mu(k)}{k} \log(1+x_k)$ to get
\begin{gather*} \TT\left(u^2 e^{-{\mathbf W}(u \cdot \bar x)}\right)=\sum_{n \geq 0} \Ch_n u^{2n}\prod_{k\geq 2} \left(1+u^{k} x_k\right)^{n\frac{\mu(k)}{k}} \\
=\sum_{n \geq 0} \Ch_n u^{2n}\prod_{k\geq 2} \sum_{m_k \geq 0} \binom{n\frac{\mu(k)}{k}}{m_k} u^{km_k} x_k^{m_k} =\sum_{n \geq 0}\sum_{\ld}\Ch_n u^{2n+|\ld|}x^{\ld}\prod_{k\geq 2}\binom{n\frac{\mu(k)}{k}}{m_k(\ld)}. \qedhere \end{gather*}
\end{proof}
Using this explicit formula we can now obtain a bound on the coefficients of $\TT(u^2e^{-{\mathbf W}(u\cdot\bar x)})$:
\begin{corollary}
\label{cor:Testimate}
There exists a constant $C$, such that
for all $r \geq 2$ and deranged partitions $\ld=[2^{m_2}3^{m_3}\cdots]$ with $|\ld| \leq r$,
\begin{gather*} \left| [u^{r} x^{\ld}] \TT\left(u^2 e^{-{\mathbf W}(u \cdot \bar x)}\right) \right| \leq C \frac{ \Gamma\left(\frac{r-|\ld|+2\ell(\ld)-1}{2}\right) }{m_2! m_3!}. \end{gather*}
\end{corollary}

\begin{proof}
We start with the special cases where $\ld$ is a deranged partition of $r$ or $r-1$.  In both cases $r \geq 2$ implies $\ld \neq \emptyset$, i.e.\ $\ell(\ld) \geq 1$  and the argument of the $\Gamma$ function on the right hand side of the statement is positive.
By Lemma~\ref{lmm:Tcoeffs} and the agreement that $\widehat \chi_{\frac12} = 0$ and $\binom{0}{m} = 0$ for all $m\geq 1$, we find that the left hand side vanishes in these cases and the statement follows.

For $|\ld|\leq r-2$ we can apply
 Corollary~\ref{crll:chiestimate} to the statement of Lemma~\ref{lmm:Tcoeffs} to get a constant $C$ such that
for all deranged partitions $\ld=[2^{m_2}3^{m_3}\cdots]$,
\begin{gather} \label{eq:Test1} \left| [u^{r} x^{\ld}]\TT\left(u^2 e^{-{\mathbf W}(u \cdot \bar x)}\right)\right|\leq C\Gamma \left(\frac{r- |\ld| -1}{2} \right)\left|\prod_{k = 2}^r\binom{(r-|\ld|)\frac{\mu(k)}{2k}}{m_k} \right|. \end{gather}
From the standard expression for the binomial coefficients $\binom{q}{m} = \frac{q(q-1)\cdots (q-m+1)}{m!}$, we get
\begin{gather*} \left| \prod_{k = 2}^r \binom{(r-|\ld|)\frac{\mu(k)}{2k}}{m_k} \right| = \prod_{k = 2}^r \frac{ \left| \prod_{s=0}^{m_k-1} \left((r-|\ld| )\frac{\mu(k)}{2k}-s\right) \right| } {m_k!} \leq \frac{ \prod_{k = 2}^r \prod_{s=0}^{m_k-1} \left((r-|\ld|)\frac{1}{2k}+s\right) } {m_2!m_3!} , \end{gather*}
where we used $|\mu(n)| \leq 1$. By using $x\Gamma(x) = \Gamma(x+1)$ repeatedly we get
\begin{gather*} \Gamma \left(\frac{r- |\ld| -1}{2} \right) = \frac{ \Gamma \left(\frac{r- |\ld|+2\ell(\ld) -1}{2} \right) }{ \prod_{s=0}^{\ell(\ld)-1} \left( \frac{r- |\ld| -1}{2} +s \right) }. \end{gather*}
Combining these observations with Eq.~\eqref{eq:Test1} gives
\begin{gather*} \left| [u^{r} x^{\ld}]\TT\left(u^2 e^{-{\mathbf W}(u \cdot \bar x)}\right)\right|\leq C\frac{\Gamma \left(\frac{r- |\ld|+2\ell(\ld) -1}{2} \right)}{{m_2!m_3!}}\frac{\prod_{k = 2}^r\prod_{s=0}^{m_k-1}\left((r-|\ld|)\frac{1}{2k}+s\right)}{\prod_{s=0}^{\ell(\ld)-1} \left( \frac{r- |\ld| -1}{2}+s \right) }. \end{gather*}
It is easy to verify that
$ \prod_{k = 2}^r \prod_{s=0}^{m_k-1} \left((r-|\ld|)\frac{1}{2k}+s\right) \leq \prod_{s=0}^{\ell(\ld)-1} \left(\frac{r-|\ld|-1}{2} +s\right) $ by using $\ell(\ld) = \sum_{k=2}^{|\ld|} m_k$ and $r-|\ld| \geq 2$.
\end{proof}

It is substantially harder to prove a good estimate for the coefficients of ${\mathbf H}(u,\bb x)$. It turns out to be convenient to include the numbers $\eta_{\lambda}$ in our estimate.

For integer partitions $\lambda,\lambda'$, we define  $\lambda\cup\lambda'$
to be the partition of $|\lambda|+|\lambda'|$ that contains the union of all
parts of $\lambda$ and $\lambda'$, i.e.~$m_{k}(\lambda \cup \lambda') = m_{k}(\lambda) + m_{k}(\lambda')$ for all $k$.
\begin{proposition}
\label{prop:Hestimate}
There is a constant $C$ such that
for all $r \geq 0$ and deranged
partitions $\ld$,
\begin{align*} \sum_{\md} \eta_{\ld \cup \md}\left|[u^r x^{\md}]{\mathbf H}(u, \bb x) \right| \leq C^{1+r+|\ld|}\Gamma\left( \frac{\ell(\ld) + \frac{5}{6} r+1}{2} \right), \end{align*}
where we sum over all deranged integer partitions $\md$.
\end{proposition}
The proof of this proposition will occupy the remainder of this subsection.

Recall from Section~\ref{sec:relating} that the series ${\mathbf H}(u,\bb x)$ is defined by ${\mathbf H}=\exp\left({\mathbf h}_1 +{\mathbf h}_2+ {\mathbf h}_3\right)$
 where
\begin{align*} {\mathbf h}_1(u,\bb x)&= \sum_{k \geq 2} u^{-2k} \frac{ {\mathbf V}( (u \cdot \bb x)_{[k]})}{k},\\
{\mathbf h}_2(u,\bb x)&= \frac{{\mathbf W}(u\cdot\bb x)}{2}, \\
{\mathbf h}_3(u,\bb x)&= u^{-2} \left(e^{{\mathbf W}(u\cdot\bar x)}-1-u^2\frac{x_2}{2}\right). \end{align*}
 \begin{lemma}\label{lem:Jsets} $[u^rx^\delta]\log\mathbf H(u,\bb x)= 0$ unless   either
 \begin{enumerate}
\item   $r \geq 2$ and   $|\ld|=r$
\item   $r \geq 1$ and  $|\ld|=r+2$ or
\item    $|\delta|=r+2k$ for some $k\geq 2$ dividing $r$ and $\delta=k\mu=(k\mu_1,\ldots,k\mu_\ell)$ for some partition $\mu$.
\end{enumerate}
\end{lemma}
 \begin{proof}
 The nonzero terms of ${\mathbf h}_2$ are all of the form (1), the nonzero terms of $ {\mathbf h}_3$ are of the form (2) and the nonzero terms of ${\mathbf h}_1$ are of the form (3). \end{proof}

 Let $J$ be the set of pairs $(r,\delta)$ satisfying any of the three conditions of Lemma~\ref{lem:Jsets}.

\begin{corollary}
\label{cor:rld}
For all $(r,\ld) \in J$, we have $|\ld| \leq 3r$.
\end{corollary}

\begin{proof}
The statement is obvious for the first two cases of Lemma~\ref{lem:Jsets}, and follows in the third case because $k$ divides $r$, so is at most $r$.
\end{proof}

\begin{remark} This corollary also follows immediately from the combinatorics:  An admissible graph $G$ with $\chi(G)= -n$ has at most $3n$ edges, so  if pairing all leaves of some forest produces an an admissible graph with $-2\chi(G)= r$ then the forest cannot have more than $3r (=6n)$ leaves.
\end{remark}

\begin{lemma}\label{lem:expbound}
All coefficients of $\log \mathbf H$ have absolute value less than $1$.
\end{lemma}
\begin{proof}
Since the nonzero terms of $\mathbf h_1,\mathbf h_2$ and $\mathbf h_3$ have no monomials in common, we can consider their coefficients separately.
It is clear from their definitions that the coefficients of  $\mathbf W(\bb x)$ and $\mathbf V(\bb x)$ are less than $1$, so the coefficients of $\mathbf h_2$ and $\mathbf h_1$ are as well.  For $\mathbf h_3$, we have
\begin{align*} e^{{\mathbf W}(\overline x)} &= \prod_{k \geq 2} (1+x_k)^{-\frac{\mu(k)}{k}} = \prod_{k \geq 2} \left( \sum_{m \geq 0} \binom{-\frac{\mu(k)}{k}}{m} x_k^m \right) \\
&= \prod_{k \geq 2} \left( \sum_{m \geq 0} \frac{-\frac{\mu(k)}{k}}{1} \cdot \frac{-\frac{\mu(k)}{k}-1}{2} \cdots \frac{-\frac{\mu(k)}{k}-m+1}{m} x_k^m \right). \end{align*}
The magnitude of the fractions $ |\frac{-\frac{\mu(k)}{k}-\ell +1}{\ell}|$
is always smaller than $1$, because $|\mu(k)| \leq 1$. Hence all
coefficients of
$e^{{\mathbf W}(\overline x)}$ (and therefore of $\mathbf h_3$) are also less than $1$.
\end{proof}

The next lemma shows that we can bound the coefficients of ${\mathbf H}(u,\bb x)$   without determining the explicit values of those coefficients.

\begin{lemma}\label{lmm:Hlsum} Let $J$ be the set of pairs $(r,\delta)$ satisfying the conditions of Lemma~\ref{lem:Jsets}.
For all $C \geq 1$%
\begin{align*} \sum_{\ell(\md) = \ell'}{C}^{|\md|}\left|[u^{r'} x^{\md}]{\mathbf H}(u, \bb x) \right|\leq C^{3r'}[u^{r'} x^{\ell'}]\exp\left(\sum_{(r,\ld) \in J}u^r x^{\ell(\ld)}\right), \end{align*}
where we sum over all deranged partitions $\md$ of length $\ell'$ on the left hand side.
\end{lemma}

\begin{proof}
By Lemma~\ref{lem:expbound}  and the positivity of the expansion
$\exp(X) =\sum_{n = 0}^\infty \frac{X^n}{n!},$
 \begin{gather*} {C}^{|\md|}\left|[u^{r'} x^{\md}]{\mathbf H}(u, \bb x) \right|\leq {C}^{|\md|} [u^{r'} x^{\md}] \exp \left( \sum_{(r,\ld) \in J} u^{r} x^{\ld} \right). \end{gather*}
By Corollary~\ref{cor:rld} the coefficient extraction operator has support only for $|\ld|\leq 3r$.
Since $C\geq 1$,   $C^{|\ld|} \leq C^{3r}$. Summing over all deranged partitions $\ld$ of the same length on the left is equivalent to  substituting $x_i = x$ for all $i$ on the right.
\end{proof}

Our next estimate will make use of the following rough bound on integer partitions.
\begin{lemma}\label{lem:PartitionBound}
The number of  integer partitions of size at most $n$ is smaller than $2^n$.
\end{lemma}
\begin{proof} Writing a string of length $n$ of $1$'s with either a $+$ or a comma in between gives a \emph{composition} of $n$, i.e.~an integer partition of $n$ with an ordering of the parts.
There are $2^{n-1}$ different such compositions of $n$. Ordering the elements of a composition gives a many-to-one function to integer partitions, which is clearly surjective for all $n\geq 1$. Thus the number of integer partitions of $n$ is bounded by $2^{n-1}$ for $n\geq 1$.
The number of integer partitions of size at most $n$ is therefore smaller than $1 + 2^{0} + 2^{1} + \ldots + 2^{n-1}$, where the initial $1$ accounts for the empty integer partition. Since
$2^{0} + \ldots + 2^{n-1} = 2^n-1$, the statement follows.
\end{proof}

Now let $K$ be the subset of pairs $(r,\ld)$ in $J$ for which $\frac56 r < \ell(\ld)$.
\begin{lemma}
\label{lmm:JKexpbound}
There is a constant $C$ such that for all $r',\ell' \geq 0$,
\begin{align*} [u^{r'} x^{\ell'}]\exp\left(\sum_{(r,\ld) \in J\setminus K}u^r x^{\ell(\ld)}\right)\leq C^{r'}. \end{align*}
Moreover, $ [u^{r'} x^{\ell'}]\exp\left(\sum_{(r,\ld) \in J\setminus K}u^r x^{\ell(\ld)}\right)= 0$ when $\frac56 r'< \ell'$.
\end{lemma}

\begin{proof}
The second statement follows immediately from the definition of $K$.

By Corollary~\ref{cor:rld}, the set $J\setminus K$ is contained in the set of all pairs $(r,\lambda)$ where $r \geq 1$ and $\lambda$ is a partition such that $|\lambda| \leq 3r$. By Lemma~\ref{lem:PartitionBound} this implies that
$ \sum_{(r,\ld) \in J\setminus K}u^r $
is bounded coefficientwise by
$ \sum_{r\geq 1} 2^{3r} u^r. $
Since $\log r \leq r$ for all $r \geq 1$, the series
$ \sum_{r\geq 1} 2^{3r}u^r =\sum_{r\geq 1} u^r 8^{r} e^{ \log r}/r, $
 is bounded coefficientwise by
$ \sum_{r\geq 1} u^r(e8)^{r} /r. $
Summing over $\ell'$ on the left hand side of the inequality in the statement of the lemma is equivalent to setting $x=1$. Hence,
\begin{align*} [u^{r'} x^{\ell'}] \exp \left( \sum_{(r,\ld) \in J\setminus K} u^r x^{\ell(\ld)} \right) \leq [u^{r'}] \exp \left( \sum_{(r,\ld) \in J\setminus K} u^r \right) \leq [u^{r'}] \exp \left( \sum_{r\geq 1} \frac{(8e)^{r}}{r} u^r \right) = (8e)^{r'}, \end{align*}
where we used $\log\frac{1}{1-x} = \sum_{r\geq1} \frac{x^r}{r}$ in the last step.
\end{proof}

\begin{corollary}
\label{cor:JKexpbound2}
There is a constant $C$ such that for all $r' \geq 0$ and $z\in\R$ with $z\geq 0$,
\begin{align*} \sum_{\ell'\geq 0} \Gamma\left(\frac{z+\ell'+1}{2}\right) [u^{r'} x^{\ell'}] \exp \left( \sum_{(r,\ld) \in J\setminus K} u^r x^{\ell(\ld)} \right) \leq C^{r'+1} \Gamma\left(\frac{z+\frac56 r'+1}{2}\right). \end{align*}

\end{corollary}
\begin{proof}
By Lemma~\ref{lmm:JKexpbound} there is a constant $C'$ such that
\begin{align*} \sum_{\ell'\geq 0} \Gamma\left(\frac{z+\ell'+1}{2}\right) [u^{r'} x^{\ell'}] \exp \left( \sum_{(r,\ld) \in J\setminus K} u^r x^{\ell(\ld)} \right) \leq {C'}^{r'} \sum_{0 \leq \ell' \leq \lfloor \frac56 r' \rfloor} \Gamma\left(\frac{z+\ell'+1}{2}\right). \end{align*}
By Corollary~\ref{cor:gamma_merge},
$ \Gamma\left(\frac{z+\ell'+1}{2}\right) {\Gamma\left(\frac{\frac56 r'-\ell'+1}{2}\right)} \leq \sqrt{\pi} {\Gamma\left(\frac{z+ \frac56 r'+1}{2}\right)} $
for all $0\leq \ell' \leq \frac56 r'.$
Hence,
$$
\sum_{0 \leq \ell' \leq \lfloor \frac56 r' \rfloor}
\Gamma\left(\frac{z+\ell'+1}{2}\right)
\leq
\sqrt{\pi}
\Gamma\left(\frac{z+ \frac56 r'+1}{2}\right)
\sum_{0 \leq \ell' \leq \lfloor \frac56 r' \rfloor}
\frac{1}
{\Gamma\left(\frac{ \frac56 r'-\ell'+1}{2}\right)}.
$$
The sum over $\ell'$ is bounded by a constant independent of $r'$ as
$$
\sum_{0 \leq \ell' \leq \lfloor \frac56 r' \rfloor}
\frac{1}
{\Gamma\left(\frac{ \frac56 r'-\ell'+1}{2}\right)}
=
\sum_{0 \leq \ell' \leq \lfloor \frac56 r' \rfloor}
\frac{1}
{\Gamma\left(\frac{ \frac56 r'-\lfloor \frac56 r' \rfloor + \ell'+1}{2}\right)}
\leq
\frac{{C''}^{\frac56 r'-\lfloor \frac56 r' \rfloor }}{
\Gamma\left(\frac{ \frac56 r'-\lfloor \frac56 r'\rfloor+1}{2}
\right)}
\sum_{0 \leq \ell' \leq \infty
}
\frac{{C''}^{\ell'}}
{\Gamma\left(\frac{ \ell'+1}{2}\right)},
$$
where we used Lemma~\ref{lmm:gamma_split} to split the $\Gamma$ function.
\end{proof}

\begin{lemma} \label{lmm:Kcases}
The set $K$ is a subset of
$$\overline K = \{(1,[3^1]),(2,[2^3]),(2,[2^14^1]),(2,[2^2]),
(3,[3^3]),
(4,[2^4]) \}.$$
\end{lemma}
\begin{proof}
By Lemma~\ref{lem:Jsets}, for all $(r,\delta)\in J$ we have  $|\delta|=r+2k$ for some $k\geq 0$, and if $k\neq 0$ then $k$ divides $r$ and $\delta=k\mu$ for some partition $\mu$.
Recall that $K$ is the subset of pairs $(r,\ld)$ in $J$ for which $\frac56 r < \ell(\ld)$.
Since each part of a deranged partition has size at least $2$, $\ell(\delta)\leq |\delta|/2$, so $\frac 56 r<\ell(\delta)$ implies $2r<6k$.
We look for pairs $(r,\delta)$ that fulfill all these conditions.

    If $k=0$, there are no solutions.

    If $k=1$, the only solutions are $r\in\{1,2\}$ and the possible deranged partitions    $\ld$ are $\ld \in \{[3^1]\}$ for $r=1$ and
$\ld \in \{[2^2],[4^1]\}$ for $r=2$.

If $k\geq 2$, the additional constraint that $k$ divides $r$ implies that   $(r,k) \in \{(2,2),(3,3),(4,2)\}$.
    For both $(r,k) = (2,2)$ and $(r,k) = (3,3)$, we  have $|\mu|= 3$, which means $\mu \in \{[1^3],[1^12^1],[3^1]\}$. Hence, $\ld \in \{[2^3],[2^14^1],[6^1]\}$ for $(r,k)=(2,2)$ and $\ld \in \{[3^3],[3^16^1],[9^1]\}$ for $(r,k)=(3,3)$.
For $(r,k) = (4,2)$ we  have $\delta=2\mu$ with $\mu \in \{[1^4],[1^22^1],[2^2],[1^13^1],[4^1]\}$ and therefore $\ld\in\{[2^4],[2^24^1],[4^2],[2^16^1],[8^1]\}$.

The requirement $\ell(\ld)>\frac56 r$ now eliminates all solutions that are not in the list $\overline K$.
\end{proof}
\begin{corollary}
\label{cor:Kbounds}
For all $(r,\ld) \in \overline K$, we have $\frac56 r < \ell(\ld) < \frac56 r + 2$.
\end{corollary}
\begin{proof}
This can be verified by checking all 6 elements in Corollary~\ref{lmm:Kcases}.
\end{proof}
\begin{lemma}
\label{lmm:Kexpbound}
Fix $(r,\ld) \in \overline K$. There is a constant $C$ such that for $r' \geq 0$ and $z\in \R$ with $z \geq 0$,
\begin{align*} \sum_{\ell' \geq 0} [u^{r'} x^{\ell'}] \exp( u^r x^{\ell(\ld)} ) \Gamma\left( \frac{z + \ell' + 1}{2} \right) \leq C^{z+r'+1} \Gamma\left( \frac{z + \frac{5}{6} r'+1}{2} \right). \end{align*}
\end{lemma}
\begin{proof}
We can expand the left hand side to get
\begin{align*} \sum_{\ell' \geq 0} [u^{r'} x^{\ell'}] \exp( u^r x^{\ell(\ld)} ) \Gamma\left( \frac{z + \ell' + 1}{2} \right) = \sum_{\ell' \geq 0} [u^{r'} x^{\ell'}] \sum_{k\geq 0} \frac{u^{rk} x^{\ell(\ld) k}}{k!} \Gamma\left( \frac{z + \ell' + 1}{2} \right). \end{align*}
The sum over $k$ only contributes if $rk=r'$. Hence, if $r$ divides $r'$ we get
\begin{align*} \sum_{\ell' \geq 0} [u^{r'} x^{\ell'}] \exp( u^r x^{\ell(\ld)} ) \Gamma\left( \frac{z + \ell' + 1}{2} \right) = \frac{1}{(r'/r)!} \Gamma\left( \frac{z + \ell(\ld) \frac{r'}{r} + 1}{2} \right) \end{align*}
and zero otherwise.
By Lemma~\ref{lmm:gamma_split} there exists a constant $C'$ such that
for $\alpha$ with $0 \leq \alpha \leq \ell(\ld)$,
\begin{align*} \frac{1}{(r'/r)!} \Gamma\left( \frac{z+\ell(\ld) \frac{r'}{r} + 1}{2} \right) \leq \frac{1}{(r'/r)!} {C'}^{\frac{z}{2} + \ell(\ld) \frac{r'}{2r} + 1} \Gamma\left( \alpha \frac{r'}{2r} +\frac12 \right) \Gamma\left( \frac{z+(\ell(\ld)-\alpha) \frac{r'}{r} + 1}{2} \right). \end{align*}
Recall that $n! = \Gamma(n+1)$. Hence,
$\Gamma\left( \alpha \frac{r'}{2r} +\frac12 \right)/(r'/r)!$ is bounded for all $r' \geq 0$ if $\alpha \leq 2$.
We may choose $\alpha = \ell(\ld) - \frac56 r$, which fulfills $0 \leq \alpha \leq 2$ as $\frac56 r \leq \ell(\ld) \leq 2+\frac56r$ by Corollary~\ref{cor:Kbounds}.
\end{proof}

\begin{lemma}
\label{lmm:Ar}
Let
$$
A_{r'} =
\sum_{\ell' \geq 0}
\Gamma\left( \frac{\ell'+1}{2} \right)
[u^{r'} x^{\ell'}]
\exp
\left(
\sum_{(r,\ld) \in J}
u^r x^{\ell(\ld)}
\right).
$$
There is a constant $C$ such that
$ A_{r'} \leq C^{r'+1} \Gamma\left(\frac{ \frac{5}{6} r'+ 1}{2} \right) $
for all $r' \geq 0$.
\end{lemma}
\begin{proof}
We can split the exponential in the definition of $A_{r'}$,
\begin{align*} \exp \left( \sum_{(r,\ld) \in J} u^r x^{\ell(\ld)} \right) = \exp \left( \sum_{(r,\ld) \in K} u^r x^{\ell(\ld)} \right) \exp \left( \sum_{(r,\ld) \in J\setminus K} u^r x^{\ell(\ld)} \right). \end{align*}
By Lemma~\ref{lmm:Kcases},
$ \exp \left( \sum_{(r,\ld) \in K} u^r x^{\ell(\ld)} \right) $ is bounded by
$ \exp \left( \sum_{(r,\ld) \in \overline K} u^r x^{\ell(\ld)} \right) $ coefficientwise.
Let $(r_1,\ld_1),\ldots,(r_6,\ld_6)$ denote the elements of $\overline K$ in Lemma~\ref{lmm:Kcases}.
With this notation,
$$
\exp
\left(
\sum_{(r,\ld) \in \overline K}
u^r x^{\ell(\ld)}
\right)
=
\prod_{i=1}^6
e^{u^{r_i} x^{\ell(\ld_i)}}.
$$
We may use this to write $A_{r'}$ as
\begin{align*} A_{r'} = \sum \Gamma\left( \frac{\sum_{i=1}^7 \ell_i' + 1}{2} \right) \left( \prod_{i=1}^6 [u^{r_i'} x^{\ell_i'}] e^{u^{r_i} x^{\ell(\ld_i)}} \right) [u^{r_7'}x^{\ell_7'}] \exp \left( \sum_{(r,\ld) \in J \setminus K} u^r x^{\ell(\ld)} \right), \end{align*}
where we have to sum over all tuples of integers $\ell_1',\ldots,\ell_7', r_1',\ldots, r_7' \geq 0$ with $r_1'+\ldots+r_7' = r'$.
Applying Corollary~\ref{cor:JKexpbound2} once and Lemma~\ref{lmm:Kexpbound} six times results in the bound
\begin{align*} A_{r'} \leq \sum {C'}^{\sum_{i=1}^7 r_i'+1} \Gamma\left( \frac{\frac{5}{6}\sum_{i=1}^7 r_i' + 1}{2} \right), \end{align*}
where $C'$ is an appropriate constant, whose existence follows from Lemma~\ref{lmm:Kexpbound} and Corollary~\ref{cor:JKexpbound2}, and we sum over all $r_1',\ldots,r_7' \geq 0$ with $r_1'+\ldots+r_7' = r'$. The terms in the sum are constant and their number is $\binom{r'+7-1}{r'}$, which is smaller than $2^{r'+6}$ by the binomial theorem.
\end{proof}

\begin{proof}[Proof of Proposition~\ref{prop:Hestimate}]
By Corollary~\ref{cor:etalambdaestimate} and Lemma~\ref{lmm:gamma_split}
there are constants $C'$ and $C''$ such that for all $\ld$ and $\md$,
\begin{gather*} \eta_{\ld \cup \md} \leq \frac{{C'}^{|\ld| + |\md|}}{\sqrt{\pi}} \Gamma\left( \frac{\ell(\ld) +\ell(\md)+1}{2}\right) \leq \frac{{C'}^{|\ld| + |\md|} {C''}^{\ell(\ld)+\ell(\md)}}{\sqrt{\pi}} \Gamma\left( \frac{\ell(\ld) +1}{2}\right) \Gamma\left( \frac{\ell(\md)+1}{2}\right). \end{gather*}
We may assume that $\widetilde C = C'C'' >1$ and use the fact that $\ell(\ld) \leq |\ld|$ and $\ell(\md) \leq |\md|$ to obtain
\begin{gather*} \sum_{\md} \eta_{\ld \cup \md} \left| [u^{r'} x^{\md}] {\mathbf H}(u, \bb x) \right| \leq \frac{{\widetilde C}^{|\ld|}}{\sqrt{\pi}} \Gamma\left( \frac{\ell(\ld) +1}{2} \right) \sum_{\ell' \geq 0} \Gamma\left( \frac{\ell'+1}{2} \right) \sum_{\ell(\md)= \ell'} {\widetilde C}^{|\md|} \left| [u^{r'} x^{\md}] {\mathbf H}(u, \bb x) \right| \intertext{for all $r' \geq 0$ and deranged partitions $\ld$. By Lemma~\ref{lmm:Hlsum}  this last expression is bounded by } \leq \frac{{\widetilde C}^{|\ld|}}{\sqrt{\pi}} \Gamma\left( \frac{\ell(\ld) +1}{2} \right) \sum_{\ell' \geq 0} \Gamma\left( \frac{\ell'+1}{2} \right) {\widetilde C}^{3r'} [u^{r'} x^{\ell'}] \exp \left( \sum_{(r,\ld) \in J} u^r x^{\ell(\ld)} \right), \end{gather*}
for all $r'$ and $\ld$.
The statement follows by estimating the sum over $\ell'$ using Lemma~\ref{lmm:Ar}.
\end{proof}

\subsection{Proof of Proposition~\ref{prop:enasy1}}
\label{sec:proofeneasy}

\begin{lemma}
\label{lmm:THestimate}
There exists a constant $C$ such that for all $n,r \geq 0$ with $r\leq 2n-2$ and deranged partitions $\ld=[2^{m_2}\cdots]$ with $|\ld|\leq 2n-r$,
\begin{gather*} \left| [u^{2n-r} x^{\ld}] \TT\left(u^2 e^{-{\mathbf W}(u \cdot \bar x)}\right) \right| \sum_{\md} \eta_{\ld \cup \md} \left| [u^{r} x^{\md}] {\mathbf H}(u, \bb x) \right| \\
\leq \frac{ C^{1+r + |\ld|} }{ \Gamma\left( \frac{m_2}{2} + 1\right) } \Gamma\left( n - \frac{ \frac16 (r +|\ld|-2m_2) +1 }{2} \right). \end{gather*}
\end{lemma}
\begin{proof}
By Corollary~\ref{cor:Testimate} and Proposition~\ref{prop:Hestimate} we can find constants $C_1$ and $C_2$ such that
for all $2n-r \geq 2$ and deranged partitions $\ld=[2^{m_2}3^{m_3}\cdots]$ with $|\ld| \leq 2n-r$,
\begin{gather*} \left| [u^{2n-r} x^{\ld}] \TT\left(u^2 e^{-{\mathbf W}(u \cdot \bar x)}\right) \right| \sum_{\md} \eta_{\ld \cup \md} \left| [u^{r} x^{\md}] {\mathbf H}(u, \bb x) \right| \\
\leq C_1 C_2^{1+r+|\ld|} \frac{ \Gamma\left(n - \frac{r+|\ld|-2\ell(\ld)+1}{2}\right) }{m_2! m_3!} \Gamma\left( \frac{\ell(\ld) + \frac{5}{6} r+1}{2} \right). \end{gather*}
Recall that $\ell(\ld) = \sum_{k=2}^{|\ld|}m_k$.
By Lemma~\ref{lmm:gamma_split} there is a constant $C_3$ such that
\begin{gather*} \Gamma\left( \frac{\ell(\ld) + \frac{5}{6} r+1}{2} \right) \\
\leq C_3^{1+\frac{\ell(\ld) + \frac{5}{6} r}{2}} C_3^{1+m_2+m_3} \Gamma\left( \frac{m_2+1}{2}\right) \Gamma\left( \frac{m_3+1}{2}\right) \Gamma\left( \frac{\ell(\ld)-m_2-m_3+ \frac{5}{6} r+1}{2} \right). \end{gather*}
By the duplication formula of the $\Gamma$ function (Eq.~\eqref{eq:duplication}),
$ \Gamma\left( \frac{m+1}{2}\right) / m! =\sqrt{\pi}/(2^{m} \Gamma(m/2+1))$.
Combining all this and using Corollary~\ref{cor:gamma_merge} to merge the $\Gamma$ functions, we find that
\begin{gather*} \left| [u^{2n-r} x^{\ld}] \TT\left(u^2 e^{-{\mathbf W}(u \cdot \bar x)}\right) \right| \sum_{\md} \eta_{\ld \cup \md} \left| [u^{r} x^{\md}] {\mathbf H}(u, \bb x) \right| \\
\leq \sqrt{\pi}^3 \frac{C_1 C_2^{1+r+|\ld|}C_3^{2+m_2+m_3+\frac{\ell(\ld) + \frac{5}{6} r}{2}} }{2^{m_2+m_3}} \frac{ \Gamma\left(n - \frac{\frac16 r+|\ld|+m_2+m_3-3\ell(\ld)+1}{2}\right) }{\Gamma\left(m_2/2+1\right) \Gamma\left(m_3/2+1\right)}. \end{gather*}
Moreover, we can use Corollary~\ref{cor:gamma_merge} to get the bound,
$$
\Gamma\left(n - \frac{\frac16 r+|\ld|+m_2+m_3-3\ell(\ld)+1}{2}\right)
\leq
\sqrt{\pi}
\frac{
\Gamma\left(n - \frac{\frac16 (r+|\ld|-2m_2)+1}{2}\right)
}{
\Gamma\left(\frac{\frac56|\ld|+\frac43 m_2+m_3-3\ell(\ld)+1}{2}\right)
},
$$
where the denominator is bounded from below as $\frac56|\ld|+\frac43 m_2+m_3-3\ell(\ld) \geq 0$ for all deranged partitions $\ld$, which follows from
$|\ld| = \sum_{k=2}^{|\ld|} k m_k$,  and $\Gamma(x)$ does not vanish for $x>0$.
\end{proof}
\begin{corollary}
\label{cor:Bns}
For given $n \geq1$ and $s \leq 2n$, let
\begin{align*} B_{n,s} = \sum \eta_{\ld \cup \md} \left( [u^{2n-r} x^{\ld}] \TT\left(u^2 e^{-{\mathbf W}(u \cdot \bar x)}\right) \right) \left( [u^{r} x^{\md}] {\mathbf H}(u,\bb x) \right), \end{align*}
where the sum is over all integers $r$ and
all pairs of deranged partitions $(\ld,\md)$ with $0\leq r \leq 2n-2$, $\ld=[2^{m_2}3^{m_3}\cdots]$ and the restriction that
$r+\sum_{k=3}^{|\ld|}km_k= r+|\ld|-2m_2= s$.
We have
$$
\sum_{s=1}^{2n}
B_{n,s}
\in
\bigO
\left(
\Gamma
\left(
n - \frac{7}{12}
\right)
\right).
$$
\end{corollary}
\begin{proof}

By Lemma~\ref{lmm:THestimate}
there exists a constant $C$ such that
\begin{gather*} |B_{n,s}| \leq \sum \frac{ C^{1+r + |\ld|} }{ \Gamma\left( \frac{m_2}{2} + 1\right) } \Gamma\left( n - \frac{ \frac16 s +1 }{2} \right), \end{gather*}
where the sum runs over all
integers $r$ and deranged partitions $\ld$
with $0\leq r \leq 2n-2$ and
$r+\sum_{k=3}^{|\ld|}km_k = s$.
This suggests treating the parts of size $2$ of the partition $\ld$ separately: We also have the bound
\begin{gather*} |B_{n,s}| \leq \sum_{ \substack{ r\geq 0, \ld^{(3)} \\
r+ |\ld^{(3)}| = s } } \sum_{m_2 \geq 0} \frac{ C^{1+r + |\ld^{(3)}|+2m_2} }{ \Gamma\left( \frac{m_2}{2} + 1\right) } \Gamma\left( n - \frac{ \frac16 s +1 }{2} \right), \end{gather*}
where we sum over pairs $(r,\ld^{(3)})$ where $r$ is an integer $\geq 0$ and $\ld^{(3)}$ is an integer partition where each part has size at least $3$ with $r+ |\ld^{(3)}| = s$. There are at most as many such pairs $(r,\ld^{(3)})$ as there are integer partitions of $s$ and, by Lemma~\ref{lem:PartitionBound},  there are fewer than $2^{s}$ integer partitions of $s$. Therefore,
\begin{gather*} |B_{n,s}| \leq 2^s \sum_{m_2 \geq 0} \frac{ C^{1+s+2m_2} }{ \Gamma\left( \frac{m_2}{2} + 1\right) } \Gamma\left( n - \frac{ \frac16 s +1 }{2} \right) = 2^s C' C^{1+s} \Gamma\left( n - \frac{ \frac16 s +1 }{2} \right) , \end{gather*}
where $C' = \sum_{m_2\geq 0} \frac{C^{2 m_2}}{ \Gamma\left( \frac{m_2}{2} + 1\right)} $, which is obviously convergent.
By Corollary~\ref{cor:gamma_merge}, we have for all $2n\geq s\geq 1$,
$ \Gamma\left( n - \frac{ \frac16 s +1 }{2} \right) \leq \sqrt{\pi} { \Gamma\left( n - \frac{7}{12} \right) / \Gamma\left( \frac{ \frac16 (s-1) +1 }{2} \right) }. $
Hence
$$
\sum_{s=1}^{2n} |B_{n,s}| \leq
\sqrt{\pi}
\sum_{s=1}^{2n}
2^s
C'
C^{1+s}
\frac{
\Gamma\left(
n -
\frac{7}{12}
 \right)
}{
\Gamma\left(
\frac{
\frac16 (s-1) +1
}{2}
 \right)
}
\leq
\sqrt{\pi}
C C'
\Gamma\left(
n -
\frac{7}{12}
 \right)
\sum_{s=1}^{\infty}
\frac{
2^s
C^{s}
}{
\Gamma\left(
\frac{
\frac16 (s-1) +1
}{2}
 \right)
},
$$
where the sum over $s$ is convergent.
\end{proof}

\begin{proof}[Proof of Proposition~\ref{prop:enasy1}]
By  Theorem~\ref{thm:eesecondexpr}  we have
\begin{align*} \ee_n =\sum_{r=0}^{2n}\sum_{\ld,\md} \eta_{\ld \cup \md} \left([u^{2n-r} x^{\ld}]\TT\left(u^2 e^{-{\mathbf W}(u \cdot \bar x)}\right)\right) \left([u^{r} x^{\md}]{\mathbf H}(u,\bb x)\right), \end{align*}
where we sum over all pairs of deranged partitions $\ld$ and $\md$.

By Lemma~\ref{lmm:Tcoeffs}, the expression
$ [u^{2n-r} x^{\ld}]\TT\left(u^2 e^{-{\mathbf W}(u \cdot \bar x)}\right) $
vanishes if $r=2n-1$ or if $r=2n$ and $\delta\neq\emptyset$, while
$[u^{1} x^{\emptyset}]\TT\left(u^2 e^{-{\mathbf W}(u \cdot \bar x)}\right)=1$.
Therefore
\begin{gather*} \sum_{r=2n-1}^{2n}\sum_{\ld, \md}\eta_{\ld \cup \md} \left([u^{2n-r} x^{\ld}]\TT\left(u^2 e^{-{\mathbf W}(u \cdot \bar x)}\right)\right) \left([u^{r}x^{\md}]{\mathbf H}(u,\bb x)\right) = \sum_{\md}\eta_{\md}[u^{2n}x^{\md}]{\mathbf H}(u,\bb x). \end{gather*}
By Proposition~\ref{prop:Hestimate},
the right hand side is bounded by
$ C^{1+2n} \Gamma\left( \frac{5}{6} n + \frac{1}{2} \right) \subset\bigO\left(\Gamma\left(n - \frac{7}{12}\right)\right), $
so
\begin{align} \label{eq:een_sumreduced} \ee_n = \sum_{r=0}^{2n-2}\sum_{\ld,\md} \eta_{\ld \cup \md} \left([u^{2n-r} x^{\ld}] \TT\left(u^2 e^{-{\mathbf W}}(u \cdot \bar x)\right)\right) \left([u^{r} x^{\md}]{\mathbf H}(u,\bb x)\right)+ \bigO\left(\Gamma\left(n - \frac{7}{12}\right)\right). \end{align}
Using the notation and statement of Corollary~\ref{cor:Bns}, this becomes
\begin{align*} \ee_n = \sum_{s=0}^{2n} B_{n,s} + \bigO \left( \Gamma \left( n - \frac{7}{12} \right) \right) = B_{n,0} + \bigO \left( \Gamma \left( n - \frac{7}{12} \right) \right), \end{align*}
where $B_{n,0}$ is given by
the expression in Eq.~\eqref{eq:een_sumreduced} restricted
to summands where $r+|\ld|-2m_2 = r + \sum_{k\geq 3} k m_k(\ld) = 0$.
The sum over $r$ therefore trivializes and we only have to account for $r=0$ and the sum over deranged partitions reduces to a sum over partitions that only have parts of size 2. We may write this as,
\begin{gather*} \sum_{m_2 \geq 0} \sum_{\md} \eta_{[2^{m_2}] \cup \md} \left( [u^{2n} x_2^{m_2} x_3^0 x_4^0 \cdots] \TT\left(u^2 e^{-{\mathbf W}(u \cdot \bar x)}\right) \right) \left( [u^{0} x^{\md}] {\mathbf H}(u,\bb x) \right) \\
= \sum_{m_2 \geq 0} \eta_{[2^{m_2}]} [u^{2n} x_2^{m_2} x_3^0 x_4^0 \cdots] \TT\left(u^2 e^{-{\mathbf W}(u \cdot \bar x)}\right) , \end{gather*}
where we used the fact that $[u^0 x^\emptyset]{\mathbf H}(u,\bb x) =1$
and $[u^0 x^{\md}]{\mathbf H}(u,\bb x) =0$
if $\md \neq \emptyset$.
By Lemma~\ref{lmm:Tcoeffs},
\begin{align*} [u^{2n} x_2^{m_2} x_3^0 x_4^0 \cdots] \TT\left(u^2 e^{-{\mathbf W}(u \cdot \bar x)}\right) = \Ch_{\frac{2n-2m_2}{2}} \binom{(2n-2m_2)\frac{\mu(2)}{4}}{m_2}. \end{align*}
Using $\mu(2)=-1$ and $\eta_{[2^{m_2}]} = \eta_{2,m_2}$ gives the statement.
\end{proof}

\section{The odd forested graph complex}
\label{sec:odd}
As remarked in Section~\ref{sec:disc}, the Euler characteristic $e(\Out(F_n))$ is equal to the Euler characteristic of Kontsevich's Lie graph complex,   which is equal to the Euler characteristic of the forested graph complex.   In \cite{Ko93}  Kontsevich defined  an {\em odd version}%
\footnote{In \cite{Willwacher}, Willwacher used the opposite notions of ``even'' and ``odd'' orientation of a graph, so that Kontsevich's ``odd'' graph complexes are Willwacher's ``even'' graph complexes.}
of the Lie graph complex. The two graph complexes differ only by the definition of the {\em orientation} of a graph.

Recall that the (even) forested graph complex is generated by all even forested graphs (see Section~\ref{sec:disc}), which are forested graphs with no automorphisms $\alpha$ that induce an odd permutation $\alpha_\Phi: E_\Phi \to E_\Phi$ on the forest edges $E_\Phi$. In other words, all automorphisms $\alpha$ of an even forested graph  satisfy $ \sign (\alpha_\Phi) =1$.

Every automorphism $\alpha$ of a graph $\grph$ also  induces automorphisms on its zeroth and first homologies, i.e.~we have $\alpha_{H_0(\grph,\Z)} \colon H_0(\grph, \Z) \to H_0(\grph, \Z)$ and $\alpha_{H_1(\grph,\Z)}\colon H_1(\grph, \Z) \to H_1(\grph, \Z)$.
An \emph{odd} forested graph is a forested graph all of whose automorphisms $\alpha$ satisfy 
$$\sign(\alpha_\Phi) \det(\alpha_{H_0(\grph,\Z)}) \det (\alpha_{H_1(\grph,\Z)}) =1.$$
For connected graphs $H_0(\grph,\Z)$ is one-dimensional and has a canonical orientation.
Hence, $\det(\alpha_{H_0(\grph,\Z)}) =1$ for all connected $\grph$.
The complex spanned by such connected odd forested graphs computes $H^*(\Out(F_{n+1});\widetilde{\mathbb Q})$, where $\widetilde{\mathbb Q}$ is the representation obtained by composing the canonical group homomorphism $\Outn \to \GL_n (\Z)$ with the determinant map.  The techniques developed in this paper can be used almost verbatim to compute the associated Euler characteristic $\eO(\Outn) = \sum_k (-1)^k \dim(H^k(\Outn, \widetilde \Q))$, which  is equal to the Euler characteristic of Kontsevich's odd Lie graph complex.

As in Proposition~\ref{prop:eOutFnFGC} we define
\begin{align*} \sum_{\substack{ [G,\Phi] \, \mathrm{odd}\\
G \, \mathrm{connected} \\
\chi(G)=-n } } (-1)^{e(\Phi)} = \eO(\Out(F_{n+1})), \end{align*}
where we sum over all isomorphism classes of connected odd forested graphs $[G,\Phi]$ of Euler characteristic $\chi(G) = -n$. As we did in Section~\ref{sec:disc} we also define
\begin{align} \label{eq:defeeodd} \ee^{\mathrm{odd}}_n = \sum_{ \substack{ [G,\Phi] \, \mathrm{odd}\\
\chi(G)=-n } } (-1)^{e(\Phi)}, \end{align}
where we sum over all isomorphism classes of \emph{possibly disconnected} odd forested graphs of Euler characteristic $\chi(G) = -n$.

Consider a disconnected forested graph $[\grph,\Phi]$ which contains two copies of the  (possibly disconnected) forested graph $[g,\varphi]$ and
let $\alpha$ be the automorphism of $[\grph,\Phi]$ that switches the two copies.
As $\alpha$ permutes the canonical basis of $H_0(\grph,\Z)$ and the canonical 
decomposition of $H_1(\grph,\Z)$ over connected components, we have
$\sign(\alpha_\Phi) \det(\alpha_{H_0(\grph,\Z)}) \det (\alpha_{H_1(\grph,\Z)}) = (-1)^{\xi(g,\varphi)},$
where $\xi(g,\varphi) = e(\varphi)+\dim H_0(g,\Z) + \dim H_1(g,\Z)$.
Using this, we find a  version of Theorem~\ref{thm:eeOutFn} for the odd case:
\begin{theorem}
\label{thm:eeOutFn_odd}
\begin{align*} \sum_{n \geq 0} \ee^{\mathrm{odd}}_n \hbar^{n} = \prod_{n = 1}^\infty \left( \frac{1}{1-(-\hbar)^n} \right)^{(-1)^n \eO(\Out(F_{n+1}))} \end{align*}
\end{theorem}
\begin{proof}

We follow the argument of the proof for Theorem~\ref{thm:eeOutFn}. Just as in the even case, 
each odd forested graph $[G,\Phi]$ can be described by 
giving a set of distinct connected odd forested graphs $[g,\varphi]$ 
together with a multiplicity for each connected graph.
Here  $[g,\varphi]$ can appear 
at most with multiplicity one in $[G,\Phi]$ if $\xi(g,\varphi)$
is odd, and with arbitrary multiplicity if $\xi(g,\varphi)$ is even.
Analogously to Theorem~\ref{thm:eeOutFn}, we get
\begin{align*} \sum_{n \geq 0} \ee^{\mathrm{odd}}_n \hbar^{n} = \left( \prod_{ \substack{ [g,\varphi] \\
\xi(g,\varphi)\, \text{odd} } } \left( 1+ (-1)^{e(\varphi)} \hbar^{ -\chi(g)} \right) \right) \left( \prod_{ \substack{ [g,\varphi] \\
\xi(g,\varphi)\, \text{even} } } \frac{1}{ 1 -(-1)^{e(\varphi)}\hbar^{-\chi(g)}} \right), \end{align*}
where we multiply over all connected odd forested graphs $[g,\varphi]$ with $\xi(g,\varphi)$ odd or even. The statement follows 
by observing that $\xi(g,\varphi)$ has the same parity as $e(\varphi) +\chi(g) = e(\varphi) + \dim(H_0(g,\Z)) - \dim(H_1(g,\Z))$ and by following the rest of the argument for Theorem~\ref{thm:eeOutFn}.
\end{proof}
By adjusting signs in Corollary~\ref{cor:mM}, we can then use Theorem~\ref{thm:eeOutFn_odd} to compute $\eO(\Out(F_{n+1}))$ from 
$ \ee^{\mathrm{odd}}_n .$

Recall that for a forested graph $(G,\Phi)$ and an automorphism
$\alpha \in \Aut(G,\Phi)$, the number $e_\alpha(\Phi)$ denotes the number of cycles of the permutation on the forest edges, $\alpha_\Phi: E_\Phi \rightarrow E_\Phi$.
\begin{lemma}
\label{lmm:detH1}
For a forested graph $(G,\Phi)$, $\sum_{\alpha \in \Aut(G,\Phi)} \det (\alpha_{H_0(\grph,\Z)}) \det (\alpha_{H_1(\grph,\Z)}) (-1)^{e_\alpha(\Phi)}$ is equal to
$(-1)^{e(\Phi)} |\Aut(G,\Phi)|$ if $(G,\Phi)$ is odd or $0$ otherwise.
\end{lemma}
\begin{proof}
The argument for Lemma~\ref{lmm:AutSum} applies.
\end{proof}

Each automorphism $\alpha$ of a forested graph $(G,\Phi)$ also provides us with a permutation $\alpha_{E_G \setminus \Phi} : {E_G \setminus E_\Phi} \rightarrow {E_G \setminus E_\Phi}$ that permutes the non-forest edges,
a set of permutations $(\alpha_e)_{e\in {E_G \setminus E_\Phi}}$ of order 1 or 2 that might change the orientation of each non-forest edge,
and a permutation $\alpha_{H_0(\Phi)}$ that permutes the connected components of the forest.

\begin{lemma}
\label{lmm:H1toEGP}
For a forested graph $(G,\Phi)$ and an automorphism
$\alpha : (G,\Phi) \to (G,\Phi)$,
$$
\det(\alpha_{H_0(\grph,\Z)})
\det(\alpha_{H_1(\grph,\Z)}) =
\sign(\alpha_{H_0(\Phi)})
\sign(\alpha_{E_G \setminus E_\Phi})
\prod_{e \in E_G \setminus E_\Phi} \sign(\alpha_e).
$$
\end{lemma}
\begin{proof}
Let $\Z E$ and $\Z V$ be the $\Z$-vector spaces generated by the edge and vertex set of $(G,\Phi)$. We have the exact sequence
$$
0 \to H_1(G,\Z) \to \Z E \to \Z V \to H_0(G,\Z) \to 0.
$$
The spaces $\Z E$ and $\Z V$ come with a natural bilinear form, hence we can dualize the usual boundary operator $\partial_1: \Z E \to \Z V$ to
$\partial_1^*: \Z V \to \Z E$. We get the isomorphism $\Z V \oplus H_1(G,\Z) \to \Z E \oplus H_0(G, \Z)$ given by $(v,c) \mapsto (c+\partial_1^* v, \partial_0 v)$.
An automorphism $\alpha$ of $G$ also gives the automorphisms 
$\alpha_{\Z V},\alpha_{\Z E}, \alpha_{\Z V \oplus H_1(G,\Z)}$ and $ \alpha_{\Z E \oplus H_0(G,\Z)}$
of the respective vector spaces.
Because $\alpha$ acts block-wise on the summands in the direct sums, their determinants factor as follows,
$$
\det (\alpha_{\Z V}) \det(\alpha_{H_1(G,\Z)}) =
\det (\alpha_{\Z V \oplus H_1(G,\Z)})
=\det (\alpha_{\Z E \oplus H_0(G,\Z)})
 = \det (\alpha_{\Z E}) \det(\alpha_{H_0(G,\Z)}),
$$
and as
$\det X = \pm 1$ for all $X\in \GL_n(\Z)$, we get
$\det(\alpha_{H_0(G,\Z)}) \det(\alpha_{H_1(G,\Z)}) = \det (\alpha_{\Z E}) \det (\alpha_{\Z V})$.

The vector space $\Z E$ can be decomposed $\Z E = \Z E_{\Phi} \oplus \Z (E_G \setminus E_\Phi)$ and $\alpha_{\Z E}$ acts block-wise on both summands as it does not mix forest and non-forested edges. It follows that
$\det (\alpha_{\Z E}) = \det (\alpha_{\Z E_\Phi}) \det (\alpha_{\Z ( E_G \setminus E_\Phi)})$.
From $\Phi$, we get the short exact sequence,
$0 \to \Z E_\Phi \to \Z V \to H_0(\Phi,\Z) \to 0$, where we used the fact that a forest has no first homology. Consequently, $\det(\alpha_{\Z V}) = \det(\alpha_{\Z E_\Phi}) \det(\alpha_{H_0(\Phi,\Z)})$ and
$$\det(\alpha_{H_0(G,\Z)})  \det(\alpha_{H_1(G,\Z)}) = 
\det(\alpha_{\Z E_\Phi})^2
\det (\alpha_{\Z ( E_G \setminus E_\Phi)})
 \det(\alpha_{H_0(\Phi,\Z)})
.$$
Ordering and directing the edges $E_G \setminus E_\Phi$ gives a basis of $\Z (E_G \setminus E_\Phi)$. The orientation can be changed by switching two edges or reversing the direction of one edge. Hence, $\det(\alpha_{\Z (E_G \setminus E_\Phi)}) = \sign(\alpha_{E_G \setminus E_\Phi}) \prod_{e \in E_G \setminus E_\Phi} \sign(\alpha_e).$
Fixing an ordering of the connected  components of $\Phi$ gives
a basis of $H_0(\Phi,\Z)$. Therefore, $\det(\alpha_{H_0(\Phi,\Z)})$ is equal to the sign of the permutation $\alpha_{H_0(\Phi)}$ that $\alpha$ induces on the components of $\Phi$.
\end{proof}

Combining Lemmas~\ref{lmm:detH1} and \ref{lmm:H1toEGP} with Eq.~\eqref{eq:defeeodd} results in
\begin{theorem}
\label{thm:tau_odd}
\begin{align*} \ee^{\mathrm{odd}}_n = \sum_{ \substack{ [G,\Phi] \\
\chi(G) = -n } } \frac{1}{|\Aut(G,\Phi)|} \sum_{\alpha \in \Aut(G,\Phi)} \sign(\alpha_{H_0(\Phi)}) \sign(\alpha_{E_G \setminus E_\Phi}) \left( \prod_{e \in E_G \setminus E_\Phi} \sign(\alpha_e) \right) (-1)^{e_\alpha(\Phi)}. \end{align*}
Here, we sum over all forested graphs $[G,\Phi]$ of Euler characteristic $\chi(G) = -n$.
\end{theorem}
This statement is the odd version of Theorem~\ref{thm:tau}.
As in Section~\ref{sec:effective} we can produce a formula for $\ee^{\mathrm{odd}}$ by counting forests and matchings separately before combining both expressions to give a counting formula for forested graphs. In the odd case, we have to change a couple of signs in the derivation to accommodate the additional $\sign$ factors in the statement above.

We define an odd version of our forest generating function $\mathbf F$ in Eq.~\eqref{eq:defF} using the same notation and by weighting each odd permutation of the connected components of the forest with a sign:
\begin{align*} \begin{aligned}  {\mathbf F^\mathrm{odd}}(u, \bb x) &= \sum_{s\geq 0} \frac{1}{s!}\sum_{(\Phi,\gamma)\in\mathcal{AF}[s]} \sign(\gamma_{H_0(\Phi)}) (-1)^{e_{\gamma}(\Phi)}x^\gamma u^{s-2k(\Phi)} , \end{aligned} \end{align*}
where $\sign(\gamma_{H_0(\Phi)})$ is the sign of the permutation induced by $\gamma$ on the connected components of $\Phi$. This factor will account for the $\sign(\alpha_{H_0(\Phi)})$ term in Theorem~\ref{thm:tau_odd}.
Using the same argument as for Proposition~\ref{prop:tree_expression}, but accounting for a sign flip for each even cycle on the set of components or equivalently, by setting $y_k = (-1)^{k+1}$ in Eq.~\eqref{eq:Vset}, results in
\begin{proposition}
\label{prop:tree_expression_odd}
\begin{align*} {\mathbf F^{\mathrm{odd}}}(u,\bb x) =\exp \left(\sum_{k \geq 1} (-1)^{k+1} u^{-2k} \frac{ {\mathbf V}( (u \cdot \bb x)_{[k]})}{k}\right). \end{align*}
\end{proposition}

To account for the factor
$\sign(\alpha_{E_G \setminus E_\Phi}) \left( \prod_{e \in E_G \setminus E_\Phi} \sign(\alpha_e) \right) $ in Theorem~\ref{thm:tau_odd}, we have to use a signed version of the matchings that we introduced in Section~\ref{sec:matchings}.
A permutation $\alpha \in \SG_{2n}$ and a fixed-point free
involution on $\{1,\ldots,2n\}$ such that $\alpha \circ \iota = \iota \circ \alpha$ give rise to a permutation $\alpha_\iota$ of the orbits of $\iota$ and a set of $n$ permutations $\alpha_{e_1}, \ldots, \alpha_{e_n}$ that permute the elements of each individual orbit of $\iota$. We define
\begin{align*} \eta^{\mathrm{odd}}_\alpha = \sum_{\iota \circ \alpha = \alpha \circ \iota} \sign(\alpha_\iota) \prod_{k=1}^n \sign(\alpha_{e_k}), \end{align*}
where we sum over all such pairs $\iota$ and $\alpha$.

Using these definitions with the argument for Proposition~\ref{prop:etaF} and Theorem~\ref{thm:tau_odd}, we get
\begin{proposition}
\begin{align*} \sum_\lambda \eta^\mathrm{odd}_\lambda[u^{2n} x^\lambda] {\mathbf F^\mathrm{odd}}(u,\bb x)=\ee^\mathrm{odd}_n.    \end{align*}
\end{proposition}

Following the argument in Section~\ref{sec:matchings} for the derivation of a formula for $\eta_\lambda$, the cycle index series for the signed version of a matching of two points is given by
$$
{\mathbf E^\mathrm{odd}}(\bb x) = \frac{1}{2!} \sum_{(\phi,\alpha)\in \mathcal{AE}[2]} \sign(\alpha)
x^\alpha = \frac{1}{2} \left( x_1^2 - x_2 \right).
$$
As in Section~\ref{sec:matchings}, we can use Proposition~\ref{prop:expformaut} to get a generating function for the numbers $\eta^\mathrm{odd}_\alpha$:
\begin{lemma}
\begin{align*} \sum_{n \geq 0} \frac{1}{(2n)!} \sum_{\alpha \in \SG_{2n}} \eta^\mathrm{odd}_\alpha x^\alpha = \exp \left( \sum_{k \geq 1} \frac{(-1)^{k+1}}{2k} \left( x_k^2 - x_{2k} \right) \right). \end{align*}
\end{lemma}
\begin{proof}
Set $y_k = (-1)^{k+1}$ after applying Proposition~\ref{prop:expformaut} as each even cycle of the permutation induced on the components is counted with a minus sign this way and the sign of a permutation is equal to $(-1)^{\text{\# of even cycles}}$.
\end{proof}

Repeating the computation for Corollary~\ref{cor:matchings} while accounting for the changed signs we find,
\begin{corollary}
\label{cor:matchings_odd}
If $\lambda=[1^{m_1}2^{m_2}\ldots n^{m_n}]$, then
$$
 \eta^\mathrm{odd}_{\lambda}  = \prod_{k = 1}^n \eta^\mathrm{odd}_{k, m_k},
 $$
 where
$$
\eta^\mathrm{odd}_{k,2s}=
\begin{cases}
k^{s} (2s -1)!!  & \text{if $k$ is odd}\\
\sum_{r=0}^{s} (-1)^{r} \binom{2s}{2r} k^r (2r-1)!! & \text{if $k$ is even}
\end{cases}
$$
and
 $$
\eta^\mathrm{odd}_{k,2s+1}=
\begin{cases}
0 & \text{if $k$ is odd}\\
(-1)^{k/2} \sum_{r=0}^{s} (-1)^{r} \binom{2s+1}{2r} k^r (2r-1)!! & \text{if $k$ is even}
\end{cases}.
$$
\end{corollary}
Via exactly the same procedure described in Theorem~\ref{thm:compute},  but substituting the odd versions of the respective series and adjusting signs in Corollary~\ref{cor:mM} analogously to the change from Theorem~\ref{thm:eeOutFn} to Theorem~\ref{thm:eeOutFn_odd}, we get an effective algorithm for computing the numbers $\eO(\Outn)$. The first few values are listed in Appendix~\ref{sec:table}.

The analytic argument in Section~\ref{sec:asymptotics} also works in the odd case, as the relevant coefficients of $\mathbf{F}$ and $\mathbf{F}^\mathrm{odd}$ agree, the values of $\eta_\alpha$ and $\eta^{\mathrm{odd}}_\alpha$ are equal for the trivial permutation and $|\eta^{\mathrm{odd}}_\alpha| \leq \eta_\alpha$ for all $\alpha$.
The modified signs  have a nontrivial consequence only in Lemma~\ref{lmm:e14}. Instead of the statement of Lemma~\ref{lmm:e14} we find that in the odd case
\begin{align*} \sum_{m =0}^n\frac{(-1)^m }{ 2^m m!}\eta^{\mathrm{odd}}_{2,m}=e^{\frac14}+\bigO\left(\frac{1}{L_n}\right), \end{align*}
after repeating the computation using the numbers from Corollary~\ref{cor:matchings_odd}. The remaining proof is completely equivalent up to the substitution of the relevant number $e^{-\frac14} \to e^{\frac14}$. Following through the argument again results in the odd version of Theorem~\ref{thm:eOutFnAsy}:
\begin{theorem}
\label{thm:eOutFnAsy_odd}
The %
Euler characteristic $\eO(\Outn)$ has the leading asymptotic behavior
\begin{gather*} \eO(\Out(F_{n})) \sim - {e^{\frac14}} \left(\frac{n}{ e}\right)^{n} \frac{1}{(n\log n)^2} \text{ as } n\rightarrow \infty. \end{gather*}
\end{theorem}
\begin{proof}[Proof of Theorem~\ref{thm:easyodd}]
Use Theorem~\ref{thm:eOutFnAsy_odd}, \cite[Thm.~A]{BV} and Stirling's formula (Lemma~\ref{lmm:stirling}).
\end{proof}

\appendix

\section{Table of \texorpdfstring{$\chi(\Outn)$}{chi(OutFn)}, \texorpdfstring{$e(\Outn)$}{eOutFn} and \texorpdfstring{$\eO(\Outn)$}{eOddOutFn} for \texorpdfstring{$n\leq 15$}{n<=15}}
\label{sec:table}
\begin{center}
{\renewcommand{\arraystretch}{1.2}
\begin{tabular}{||r|r|r|r||}
\hline
$n$ & $\chi(\Out(F_n))$ & $e(\Out(F_n))$ & $\eO(\Out(F_n))$\\ \hline \hline
$2$ & $-\frac{1}{24} \approx -0.042$ & $1$ & $0$
\\ \hline
$3$ & $-\frac{1}{48} \approx -0.021$ & $1$ & $0$
\\ \hline
$4$ & $-\frac{161}{5760} \approx -0.028$ & $2$ & $-1$
\\ \hline
$5$ & $-\frac{367}{5760} \approx -0.064$ & $1$ & $0$
\\ \hline
$6$ & $-\frac{120257}{580608} \approx -0.21$ & $2$ & $-1$
\\ \hline
$7$ & $-\frac{39793}{45360} \approx -0.88$ & $1$ & $-2$
\\ \hline
$8$ & $-\frac{6389072441}{1393459200} \approx -4.6$ & $1$ & $-8$
\\ \hline
$9$ & $-\frac{993607187}{34836480} \approx -29$ & $-21$ & $-38$
\\ \hline
$10$ & $-\frac{5048071877071}{24524881920} \approx -206$ & $-124$ & $-275$
\\ \hline
$11$ & $-\frac{9718190078959}{5748019200} \approx -1691$ & $-1202$ & $-2225$
\\ \hline
$12$ & $-\frac{375393773534736899347}{24103053950976000} \approx -15575$ & $-10738$ & $-20358$
\\ \hline
$13$ & $-\frac{2495397080915203519}{15692092416000} \approx -159023$ & $-112901$ & $-207321$
\\ \hline
$14$ & $-\frac{1031156416543036906701911}{578473294823424000} \approx -1782548$ & $-1271148$ & $-2320136$
\\ \hline
$15$ & $-\frac{6147011108414481406421}{282457663488000} \approx -21762593$ & $-15668391$ & $-28287416$
\\ \hline
\end{tabular}
}
\vspace{.5cm}
\end{center}
The numbers $\chi(\Outn)$ were computed using
\cite[Proposition 8.5]{BV}, the numbers $e(\Outn)$ were
 computed as described in Theorem~\ref{thm:compute} and
the numbers $\eO(\Outn)$ similarly as discussed in Section~\ref{sec:odd}.
See \cite{BVer} for some programming details of these computations.

\end{document}